\documentclass[11pt]{article}

\usepackage{amssymb,latexsym}
\usepackage{amsmath,amscd}
\usepackage{theorem}
\usepackage[all]{xy}
\usepackage{url}

\title{\bf Principal actions of stacky Lie groupoids}

\author{Henrique Bursztyn\thanks{{\tt henrique@impa.br}}\\[0.1cm]
        Instituto Nacional de Matem\'atica Pura e Aplicada\\
        Estrada Dona Castorina 110 \\
        22460-320, Rio de Janeiro,  Brazil.
        \\[0.2cm]
        Francesco Noseda  \thanks{{\tt noseda@im.ufrj.br}}
        \\[0.1cm]
        Instituto de Matem\'atica - Universidade Federal do Rio de Janeiro\\
        Avenida Athos da Silveira Ramos 149 (Centro de Tecnologia - Bloco C)  \\
        21941-909, Rio de Janeiro,  Brazil.
        \\[0.2cm]
        Chenchang Zhu  \thanks{{\tt zhu@uni-math.gwdg.de}}
        \\[0.1cm]
        Mathematisches Institut -
        Universit\"at G\"ottingen \\
        Bunsenstrasse 3-5 \\
         D-37073, G\"ottingen, Germany.}
\date{}

\newcommand{\cl}[1]{\mathcal{#1}}
\newcommand{\mr}[1]{\mathrm{#1}}

\newcommand{\G}{\mathcal{G}}
\newcommand{\X}{\mathcal{X}}

\newcommand{\cX}{\mathcal{X}}

\newcommand{\CFG}{\mathrm{CFG}}
\newcommand{\Obj}{\mathrm{Obj}}

\newcommand{\ra}{\rightarrow}

\newcommand{\map}{\longrightarrow}
\newcommand{\rra}{\rightrightarrows}
\newcommand{\lmap}[1]{\stackrel{#1}{\longrightarrow}} % Labeled map
\newcommand{\then}{\Longrightarrow}

\newcommand{\sour}        {\mathsf{s}}
\newcommand{\tar}         {{\mathsf{t}}}
\newcommand{\inv}         {\mathsf{i}}
\newcommand{\mult}        {\mathsf{m}}
\newcommand{\un}          {\mathsf{1}}

\newcommand{\vep}{\varepsilon}

\newcommand{\bs}{\backslash}
\newcommand{\slbtimes}[3]{\,{}_{#1}\!\underset{#2}{\times}\!{}_{#3}\,} % super labeled times
\newcommand{\pq}{/_{\!p}\,} % prequotient

\newcommand{\scirc}{{\scriptstyle \circ}}
\newcommand{\sscirc}{{\scriptscriptstyle \circ}}
\newcommand{\hc}{\,{\scriptstyle \circ}\,} %horizontal composition

\newcommand{\act}{\mathsf{A}}
\newcommand{\ma}{\mathsf{a}}

\newcommand{\p}{\mathsf{p}}

\newcommand{\Proj}{\mathsf{r}}
\newcommand{\Q} {\widehat{\Delta}}
\newcommand{\q} {\mathsf{q}}
\newcommand{\stack} {\mathsf{b}}

\newtheorem{lemma} {Lemma} [section]
\newtheorem{proposition} [lemma] {Proposition}

\newtheorem{theorem} [lemma] {Theorem}
\newtheorem{corollary} [lemma] {Corollary}
\newtheorem{definition}[lemma] {Definition}

\theorembodyfont{\rm} %%% This is to typeset examples
\newtheorem{example}[lemma] {Example}

\newtheorem{remark}[lemma]{Remark}

\newenvironment{proof}{{\sc Proof:}}{{\hspace*{\fill} $\square$\\}}

\numberwithin{equation}{section}

\begin{document}

\maketitle

\vspace{-6mm}

\begin{center}
%\date{October 30, 2015}
\texttt{\today}
\end{center}

\vspace{0mm}

\begin{abstract}
Stacky Lie groupoids are generalizations of Lie groupoids in which
the ``space of arrows'' of the groupoid is a differentiable stack.
In this paper, we consider actions of stacky Lie groupoids on
differentiable stacks and their associated quotients. We provide a
characteri\-zation of {\em principal actions} of stacky Lie groupoids, i.e.,
actions whose quotients are again differentiable
stacks in such a way that the projection onto the quotient is a
principal bundle. As an application, we extend the notion of Morita equivalence of Lie groupoids to the
realm of stacky Lie groupoids, providing examples that naturally
arise from non-integrable Lie algebroids.
\end{abstract}

%MSC2000 Subject Classification Numbers:
%Keywords:

%%%%%%%%%%%%%%%%%%%%%%%%%%%%%%%%%%%%%%%%%%%%%%%%%%%%%%%%%%%%%%%%%%%%%%%%%%%%
%%%%%%%%%%%%%%%%%%%%%%%%%%%%%%%%%%%%%%%%%%%%%%%%%%%%%%%%%%%%%%%%%%%%%%%%%%%%
\section{Introduction}
%%%%%%%%%%%%%%%%%%%%%%%%%%%%%%%%%%%%%%%%%%%%%%%%%%%%%%%%%%%%%%%%%%%%%%%%%%%%
%%%%%%%%%%%%%%%%%%%%%%%%%%%%%%%%%%%%%%%%%%%%%%%%%%%%%%%%%%%%%%%%%%%%%%%%%%%%

Lie groupoids have widespread use in several areas of mathematics,
and in recent years some of their ``higher'' versions have drawn
much attention, particularly in the study of higher categorical
structures, higher stacks and higher gauge theory, see e.g.
\cite{bala:2gp,basch,bartels,breen,brown,duskin,Moer,nss:1}. This
paper concerns \textit{stacky Lie groupoids}, which are
generalizations of Lie groupoids $\G\rightrightarrows M$ where $\G$
is allowed to be a differentiable stack, while $M$ is a smooth
manifold; more precisely, stacky Lie groupoids may be viewed as Lie 2-groupoids, i.e.,
2-truncations of Kan simplicial manifolds \cite{z:tgpd}. These objects were introduced in \cite{tz} in connection
with the problem of integrating Lie algebroids.

Recall that, while any finite-dimensional Lie algebra can be
integrated to a Lie group, not every Lie algebroid integrates to a
Lie groupoid, see \cite{cf} and references therein. Whenever a Lie algebroid is integrable, the
so-called ``path-space construction''
\cite{cattaneofelder,cf,severa} provides a concrete way to obtain an integrating
Lie groupoid. However, for a non-integrable Lie algebroid this construction only leads to a {\it topological}
groupoid. The starting point in \cite{tz}, see also \cite{WeInt}, is the observation that the topological groupoids arising in this way
naturally carry the structure of {\em
differentiable stacks}, and that this additional information allows
one to recover the underlying Lie-algebroid data; as a consequence,
\cite{tz} establishes a correspondence between (arbitrary) Lie
algebroids and \'etale stacky Lie groupoids, placing the usual
correspondence between integrable Lie algebroids and Lie groupoids
in a broader framework.

The main purpose of this paper is to study actions of stacky Lie
groupoids on differentiable stacks, with focus on the notion of {\em principality}, inspired by the usual study of  actions of Lie groupoids on manifolds. To put our results in context, recall that when a Lie groupoid $G$ acts on a manifold $X$, the orbit space $X/G$ generally fails to be a smooth manifold. An action is {\em principal}
if $X/G$ is again a manifold and, additionally, the quotient map $X\to X/G$ makes $X$ into a principal $G$-bundle with base $X/G$. A classical result in geometry asserts that an action is principal if and only if it is free and proper, and this completely describes $G$-actions that correspond to
principal bundles in the (finite-dimensional) smooth setting. Obtaining similar characterizations of principal actions (only in terms of the action itself, not involving the quotient) in other categories is often challenging,
see e.g. \cite[Sec.~9]{MZ} for a discussion of this problem in different contexts. This paper concerns this issue in the realm of differentiable stacks.

More specifically, given an action of a stacky Lie groupoid $\G$ on a differentiable stack $\X$, we first address the construction of the ``orbit space'' $\X/\G$, building on \cite{breen}. Our main result  (Theorem~\ref{TheDiagonalThenPrincipal}) then provides a complete
characterization of principal actions, i.e., those for which this quotient inherits the structure of
a {\em differentiable} stack in such a way that the natural projection $\X
\to \X/\G$ defines a principal $\G$-bundle:
\\

\noindent
\textbf{Theorem.} \textit{An action of a stacky Lie groupoid $\cl{G}$
on a differentiable stack $\cl{X}$ is principal if and only
if the action-projection map $\cl{X}\times_M \cl{G}\ra \cl{X}\times \cl{X}$
is {\em weakly representable}}.\\

\noindent
The {\em action-projection map} above is defined by two maps $\cl{X}\times_M \cl{G}\ra \cl{X}$, the first being
the natural projection while the second is the action map, see \eqref{eq:Delta}. The notion of {\em weak representability} of a morphism between differentiable stacks is introduced in Def.~\ref{def:weakep}.
The point to be remarked here is that the sole condition of weak representability of the action-projection map simultaneously encodes a freeness-type property of the action (that we call {\em 1-freeness}, see Def.~\ref{def1free}) and guarantees the differentiability of the quotient stack. We verify that this condition
automatically holds whenever $\G$ is an ordinary Lie groupoid (see Corollary~\ref{cor:ginot1}); in particular, this recovers the well-known fact that {\em any} smooth action of a Lie groupoid $G$ on a manifold $X$ is principal
(in the stacky sense) with base the quotient stack $[X/G]$. The proof of the previous theorem is presented in Section~\ref{sec:principal}. Just as in the classical smooth setting, our characterization of principal actions of stacky Lie groupoids provides a complementary approach to other existing viewpoints to higher principal bundles (defined e.g. through gluing local trivial data, or through classifying stacks),  as found e.g. in \cite{basch,bartels,nss:1,wockel}.

%
%\comment{+ The proof of the last theorem, which relies on ****, is presented in sect 5.
%The proof of this result turns out to be quite interesting and rather non-trivial. The traditional tools in differential geometry dealing with free and proper actions are not helpful here. One needs to invent new methods, combining the techniques from stack theory (such as how to define quotient stacks) and the geometrical intuition on free and proper action (which leads our definition of 1-free action in Def. 4.7).

%\comment{
%Moreover, it then turns out that our principal actions coincide with other existing versions of higher principal bundles, defined through gluing local trivial data, or through a differential topological point of view using classifying stacks [3, 4, +36, 51].
%+
%}

In the classical theory of Lie groupoids and its various applications, a central role is played by
the notion of {\em Morita equivalence} (see e.g. \cite{bx,MM,MM2}). One of the applications of our results is an extension of this notion to the realm of stacky Lie groupoids.
Indeed, a common approach to Morita equivalence of Lie groupoids uses principal bundles to define ``generalized morphisms'', in such a way that Morita equivalence is expressed by {\em biprincipal bibundles} (see e.g. \cite[Sec.~2]{MM2}). Using (bi)principal (bi)bundles of stacky Lie groupoids, we define Morita equivalence analogously, and the previous theorem is key in showing that
stacky biprincipal bibundles can be ``composed'', which leads
to our second main result:\\

\noindent
\textbf{Theorem.} \textit{Morita equivalence of stacky Lie groupoids is
an equivalence relation, which (faithfully) extends the
Morita equivalence relation for ordinary Lie groupoids.}\\

\noindent
See Theorem \ref{CorMorEqRel} and Proposition \ref{prop:equivusual} for details.
We remark that  another approach to categorified bibundles and Morita equivalence of 2-groupoids through simplicial methods is developed in \cite[Sec.~6]{duli}; for strict 2-groupoids, another viewpoint to Morita equivalence can be found in \cite{ginotstienon}.
%\comment{--------------------------------------------------------------- skip...
%** comment about the next main result: I just add one more new result of me and christian there, which shows that it can go high, thus it must be the correct notion!!
%------
%...developed in [29, Sec. 6];
%+
%this approach is then generalized to all higher groupoids in [Christian Blohmann and Chenchang Zhu, Morita equivalence for higher groupoids];
%+}

As a concrete illustration of the last result, in Section~\ref{subsec:ex} we consider transitive Lie algebroids $A_\omega \to M$ given as rank-one extensions of $TM$ corresponding to closed 2-forms $\omega\in \Omega^2(M)$, as in \cite[Ex.~3.7]{cf}.
The integrability of this class of Lie algebroids is governed by the group of periods  of $\omega$, $\Gamma\subset \mathbb{R}$. When $\Gamma$ is discrete, so that $\mathbb{R}/\Gamma \cong S^1$,
$A_\omega$ is integrable: in this case, the class of $\omega$ determines a principal $(\mathbb{R}/\Gamma)$-bundle over $M$,
and $A_\omega$ is identified with its Atiyah Lie algebroid. Its canonical Lie groupoid integration is Morita equivalent to the abelian group $\mathbb{R}/\Gamma$. Our theory allows us to extend this picture to the non-integrable case, i.e., when $\Gamma \subset \mathbb{R}$ is dense: in this case, we verify that the stacky Lie groupoid associated
with $A_\omega$ is Morita equivalent to the 2-group $[\mathbb{R}/\Gamma]$.

Much of our motivation for this work comes from Poisson geometry,
where the existing notion of Morita equivalence \cite{x2,x1} only
applies to {\em integrable} Poisson manifolds, i.e., Poisson
manifolds whose underlying Lie algebroids are integrable. A possible
approach to Morita equivalence of non-integrable Poisson manifolds
via stacks was suggested in \cite[Section~9.3]{bw2} (see also
\cite{bw3}), and this paper may be regarded as the first
foundational step in this direction. Indeed,  in light of the
correspondence between Lie algebroids and stacky Lie groupoids in
\cite{tz}, one should expect to have a description of Morita equivalence
of stacky Lie groupoids only in terms of Lie-algebroid data; this
could then be specialized to Poisson manifolds
(see \cite{tz2}). We plan to address these issues in subsequent
work.

For the reader's convenience, we outline the structure and content of the paper:

\begin{itemize}
\item Section \ref{sec:prelim} recalls the main definitions concerning
differentiable stacks and Lie groupoids and collects some technical
results used in the sequel. An important concept introduced in this
section is that of {\em weak representability} of a morphism of
differentiable stacks (Def.~\ref{def:weakep}), which plays a key role
in the study of principal actions.

\item Section \ref{sec:Grpd} recalls the notion of stacky Lie
groupoid and introduces actions of stacky Lie groupoids on
(differentiable) stacks, pointing out their key features. The notion
of {\em principal bundle} for stacky Lie groupoids is also
discussed in this section (Def.~\ref{DefPrincBundle}), along with
some of its basic properties and examples.

\item In Section \ref{sec:Quotients}, we define the ``quotient
space'' associated with an action of a stacky Lie groupoid
$\cl{G}\rra M$ on a differentiable stack $\cl{X}$. This quotient is
initially defined as a category fibred in groupoids, referred to as
the ``prequotient'' (see Prop.~\ref{PropPreqCfg}); its
stackification, denoted by $\cl{X}/\cl{G}$, is our object of
interest. The main properties of (pre)quotients are presented in
this section.

\item Section \ref{sec:principal} contains the main result of
the paper: Theorem~\ref{TheDiagonalThenPrincipal}, which provides a
charac\-teri\-zation of principal actions of stacky Lie groupoids, i.e.,
it gives a necessary and sufficient condition ensuring that an
action of a stacky Lie groupoid $\G$ on a differentiable stack $\X$
gives rise to a principal bundle $\X\to \X/\G$.
As an application, we show that the usual ``composition'' (or
``tensor product'') of principal bundles of Lie groupoids naturally
extends to stacky Lie groupoids (Prop.~\ref{prop:tensor}).

\item In Section~\ref{sec:morita}, following the usual theory of Lie groupoids, we consider
biprincipal bibundles, i.e., differentiable stacks carrying
commuting principal actions of two stacky Lie groupoids, one on the
right and the other on the left. These are the central objects
for the definition of {\em Morita equivalence} of stacky Lie
groupoids. We present a concrete example of Morita equivalence
arising from a non-integrable transitive Lie algebroid that generalizes
the usual Atiyah algebroid associated with a principal $S^1$-bundle.
We verify two key properties of our extended notion of Morita equivalence: that it is an equivalence relation (Thm.~\ref{CorMorEqRel}), and that it recovers  the classical one when
restricted to Lie groupoids (Prop.~\ref{prop:equivusual}).

\item The Appendices, organized in four sections, collect some
technical material, including proofs of auxiliary results needed
throughout the paper.
\end{itemize}

\smallskip

\noindent {\bf Acknowledgements.} We have benefited from discussions
with many colleagues, including A. Cabrera, M. Crainic, R.
Fernandes, M. Gualtieri, G. Ginot, E. Lerman, D. Li, R. Meyer, I. Moerdijk and
A. Weinstein. We are particularly indebted to Davide Stefani and
Matias del Hoyo for many helpful comments. We are
thankful to several institutions for hosting us during various
stages of this project, including the University of Toronto, IST,
IMPA, and G\"ottingen University. H.B. acknowledges the financial
support of Faperj and CNPq. C. Z. has been supported by the German
Research Foundation (Deutsche  Forschungsgemeinschaft (DFG)) through
the Institutional Strategy of the University of G\"ottingen.

\tableofcontents

%%%%%%%%%%%%%%%%%%%%%%%%%%%%%%%%%%%%%%%%%%%%%%%%%%%%%%%%%%%%%%%%%%%%%%%%%%%%%%%%%
\section{Preliminaries}\label{sec:prelim}

In this section, we collect basic facts about stacks used in the
sequel of the paper. Stacks have been extensively studied in
algebraic geometry, see e.g. \cite{SGA4,dm,lmb,v1,v2}; more
recently, there has been an increasing interest in stacks in the
categories of topological spaces and smooth manifolds, see e.g.
\cite{bx,metzler,pronk}. This paper focuses on stacks in the
differentiable category, in the spirit of \cite{bx,metzler}, where
details and proofs omitted here can be found; see also
Remark~\ref{rem:higher}.

Before moving on, we set up some notation. Given a category $\cl{X}$
and an object $x$ of $\cl{X}$, we will use either the notation
$x\in\mr{Obj}(\cl{X})$, or simply $x\in\cl{X}$. We denote the set of
morphisms from $x$ to $y$ by $\mathrm{Hom}_{\cl{X}}(x,y)$. For
2-categories we will often write compositions of 2-morphisms as
follows: \textit{horizontal compositions} are denoted by `$\circ$'
or juxtaposition (the same notation will be used for 1-morphisms),
and \textit{vertical compositions} by `$*$'. For example, if $A$,
$B$, $C$ are objects in a 2-category, and $a$, $b$, $c$, $d$ are
1-morphisms, and $\alpha$ and $\beta$ are 2-morphisms, we write
$b\scirc a=ba$, $\beta\scirc\alpha=\beta\alpha$ and
$\beta\ast\alpha$ for the compositions depicted in the diagrams
below:
$$
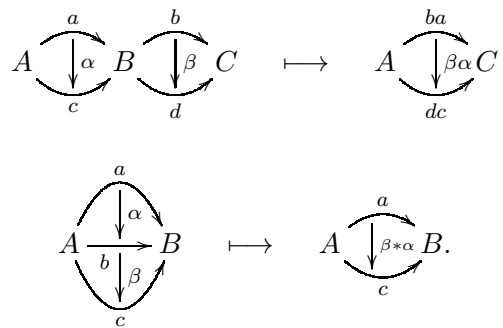
\xymatrix{
A\ar@/^1pc/[r]^{a}_{}="1"\ar@/_1pc/[r]_{c}^{}="2"&B
\ar@/^1pc/[r]^{b}_{}="3"\ar@/_1pc/[r]_{d}^{}="4"&C
\ar"1";"2"^{\alpha}\ar"3";"4"^{\beta} }\quad \longmapsto \quad
\xymatrix{ A\ar@/^1pc/[r]^{ba}_{}="5"\ar@/_1pc/[r]_{dc}^{}="6"&C
\ar"5";"6"^{\beta\alpha} }
$$
$$\xymatrix{
A\ar@/^2pc/[r]^{a}_{}="1"\ar[r]_(0.35){b}^{}="2"\ar@/_2pc/[r]_{c}^{}="3"&B
\ar"1";"2"^{\alpha}\ar"2";"3"^{\beta}} \quad \longmapsto \quad
\xymatrix{
A\ar@/^1pc/[r]^{a}_(0.4){}="1"\ar@/_1pc/[r]_{c}^(0.4){}="2"&B.
\ar"1";"2"^{{\scriptscriptstyle \beta\ast\alpha}} }
$$

For identity morphisms and 2-morphisms we will use the notation
`$\mr{id}$' as follows:
$$ A\lmap{ \mr{id}_A} A \qquad\qquad
\xymatrix{
A\ar@/^1pc/[r]^{a}_(0.4){}="1"\ar@/_1pc/[r]_{a}^(0.4){}="2"&B.
\ar"1";"2"^{{\scriptscriptstyle \mr{id}_a}} }
$$
A square
$$\xymatrix{
A\ar[r]^-{a}\ar[d]_-{c}&B\ar[d]^-{b}\\
C\ar[r]_-{d}&D }
$$
is called {\em 2-commutative} if there is a given 2-isomorphism
$\alpha: dc\ra ba$, in which case we say that the square is
2-commutative {\em with respect to $\alpha$}.

%%%%%%%%%%%%%%%%%%%%%%%%%%%%%%%%%%%%%%%%%%%%%%%%%%%%%%%%%%%%%%%%%%%%%%%%%%%%%%%%%%
\subsection{Categories fibred in groupoids}\label{subsec:cfg}

%For the purposes of this paper, it will be convenient to adopt the
%following common point of view to stacks via categories fibred in
%groupoids.

Let $\cl{C}$ be the category of smooth manifolds\footnote{Manifolds
are not necessarily assumed to be Hausdorff, as this property fails
in many natural examples of spaces of arrows of Lie groupoids.}. We
endow $\cl{C}$ with the Grothendieck topology given by open covers.
We recall the definition of the (strict) 2-category of categories
fibred in groupoids over $\cl{C}$, denoted by $\CFG_\cl{C}$. A
\textbf{category fibred in groupoids} over $\cl{C}$, i.e., an object
in $\CFG_\cl{C}$, is a pair $(\X, \pi)$, where $\X$ is a category
and $\pi:\cl{X}\ra \cl{C}$ is a functor, satisfying the following
conditions:
\begin{enumerate}
\item[(i)]
Any diagram
$$\xymatrix{
&y\ar@{-}[d]\\
U\ar[r]_-{f}&V }
$$
can be completed to a commutative diagram
$$\xymatrix{
x\ar@{-}[d]\ar[r]^{a}&y\ar@{-}[d]\\
U\ar[r]_-{f}&V }
$$
where $f:U\ra V$ is a morphism in $\cl{C}$, $a:x\ra y$ is a morphism
in $\cl{X}$, the vertical lines mean that $\pi(x)=U$ and $\pi(y)=V$,
and the commutativity means that $\pi(a)=f$.

\item[(ii)] Any morphism $a:x\ra y$ in $\cl{X}$ is cartesian, i.e.,
for any commutative diagram of solid arrows as below,
$$\xymatrix{
z\ar@/^1pc/[rr]^{b}_(0.4){}\ar@{-}[d]\ar@{-->}[r]_-{c}
&x\ar[r]_-{a}\ar@{-}[d]&y\ar@{-}[d]\\
W\ar[r]_-{g}&U\ar[r]_-{f}&V }
$$
(i.e., $\pi(a)=f$ and $\pi(b)=fg$), there exists a unique $c$ that
makes the diagram commute (i.e., $ac=b$ and $\pi(c)=g$).
\end{enumerate}

If there is no risk of confusion, we simplify notation and denote a
category fibred in groupoids $(\X,\pi)$ simply by $\X$. We may also
use the notation $\pi_\X$ for $\pi$ if $\X$ is not clear from the
context.

If $\cl{X}$ is a category fibred in groupoids and $U$ is a manifold,
we define the {\bf fiber} of $\cl{X}$ over $U$, denoted  by either
$\cl{X}(U)$ or $\cl{X}_U$, as the category whose objects are the
objects $x$ of $\cl{X}$ that lie over $U$ (i.e., $\pi(x)=U$), and
whose morphisms $a: x\to y$ in $\cl{X}_U$ are those in $\cl{X}$ that
lie over the identity of $U$ (i.e., $\pi(a)=\mr{id}_U$). Conditions
(i) and (ii) above imply that the fibers of $\pi$ over any object of
$\cl{C}$ are groupoids (i.e., categories in which all the morphisms
are invertible). When $a: x\to y$ is a cartesian arrow, with
$f=\pi(a): U\to V$, we refer to $x$ as the {\bf pullback} of $y$ by
$f$. (Note that $x$ is uniquely defined, up to canonical
isomorphism, by $y$ and $f$.) We use the notation $y|_U$ or $f^*y$
for $x$. Given a morphism $b: y'\to y$ over ${\rm{id}}_V$, there is
an induced morphism $f^*b = b|_U : f^*y' \to f^*y$.

A {\bf morphism} between categories fibred in groupoids
$(\X_1,\pi_1)$ and $(\X_2,\pi_2)$ is a functor
$$
F:\cl{X}_1\ra \cl{X}_2
$$
such that $\pi_2 F=\pi_1$. The {\bf 2-morphisms} between
$F,F':\cl{X}_1\ra \cl{X}_2$ are the natural transformations
$\eta:F\ra F'$ such that $\pi_2(\eta(x):F(x)\to
F'(x))=\mr{id}_{\pi_1(x)}$ for any object $x\in\cl{X}_1$. We recall
that any 2-morphism in $\CFG_\cl{C}$ is an isomorphism with respect
to vertical composition.

Two categories fibred in groupoids $\cl{X}$ and $\cl{Y}$ are
\textbf{isomorphic} if there are morphisms $F:\cl{X}\ra\cl{Y}$ and
${F'}:\cl{Y}\ra \cl{X}$ such that the compositions $F {F'}$ and $F'
F$ are isomorphic to the corresponding identities. We recall that
$F$ is an isomorphism in this sense if and only if for any manifold
$U$ the restriction $F_U:\cl{X}_U\ra \cl{Y}_U$ is an equivalence of
categories.

Any manifold $X$ naturally gives rise to a category fibred in
groupoids, still denoted by $X$, whose fiber over a manifold $U$ is
given by $\mathrm{Hom}(U,X)$. A category fibred in groupoids
$\cl{X}$ is \textbf{representable} if there is a manifold $X$ (whose
associated category fibred in groupoids is) isomorphic to it.

We also recall the {\bf fibred product} of morphisms of categories
fibred in groupoids. Let
$$
F_i:\cl{X}_i\ra\cl{Y}
$$
be morphisms of categories fibred in groupoids, for $i=1,2$. The
objects of the fibred product $\cl{X}_1\times_\cl{Y}\cl{X}_2$ (we
will also use the notation $\X_1 \times_{F_1,\cl{Y},F_2} \X_2$) are
triples $(x_1,a,x_2)$ with $x_i\in\cl{X}_i$ and $a:F_1(x_1)\ra
F_2(x_2)$, where $x_1,x_2$ are assumed to lie over the same manifold
$U$, and $a$ lies over $\mr{id}_U$ (hence it is an isomorphism). A
morphism $(b_1,b_2): (x_1,a,x_2)\ra (x_1',a',x_2')$ is given by a
pair of morphisms $b_i: x_i\ra x_i'$ in $\cl{X}_i$, for $i=1,2$,
such that the diagram
$$
\xymatrix{
F_1(x_1)\ar[r]^-{a}\ar[d]_-{F_1(b_1)}&F_2(x_2)\ar[d]^-{F_2(b_2)}\\
F_1(x_1')\ar[r]_-{a'}&F_2(x_2') }
$$
commutes.

We recall that a diagram (in $\CFG_\cl{C}$)
\begin{equation}\label{eq:2cart}
\xymatrix{
\cl{W}\ar[r]^-{}\ar[d]_-{F_1'}&\cl{X}_1\ar[d]^-{F_1}\\
\cl{X}_2\ar[r]_-{F_2}&\cl{Y} }
\end{equation}
is called {\bf 2-cartesian} if it is 2-commutative and the induced
map from $\cl{W}$ to $\X_1\times_\cl{Y} \X_2$ is an isomorphism. In
this case, we refer to this square as a {\bf pullback} diagram, and
we say that $F_1'$ is the {\bf base change} of $F_1$ by $F_2$.
(Occasionally, we may refer to 2-cartesian diagrams just as
cartesian.)

A morphism $F:\X\ra \cl{Y}$ of categories fibred in groupoids is
said to be a \textbf{monomorphism} if, for any manifold $U$, the
restriction $F_U:\cl{X}_U\ra\cl{Y}_U$ of $F$ over $U$ is fully
faithful. The morphism $F$ is said to be an \textbf{epimorphism} if,
for any $U\in\cl{C}$ and any $y\in \cl{Y}_U$, there exists a cover
$(U_\alpha\ra U)_\alpha$ and, for any $\alpha$, there exists
$x_\alpha\in \X_{U_\alpha}$ such that $F(x_\alpha)\simeq
y_{|U_\alpha}$ in $\cl{Y}_{U_\alpha}$. For a morphism $F: X\ra Y$
between manifolds, the condition that $F$ is an epimorphism is
equivalent to the existence of local sections around any point of
$Y$. We also recall that being an epimorphism (resp. monomorphism)
is stable under composition and base change (i.e,, in a 2-cartesian
square \eqref{eq:2cart}, if $F_1$ has this property, then so does
$F_1'$).

A morphism $\cl{X}\ra\cl{Y}$ of categories fibred in groupoids is
called \textbf{representable} if, for any manifold $Y$ and morphism
$Y\ra\cl{Y}$, the fibred product $Y\times_\cl{Y} \cl{X}$ is
representable; this property is preserved under composition and base
change. Given a morphism $X\ra Y$ between manifolds, it is
representable if and only is it is a submersion (see e.g.
\cite[Lem.~71]{metzler}). As a consequence, for a representable
morphism  $\cl{X}\ra\cl{Y}$ of categories fibred in groupoids and a
morphism $Y\ra \cl{Y}$ from a manifold $Y$, the induced map
$Y\times_\cl{Y} \cl{X} \ra Y$ is automatically a submersion (so
representable morphisms are also referred to as {\bf representable
submersions}). It also follows that $\cl{X}\ra\cl{Y}$ is a
representable epimorphism if and only if  $Y\times_\cl{Y} \cl{X} \ra
Y$ is a surjective submersion of manifolds for all $Y\ra \cl{Y}$,
where $Y$ is a manifold.

%%%%%%%%%%%%%%%%%%%%%%%%%%%%%%%%%%%%%%%%%%%%%%%%%%%%%%%%%%%%%%%%%%%%
\subsection{Differentiable stacks}\label{subsec:stacks}

\begin{definition}
A category fibred in groupoids $\cl{X}$ is called a \textbf{stack}
if the following two conditions are satisfied:

\begin{enumerate}
\item[(A1)]  \label{itm:ax-1}  Given a manifold $U$ and
two objects $x, y$ in $\cX_U$, for every open cover $(U_\alpha\to
U)_\alpha$, and for every collection of isomorphisms $\phi_\alpha:
x_{|U_\alpha} \to y_{|U_\alpha}$ over $U_\alpha$ such that
$\phi_{\alpha |U_{\alpha\beta}}= \phi_{\beta |U_{\alpha\beta}}$,
there is a unique isomorphism $\phi: x\to y$ such that
$\phi_{|U_\alpha}=\phi_\alpha$. (Here $U_{\alpha\beta}=U_\alpha
\times_U U_\beta$.)

\item[(A2)]  \label{itm:ax-2}  Let $U$ be a manifold and
 $(U_\alpha\to U)_\alpha$ be an open cover. Let $x_\alpha$ be an object in $\cX_{U_\alpha}$,
and let $\phi_{\beta\alpha}: x_\alpha|_{U_{\alpha\beta}} \to
x_\beta|_{U_{\alpha\beta}}$ be morphisms over $U_{\alpha\beta}$
satisfying $\phi_{\alpha\beta}\circ \phi_{\beta\gamma}
=\phi_{\alpha\gamma}$ (over $U_{\alpha\beta\gamma}= U_\alpha\times_U
U_\beta \times_U U_\gamma$). Then there exist an object $x$ over
$U$, and isomorphisms $\phi_\alpha: x|_{U_\alpha}\to x_\alpha$ over
$U_\alpha$ such that $\phi_\beta = \phi_{\beta\alpha}\circ
\phi_\alpha$ (over $U_{\alpha\beta}$).
\end{enumerate}

A category fibred in groupoids is a \textbf{prestack} if it
satisfies (A1). Note that (A1) implies that $x$ in (A2) is unique up
to canonical isomorphism.
\end{definition}

A morphism between stacks is just a morphism of the underlying
categories fibred in groupoids. If $F:\cl{X}\ra\cl{Y}$ is a morphism
of stacks, then the definition of epimorphism given in
Section~\ref{subsec:cfg} is equivalent to the one given in
\cite[Def.~2.3]{bx}.

We will need the notion of stackification of a category fibred in
groupoids. We recall its main, well-known properties (see e.g.
\cite[Prop.~52, Lem.~53]{metzler}):
%\cite[Prop.~52]{metzler}:

\begin{proposition}\label{PropStackification}
Let $\cl{X}$ be a category fibred in groupoids. Then there is a
stack $\cl{X}^\sharp$, called the \textbf{stackification} of
$\cl{X}$, and a morphism $\stack:\cl{X}\ra\cl{X}^\sharp$ such that
that the following properties hold:
\begin{itemize}
\item[(i)] For any stack $\cl{S}$ and
morphism $F:\cl{X}\ra \cl{S}$ there is a pair $(F^\sharp,\zeta)$,
where $F^\sharp:\cl{X}^\sharp\ra \cl{S}$ and $\zeta: F^\sharp \stack
\lmap{\sim} F$.

\item[(ii)] Let $\cl{S}$ be a stack, $F_i: \cl{X}\ra \cl{S}$, $i=1,2$, be morphisms and
$\eta: F_1\lmap{\sim} F_2$. Let $F_i^\sharp: \cl{X}^\sharp\ra
\cl{S}$ and $\zeta_i: F_i^\sharp \stack \lmap{\sim} F_i$, $i=1,2$,
be as in $(i)$. Then there exists a unique $\eta^\sharp: F_1^\sharp
\lmap{\sim} F_2^\sharp$ such that $\eta
* \zeta_1=\zeta_2
*(\eta^\sharp \,\scirc\,\mr{id}_{\stack})$.

\end{itemize}
Moreover,

\begin{itemize}
\item[(iii)] If $\X$ is a prestack then the stackification map
$\stack:\cl{X}\ra\cl{X}^\sharp$ is a monomorphism and an
epimorphism.

\item[(iv)] Let $F:\cl{X}\ra\cl{Y}$ be a morphism of prestacks. Then, the
stackified map $F^\sharp:\cl{X}^\sharp\ra \cl{Y}^\sharp$ is an
isomorphism if and only if $F$ is a monomorphism and an epimorphism.

\end{itemize}
\end{proposition}

\begin{definition}\label{DefDiffStack}
A stack $\cl{X}$ is called \textbf{differentiable} if there exists a
representable epimorphism $X\ra\cl{X}$ from a manifold $X$. We call
such a morphism an \textbf{atlas} of $\cl{X}$.
\end{definition}

We mention some standard but important examples.

\begin{example}
\
\begin{itemize}
\item[(a)] The category fibred in groupoids associated with a manifold $X$
is a differentiable stack. The manifold $X$ itself is an atlas.

\item[(b)] If a Lie group $G$ acts on a manifold $X$, then there is an associated
{\em quotient stack} $[X/G]$, which is a differentiable stack for
which $X$ can be taken as an atlas. In particular, when $X$ is a
point, the quotient stack is called the {\bf classifying space} of
the Lie group $G$, and it is denoted by $BG$.

\item[(c)] The definition of quotient stack can be generalized to the
setting of Lie groupoids acting on manifolds. We will provide more
details and references in Section~\ref{subsec:LG}.
\end{itemize}
\end{example}

\begin{remark}\label{rem:higher}
We point out that there is a broader viewpoint to stacks through the
theory of {\em higher stacks} developed in
\cite{lurie,Rezk,toen-vezzosi:hagI}.
This theory unifies all levels of $n$-stacks, where $n$ is a
non-negative integer, or $\infty$.
In this hierarchy, a $0$-stack is a sheaf. Recall that a presheaf in
a category $\mathcal{C}$ is a contravariant functor from
$\mathcal{C}$ to the category of sets. If the category is endowed
with a Grothendieck topology, then one may define the so-called
local isomorphisms in the category of presheaves on $\mathcal{C}$.
Localizing with respect to local isomorphisms gives us the category
of sheaves on $\mathcal{C}$. This way of defining sheaves, going
back to \cite{SGA4}, is totally internal; one arrives at the concept
of sheaves without explicitly defining what they are.

For higher stacks, one passes from a category to a simplicial
enriched category and from sets to simplicial sets. In this
framework, one may consider geometric stacks
\cite{toen-vezzosi:hagII}, which extend what we call differentiable
stacks in this paper; roughly, these are the higher stacks
presentable by higher groupoids (see also
\cite{Pridham:higher-stack}). Although the theory is mostly driven
by algebraic geometric applications, the same ideas carry over to
differential geometry, see e.g. \cite{nss:2} (and the theory in fact
simplifies in the context of manifolds). From this perspective, a
stack is the ``stackification'' (analogous to the localization in
the case of sheaves) of a (higher) functor from the category of
manifolds to the bicategory of groupoids. Such a functor may be
explicitly expressed as a category fibred in groupoids, which is the viewpoint taken
in this paper.
\end{remark}

%%%%%%%%%%%%%%%%%%%%%%%%%%%%%%%%%%%%%%%%%%%%%%%%%%%%%%%%%%%%%%%%%%%%%
\subsection{Morphisms of differentiable stacks}

In the sequel, the following weak form of representability will be
central:

\begin{definition}\label{def:weakep}
A morphism of differentiable stacks $F: \X \ra \cl{Y}$ is called
{\bf weakly representable} if there exists an {\em atlas} $Y\to
\cl{Y}$ such that the fibred product $\X \times_{\cl{Y}} Y$ is
representable.
\end{definition}

\begin{proposition}\label{prop:wrep}
A morphism $F: \X \ra \cl{Y}$ is weakly representable if and only if
for all representable morphisms $U\to \cl{Y}$, where $U$ is a
manifold, the fibred product $\X \times_{\cl{Y}} U$ is
representable.
\end{proposition}

\begin{proof}
The `if' part is clear since $\cl{Y}$ is assumed to have an atlas.
For the converse, take an atlas $Y\to \cl{Y}$ such that $\X
\times_{\cl{Y}} Y$ is representable, and a representable map $U\to
\cl{Y}$ from a manifold $U$. We have to show that $\X
\times_{\cl{Y}} U$ is representable. Taking the fibred product
$P=U\times_\cl{Y} Y$, the projection $P\to U$ is a surjective
submersion; moreover, $P\to Y$ is a submersion, so that $\X
\times_{\cl{Y}} P$ is representable. There exists an open cover
$(U_\alpha)_\alpha$ of $U$ and manifolds $(F_\alpha)_\alpha$ such
that $U_\alpha\times F_\alpha$ is open embedded in $P$ and the
restriction of $P\to U$ is the projection $U_\alpha\times
F_\alpha\to U_{\alpha}$. The fibred products $\X \times_{\cl{Y}}
U_\alpha$ form an open cover of $\X \times_{\cl{Y}} U$ (in the stack
sense) so that it is enough to show that each $\X \times_{\cl{Y}}
U_\alpha$ is representable. From the abstract properties of fibred
products we have
$$ (\X \times_{\cl{Y}} U_\alpha)\times {F_\alpha}=
(\X \times_{\cl{Y}} U_\alpha)\times_{U_\alpha} (U_\alpha\times
F_\alpha)= (\X \times_{\cl{Y}} P )\times_{P} (U_\alpha\times
F_\alpha),$$ so that $(\X \times_{\cl{Y}} U_\alpha)\times
{F_\alpha}$ is representable, and we conclude  that $\X
\times_{\cl{Y}} U_\alpha$ is representable, as required.
\end{proof}

It is clear that all representable morphisms of differentiable
stacks are weakly representable. Note that any map $F: X\to Y$
between manifolds is weakly representable, but it is representable
if and only if it is a submersion. In fact, a morphism $\X \to Y$
into a manifold $Y$ is weakly representable if and only if $\X$ is
representable. (Note that the notions of representable and weakly
representable coincide in the topological setting, see e.g.
\cite[Lemma~2.6(3)]{hein}).

\begin{proposition}\label{propwrf} Let $F:\cl{X}\ra\cl{Y}$ be a
weakly representable morphism of differentiable stacks. Then the
functor $F$ is faithful.
\end{proposition}

\begin{proof}
It is well known that $F$ is faithful if (and only if) the functor
$F_U:\cl{X}_U\ra\cl{Y}_U$ is faithful for any manifold $U$. Since
$\cl{X}_U$ is a groupoid, it is enough to show that, for any arrow
$a:x\ra x$ in $\cl{X}_U$, if $F(a)=\mr{id}_{F(x)}$ then
$a=\mr{id}_x$. Fix such an arrow $a$.

Since $F$ is weakly representable, there is an atlas
$\pi:Y\ra\cl{Y}$ such that $\cl{X}\times_\cl{Y} Y$ is representable.
The object $x\in\cl{X}_U$ corresponds to a morphism $x:U\ra \cl{X}$.
The fibred product $U'=U\times_\cl{Y} Y$ is a manifold, and the
projection $U'\ra U$ is a surjective submersion. So we can take an
open cover $(U_\alpha)$ of $U$ with local sections $U_\alpha\ra U'$
of $U'\ra U$. For any value of the index $\alpha$, we get a
2-commutative diagram
$$\xymatrix{
& Y\ar[d]^{\pi}\\
U_\alpha\ar[ru]^-{y_\alpha}\ar[r]_-{F(x_{|U_\alpha})}&\cl{Y}. }$$ We
will show that $a_{|U_\alpha}=\mr{id}_{x_{|U_\alpha}}$, so that
$a=\mr{id}_x$ (since $\cl{X}$ is a stack), as desired. Interpreting
$y_\alpha$ as an object of $Y$ over $U_\alpha$, the 2-isomorphism
$d_\alpha$ that makes the above diagram 2-commute is interpreted as
an isomorphism $d_\alpha: F(x_{|U_\alpha})\ra \pi(y_\alpha)$ in
$\cl{Y}_U$, so that the triple $(x_{|U_\alpha},d_\alpha, y_\alpha)$
is an object of $\cl{X}\times_\cl{Y} Y$ over $U_\alpha$. The pair
$(a_{|U_\alpha},\mr{id}_{y_\alpha})$ is a morphism of
$(x_{|U_\alpha},d_\alpha, y_\alpha)$ in itself in the category
$(\cl{X}\times_\cl{Y} Y)_{U_\alpha}$; this follows from  the
hypothesis $F(a)=\mr{id}_{F(x)}$, which implies that
$F(a_{|U_\alpha})= \mr{id}_{F(x_{|U_\alpha})}$. Since
$\cl{X}\times_\cl{Y} Y$ is representable, there is a unique arrow of
$(x_{|U_\alpha},d_\alpha, y_\alpha)$ in itself, namely the identity.
It follows that $a_{|U_\alpha}=\mr{id}_{x_{|U_\alpha}}$ and the
proof is complete.

\end{proof}

\begin{definition}\label{def:weak}
A morphism $F:\cl{X}\ra \cl{Y}$ of differentiable stacks is called a
{\bf submersion} (resp. {\bf immersion}) if there exist atlases
$X\ra\cl{X}$ and $Y\ra\cl{Y}$ such that the induced map of manifolds
$Y\times_{\cl{Y}} X\ra Y$ is a submersion (resp. {\bf immersion}).
\end{definition}

The notion of immersion will be used only for \'etale stacks (see
Section~\ref{subsec:LG}).

One may verify that, for manifolds, the above definitions coincide
with the usual notions of submersion and immersion. Also, any
representable morphism is a submersion. Note, however, that
submersions (or immersions) need not be representable morphisms; in
fact, most submersions we will deal with are of the form $\X \to Y$,
where $Y$ is a manifold, and, as mentioned, such a map cannot be
representable unless $\X$ is.

%\comment{Remark alternative descriptions? a map $\X \ra \cl{Y}$ is
%submersion iff there are atlases such that map is covered by
%submersion $X_0\to Y_0$ ($X_0$ need not be the fiber product); for
%immersion, we have a similar characterization, by $X_0$ must be the
%fibred product $\X\times_{\cl{Y}}Y_0$...}

We will be particularly interested in submersions, and it will be
useful to have equivalent characterizations of such morphisms.

\begin{proposition}\label{prop:charactsub}
Given a morphism of differentiable stacks $F: \X\ra \cl{Y}$, the following are equivalent:
\begin{itemize}
\item[(a)] $F$ is a submersion.

\item[(b)] There exists an atlas $X\ra\cl{X}$ such
that the composition $X\ra\cl{Y}$ is representable.

\item[(c)] For all representable morphisms $U\ra\cl{X}$ and $V\ra\cl{Y}$
from manifolds $U$ and $V$, the induced map of manifolds
$V\times_{\cl{Y}} U\ra V$ is a submersion.

\item[(d)] For all representable morphisms $U\ra \cl{X}$ from a manifold
$U$, the composition $U\ra \cl{Y}$ is representable.

\end{itemize}
\end{proposition}

Submersions satisfy the following natural properties:

\begin{proposition}\label{prop:propsubm}
The following holds:

\begin{enumerate}
\item[(a)] The composition of submersions is a submersion.

\item[(b)] A base change of a submersion is a submersion.

\item[(c)] Let $F: \cl{X}\ra\cl{Y}$ and $F':\cl{Y}\ra \cl{Z}$ be morphisms
of differentiable stacks. If $F$ and $F' F$ are submersions and $F$
is an epimorphism, then $F'$ is a submersion.

\end{enumerate}
\end{proposition}

The proofs of Propositions~\ref{prop:charactsub} and
\ref{prop:propsubm} can be found in Appendix~\ref{app:morphisms}.

%%%%%%%%%%%%%%%%%%%%%%%%%%%%%%%%%%%%%%%%%%%%%%%%%%%%%%%%%%%%%%%%%%%%%%%%%%%%%
\subsection{Lie groupoids and Hilsum-Skandalis maps}\label{subsec:LG}

We recall some basic facts about Lie groupoids and their relation
with differentiable stacks. See e.g. \cite{bx,MM,MM2} for more details,
and \cite{dH} for a geometric viewpoint.

A \textit{groupoid} is a category in which all the morphisms are
invertible. Hence a groupoid consists of a set $G_0$ of objects, a
set $G$ of morphisms, and structural maps satisfying suitable
compatibility conditions. We denote the source and target maps by
$s$ and $t$, we write $i$  for the inversion map, $1$ for the unit
map, and $m$ for the multiplication:
\begin{equation}\label{eq:structmaps1}
G\times_{s,G_0,t}G\lmap{m}{G}\lmap{s,t} G_0\lmap{1}{G}\lmap{i} {G},
\end{equation}
where
$$
G\times_{s,G_0,t}G = \{(g,h)\,|\, s(g)=t(h)\}.
$$
We will use the notation $m(g,h)=g\cdot h=gh$, $1(g_0)=1_{g_0}$ and
$i(g)=g^{-1}$. We may also denote the structural maps by $s_G$,
$t_G$ etc. if we need to be more specific. We denote a groupoid by
$G\rra G_0$, or simply by $G$ if there is no risk of confusion. In
this context, a  morphism is just a functor.

A \textbf{Lie groupoid} $G\rra G_0$  is a groupoid in the category
of smooth manifolds, such that source and target maps are
submersions (necessarily surjective). It is called {\bf \'etale} if
the source map (or, equivalently, the target map) is a local
diffeomorphism, or equivalently, if $G$ and $G_0$ have the same
dimension.

A right \textbf{action} of a Lie groupoid $G$ on a manifold $X$ is
defined by a pair of maps $a: X\ra G_0$ and
$$
X\times_{a,G_0,t} G\map X, \qquad (x,g)\mapsto x\cdot g=xg
$$
such that $a(xg)=s(g)$ and
$$
(xg)h=x(gh),\qquad  x1=x.
$$
We say that $G$ {\bf acts on $X$ along $a: X\to G_0$}, and the map
$a$ is often referred to as the {\bf moment map} of the action. A
$G$-{\bf equivariant map} between manifolds equipped with
$G$-actions is a map that commutes with moment maps and actions.

A right \textbf{$G$-bundle} is a manifold $P$ equipped with a right
$G$-action along $a:P \to G_0$ and a map
$$
r: P\ra S,
$$
where $S$ is a manifold, such that the action is {\bf on the fibers}
of $r$, i.e., $r(zg)= r(z)$ for composable $z\in P$ and $g\in G$. A
right $G$-bundle is \textbf{principal} if $r$ is a surjective
submersion and the induced map
\begin{equation}\label{eq:principalact}
P\times_{a,G_0,t} G\map P\times_S P, \qquad (z,g)\mapsto (z,zg),
\end{equation}
is a diffeomorphism. Similar definitions hold for left actions and
left bundles.

For a given right action of a Lie groupoid $G$ on a manifold $X$
along $a: X\to G_0$, there is an associated differentiable stack,
called the \textbf{quotient stack} and denoted by $[X/G]$.  The
objects of $[X/G]$ are principal right $G$-bundles equipped with a
$G$-equivariant map $P\ra X$, while the morphisms of $[X/G]$ are
morphisms of principal bundles (over different bases, in general)
that commute with the maps to $X$. Any $G$-equivariant map $f: X\to
Y$ naturally induces a morphism $[X/G]\to [Y/G]$.

\begin{remark}\label{rem:equivsub}
Note that the induced morphism of quotient stacks $[X/G]\to [Y/G]$
is a submersion (in the sense of Def.~\ref{def:weak}) if $f$ is.
\end{remark}

There is a map of stacks $X\ra [X/G]$, taking a smooth map of
manifolds $f: U\to X$ to
$$
P= (af)^*G= U \times_{a f,G_0,t} G,
$$
which is naturally a principal right $G$-bundle over $U$, equipped
with the equivariant map $P\to X$, $(u,g)\mapsto f(u)g$. The map
$X\ra [X/G]$ defines an atlas of the quotient stack. A particular
case of this construction is when $X=G_0$, equipped with its
canonical $G$-action: $(g_0,g)\mapsto s(g)$. In this case the
quotient stack is called the \textbf{classifying space} of the
groupoid $G$ and it is denoted by $BG$.

There is a close relation between Lie groupoids and differentiable
stacks endowed with an atlas. On the one hand, given a Lie groupoid
$G\rra G_0$, one considers the associated classifying space $BG$,
which comes with an atlas $G_0\ra BG$ fitting into the following
2-cartesian square:
$$
\xymatrix{
G\ar[r]^-{t}\ar[d]_-{s}&G_0\ar[d]\\
G_0\ar[r]&BG. }
$$
Conversely, given a differentiable stack $\cl{G}$ endowed with an
atlas $G_0\ra \cl{G}$, we define $G= G_0\times_\cl{G}G_0$, which has
an induced Lie groupoid structure over $G_0$ such that $BG$ is
canonically isomorphic to $\cl{G}$. The Lie groupoid $G\rra G_0$ is
called a \textbf{presentation} of $\cl{G}$ (and of $BG$). We say
that the differentiable stack $\cl{G}$ is {\bf \'etale} if it can be
presented by an \'etale Lie groupoid.

\begin{example}\label{ex:action}
Given a (right) action of a Lie groupoid $G \rra G_0$ on a manifold
$X$ along $a: X\to G_0$, the quotient stack $[X/G]$ is presented by
the {\it action} (or {\it translation}) groupoid $X \rtimes G$: its
space of arrows is $X \times_{a,G_0,t} G$, the source map is
$(x,g)\mapsto xg$, the target map is $(x,g)\mapsto x$, and the
multiplication is given by $(x,g)(y,h) = (x,gh)$.
\end{example}

A $G$-$H$-\textbf{bibundle} is defined by a manifold $P$ and Lie
groupoids $G\rra G_0$ and $H\rra H_0$  so that $P$ carries a left
$G$-action along $a: P \to G_0$ and on the fibers of $b$, and a
right $H$-action along $b: P\to H_0$ and on the fibers of $a$, in
such a way that the two actions commute. We represent such bibundle
by the diagram
\begin{equation}\label{eq:GHbib}
 \xymatrix{
G\ar@<0.5ex>[d]\ar@<-0.5ex>[d]&& H\ar@<0.5ex>[d]\ar@<-0.5ex>[d]\\
G_0&P\ar[l]^-{a}\ar[r]_-{b}&H_0 . }
\end{equation}
We will simply write $P$ for the bibundle. Two $G$-$H$ bibundles $P$
and $Q$ are \textbf{isomorphic} if there is a diffeomorphism $P\to
Q$ preserving the actions.

A bibundle \eqref{eq:GHbib} is called \textbf{right principal} if it
is a principal $H$-bundle over $G_0$ (in particular, $a$ is a
surjective submersion), and it is called \textbf{biprincipal} if it
is also a principal $G$-bundle over $H_0$ (so that  $b$ is also a
surjective submersion). We say that two groupoids $G$ and $H$ are
\textbf{Morita equivalent} if there exists a biprincipal
$G$-$H$-bibundle.

Given a right principal $H$-$K$-bibundle $Q$,
$$ \xymatrix{
H\ar@<0.5ex>[d]\ar@<-0.5ex>[d]&& K\ar@<0.5ex>[d]\ar@<-0.5ex>[d]\\
H_0&Q\ar[l]^-{}\ar[r]_-{}&K_0 , }
$$
and a right principal $H$-bundle $P$, the quotient of $P\times_{H_0}
Q$ by the $H$-action $(z,w)h=(z h,h^{-1}w)$ is a right principal
$K$-bundle, denoted by $P\otimes_H Q$. Moreover, if $P$ is a right
principal $G$-$H$-bibundle, so that we have
$$
\xymatrix{ G\ar@<0.5ex>[d]\ar@<-0.5ex>[d]&&
H\ar@<0.5ex>[d]\ar@<-0.5ex>[d]&&K\ar@<0.5ex>[d]\ar@<-0.5ex>[d]\\
G_0&P\ar[l]_-{}\ar[r]^-{}&H_0&Q\ar[l]_-{}\ar[r]^-{}&K_0, }
$$
then $P\otimes_H Q$ is naturally a right principal $G$-$K$-bibundle,
$$
\xymatrix{
G\ar@<0.5ex>[d]\ar@<-0.5ex>[d]&& K\ar@<0.5ex>[d]\ar@<-0.5ex>[d]\\
G_0&P\otimes_H Q\ar[l]^-{}\ar[r]_-{}&K_0. }
$$
If $P$ and $Q$ are biprincipal, then so is $P\otimes_H Q$, which
implies that Morita equivalence of Lie groupoids is a transitive
relation (the fact that it is symmetric and reflexive is
straightforward).

A {\bf Hilsum-Skandalis map} between Lie groupoids $G$ and $H$ is an
isomorphism class of right principal $G$-$H$-bibundles. Any groupoid
morphism $\phi: G\to H$ gives rise to a Hilsum-Skandalis map through
the right principal $G$-$H$-bibundle defined by the fibred product
\begin{equation}\label{eq:HSbb}
G_0\times_{\phi_0,H_0,t}H,
\end{equation}
with actions given by $g(g_0,h)=(t(g),\phi(g)h)$ and $(g_0,h)h' =
(g_0,hh')$. Since right principal bibundles represent
Hilsum-Skandalis maps, we also refer to them as {\bf
Hilsum-Skandalis bibundles}.

Given a right principal $G$-$H$-bibundle $Q$, there is an associated
morphism between the stacks $BG$ and $BH$, sending a principal right
$G$-bundle $P$ (that is, an object of $BG$) to the principal right
$H$-bundle $P\otimes_G Q$. In this way a Hilsum-Skandalis map from
$G$ to $H$ induces a morphism $BG\to BH$, defined up to
2-isomorphism. In particular, any groupoid morphism $\phi: G\to H$
gives rise to a morphism $BG\to BH$ via \eqref{eq:HSbb}; moreover,
this map $BG\to BH$ is an isomorphism whenever the groupoid morphism
$\phi$ is a {\bf weak equivalence}, that is, when the following two
conditions are satisfied  (see e.g. \cite[Sec.~5.4]{MM}): (1) the
map $s \circ \mathrm{pr}_2 : G_0 \times_{\phi_0, H_0, t} H\to H_0$
is a surjective submersion, and (2) the square
$$
\xymatrix{G \ar[r]_\phi \ar[d]^{(s,t)} & H \ar[d]_{(s,t)}\\
G_0\times G_0 \ar[r]_{\phi_0\times \phi_0} & H_0\times H_0}
$$
is cartesian.

Conversely, any morphism $BG\ra BH$ of stacks, taken up to
2-isomorphism, is presented by a unique Hilsum-Skandalis map between
the groupoids $G$ and $H$. Indeed, for a given morphism $F: BG\to
BH$, a representative of the corresponding Hilsum-Skandalis map is
given by the bibundle
$$
G_0\times_{BH} H_0,
$$
which can be identified with $F(G)$ as a principal right $H$-bundle
over $G_0$ (here we view $G$ as a principal $G$-bundle relative to
right multiplication); we denote by $r: F(G)\to G_0$ the projection.
The left $G$-action on $F(G)$ is given by $F(m): G\times_{s,G_0,r}
F(G)=s^*(F(G))\to F(G)$, where we view the multiplication $m:
G\times_{s,G_0,t} G = s^*G \to G$ as a morphism of principal right
$G$-bundles (covering the map $t: G\to G_0$).

Given Lie groupoids $X$, $G$, and $H$, and morphisms $a: X\to H$,
$b: G\to H$ so that either $a_0: X_0\to H_0$ or $b_0: G_0\to H_0$ is
a submersion, we denote by $X\times_H G$ the Lie groupoid given by
their {\bf weak fibred product} \cite[Sec.~5.3]{MM}: its space of
objects is $X_0\times_{a_0,H_0,s} H \times_{t,H_0,b_0} G_0$, its
space of arrows is $X \times_{sa,H_0,s} H \times_{t,H_0,sb} G$, and
multiplication is given by $(x,h,g)(x',h',g')=(xx', h', gg')$
(source and target maps are given by $(x,h,g)\mapsto (s(x),h,s(g))$
and $(x,h,g)\mapsto (t(x),b(g)ha(x)^{-1},t(g))$, respectively). Note
that there is a canonical map
\begin{equation}\label{eq:canmap}
B(X\times_H G) \to BX \times_{BH} BG,
\end{equation}
induced by the natural maps $B(X\times_H G) \to BX$ and $B(X\times_H
G) \to BG$, which correspond to the groupoid morphisms $X\times_H G
\to X$ and $X\times_H G \to G$. We recall the following fact:

\begin{proposition}\label{prop:Bprod}
The canonical map $B(X\times_H G) \to BX \times_{BH} BG$ in
\eqref{eq:canmap} is an isomorphism.
\end{proposition}
\begin{proof}
We will describe the inverse map. Given $(P,\varphi, Q)$ in $BX
\times_{BH} BG$ over a manifold $U$, we have that $P\times_U Q$ is a
principal right $X\times_H G$ bundle over $U$ as follows. Let $j:
P\to X_0$ and $k: Q\to G_0$ be the moment maps for the actions on
$P$ and $Q$, and note that the images of $P$ and $Q$ in $BH$ are the
$H$-bundles $(P\times_{a_0j,H_0,t} H)/X$ and $(Q\times_{b_0k,H_0,t}
H)/G$, respectively. Denoting an object in $X\times_H G$ by
$(x_0,h,g_0) \in X_0\times_{H_0} H \times_{H_0} G_0$, the moment map
for the action on $P\times_U Q$ is the map
$$
(p,q)\mapsto (x_0,h,g_0),
$$
where $x_0=j(p)$, $g_0=k(q)$, and $h \in H$ is defined as follows:
given $(z,z')\in P\times_U Q $ and the $H$-equivariant map $\varphi:
(P\times_{a_0j,H_0,t} H)/X \to (Q\times_{b_0k,H_0,t} H)/G$, we let
$[z,1]$ denote the $X$-orbit of $(z,1) \in P\times_{H_0} H$, and $h$
is uniquely defined by the condition that $[z',h]=\varphi([z,1])$.
The right action of $X\times_H G$ on $P\times_U Q$ is
$(z,z')(x,h,g)=(zx,z'g)$.
\end{proof}

\subsubsection*{Some examples of Hilsum-Skandalis bibundles}

The next two lemmas give explicit examples of Hilsum-Skandalis
bibundles that we will need later, see
Sections~\ref{subsecTechnicalLemmas} and \ref{SubsecTensorBundles}.

\begin{lemma}\label{lem:HSfacts1}
Let $\G$ and $\X$ be differentiable stacks, presented by $G\rra G_0$
and $X\rra X_0$, and let $M$ be a manifold. Consider a  submersion
$\G\to M$ and a morphism $\X \to M$. Then the fibred product
$\X\times_M \G$ is a differentiable stack presented by $X\times_M
G\rra X_0\times_M G_0$, and the projection $\X\times_M \G \to \X$
corresponds to the Hilsum-Skandalis bibundle $G_0 \times_M X$, with
actions given by $(x,g)(g_0,x')= (t(g),x x')$ (for the moment map
$(g_0,x)\mapsto (t(x),g_0)$) and $(g_0,x')x=(g_0,x'x)$ (for the
moment map $(g_0,x)\mapsto s(x)$).
\end{lemma}

\begin{proof}
The assertion that $\X\times_M \G=BX\times_M BG$ is a differentiable
stack presented by $X\times_M G\rra X_0\times_M G_0$ is a
consequence of the canonical isomorphism $B(X\times_M G)\to BX
\times_M BG$ of Prop.~\ref{prop:Bprod} (the fact that $BG\to M$ is a
 submersion guarantees that $G_0\to M$ is a submersion, and hence
$X\times_M G$ is a Lie groupoid). We note that this isomorphism
commutes with the projections to $BX$, where the projection
$B(X\times_M G)\to BX$ is associated with the groupoid projection
$X\times_M G \to X$. We conclude (cf. \eqref{eq:HSbb}) that this
projection is represented by the Hilsum-Skandalis bibundle
$(X_0\times_M G_0)\times_{X_0} X \simeq G_0 \times_M X$, with the
actions described in the statement of the Lemma.
\end{proof}

\begin{lemma}\label{lem:HSfacts2}
Let $E_1$ and $E_2$ be Hilsum-Skandalis bibundles for stack
morphisms $F_1:BX\to BY$ and $F_2: BX\to BZ$, respectively. Assume
that one of the two maps is a submersion. Let $M$ be a manifold, and
suppose that we have maps $F_3: BY\to M$ and $F_4: BZ\to M$ so that
$F_3F_1 = F_4 F_2$. Then $E_1 \times_{X_0} E_2$ is a
Hilsum-Skandalis bibundle for the induced morphism $F: BX \to
BY\times_M BZ$.
\end{lemma}

\begin{proof}
Using the identifications $E_1=F_1(X)$ and $E_2=F_2(X)$ and denoting
by $\psi: BY\times_M BZ \to B(Y\times_M Z)$ the isomorphism
described in the proof of Prop.~\ref{prop:Bprod}, we conclude that
$\psi F(X)=F_1(X)\times_{X_0} F_2(X)$. The natural left $X$-action
on $F_1(X)\times_{X_0} F_2(X)$ coincides with the diagonal action.
\end{proof}

%%%%%%%%%%%%%%%%%%%%%%%%%%%%%%%%%%%%%%%%%%%%%%%%%%%%%%%%%%%%%%%%%%%%%%%%%%%%%%%%%%%%
\section{Stacky Lie groupoids and actions}\label{sec:Grpd}

In this section we define stacky Lie groupoids and their actions on
differentiable stacks.

%%%%%%%%%%%%%%%%%%%%%%%%%%%%%%%%%%%%%%%%%%%%%%%%%%%%%%%%%%%%%%%%%%%%%%%%%%%%%%%%%%%%%%%%%%%%%%%%
\subsection{Stacky Lie groupoids}\label{subsecGrpdActions}

Let $M$ be a manifold\footnote{The manifold $M$ is assumed to be
Hausdorff and second countable.} and $\G$ be a category fibred in
groupoids over $\cl{C}$. The manifold $M$ is to be thought of as the
space of units of a groupoid structure on $\G$. We use the same
symbol $M$ for the differentiable stack associated with $M$. We will
consider the following data and conditions:

\begin{itemize}

\item[{\bf (g1)}] Morphisms $\sour$, $\tar$, $\un$, $\inv$, and $\mult$
(called {\it source}, {\it target}, {\it unit}, {\it inverse}, and
{\it multiplication} maps, respectively) as follows:
\begin{align*}
&\cl{G}\lmap{\sour,\tar} M\lmap{\un}\cl{G}\lmap{\inv} \cl{G},\\
& \cl{G} \times_{\sour,M,\tar}
%\slbtimes{\sour}{M}{\tar}
\cl{G}\lmap{\mult}\cl{G},\qquad
\mult(g,h)=g\cdot h=gh.
\end{align*}

For the identity and the inverse we will also use the notation
$$
\un(x)=1_x=1,\qquad\qquad \inv(g)=g^{-1}.
$$

The multiplication morphism is defined on the fibred product of
$\sour:\cl{G}\to M$ and $\tar:\cl{G}\to M$,
\begin{equation}\label{eq:fibredprod}
\xymatrix{
 \cl{G} \times_{\sour,M,\tar}
 %\slbtimes{\sour}{M}{\tar}
 \cl{G}\ar[r]^-{n_\sour}\ar[d]_-{n_\tar} &\cl{G}\ar[d]^-{\tar} \\
\cl{G} \ar[r]_-{\sour} & M, }
\end{equation}
which generalizes the space of composable arrows on Lie groupoids.

\item[{\bf (g2)}] The morphisms $\sour$, $\tar$, $\un$, $\inv$, and $\mult$ are assumed to satisfy the following
identities
\begin{eqnarray*}
 \sour \un &=&\mr{id}_M \\
\tar \un&=& \mr{id}_M\\
\sour \inv&=&\tar\\
\tar \inv&=&\sour\\
\sour \mult&=& \sour n_{\sour}\\
\tar \mult&=& \tar n_{\tar}.
\end{eqnarray*}

We also require axioms analogous to those of Lie groupoids, but now
in a weaker sense (see {\bf (g3)} below). Let us consider the
morphisms
\begin{itemize}
\item[$\bullet$] $\mult (\mr{id}\times \mult):
\cl{G} \times_{\sour,M,\tar} \cl{G}\times_{\sour,M,\tar}\cl{G} \map
\cl{G}$ and $\mult (\mult \times \mr{id}):
\cl{G}\times_{\sour,M,\tar}\cl{G}\times_{\sour,M,\tar}\cl{G} \map
\cl{G}$, encoding the two possible ways to compose three elements in
$\cl{G}$;

\item[$\bullet$] $\mult \langle \un \tar,\mr{id}\rangle: \cl{G}\map \cl{G}$ and $\mult \langle \mr{id},
\un \sour \rangle :\cl{G}\map \cl{G}$, encoding multiplication by
the identity on the left and on the right. Here we use the notation
$\langle \un \tar,\mr{id}\rangle: \cl{G} \to
\cl{G}\times_{\sour,M,\tar} \G$ for the map induced by $\un \tar:
\G\to \G$ and $\mr{id}: \G\to \G$.

\item[$\bullet$] $\mult \langle \inv,\mr{id}\rangle: \cl{G}\map \cl{G}$ and
$\mult \langle \mr{id}, \inv \rangle :\cl{G}\map \cl{G}$, encoding
multiplication by the inverse on the left and on the right.

\end{itemize}

\item[{\bf (g3)}] Five 2-isomorphisms $\alpha$, $\lambda$, $\rho$, $\i_l$
and $\i_r$,
\begin{align*}
& \alpha: \mult (\mr{id}\times \mult)\lmap{\sim} \mult(\mult \times \mr{id}), \\
&  \lambda: \mult\langle \un \tar,\mr{id}\rangle\lmap{\sim} \mr{id}, \;\;
\rho: \mult\langle \mr{id}, \un \sour \rangle\lmap{\sim} \mr{id},\\
&  \i_l: \mult \langle \inv,\mr{id}\rangle\lmap{\sim} \un \sour,
\;\; \i_r: \mult \langle \mr{id}, \inv \rangle\lmap{\sim} \un \tar.
\end{align*}

These 2-isomorphisms represent weaker forms of the associativity,
identity and inversion axioms on groupoids, respectively.

\item[{\bf (g4)}] The 2-isomorphisms $\alpha$, $\lambda$, $\rho$, $\i_l$
and $\i_r$ satisfy the \textit{higher coherence} conditions given by
the commutativity of the following diagrams, displayed with their
corresponding labels on the left\footnote{In order to simplify our
notation, we will often write expressions of the form $(gh)l$ simply
as $gh \cdot l$; in other words, we will implicitly assume the
priority of juxtaposition over ``$\cdot$''.} :
$$ (kghl):\qquad\xymatrix{
(kg\cdot h)l& kg \cdot hl\ar[l]_-{\alpha}&
k(g\cdot hl)\ar[l]_-{\alpha}\ar[d]^-{\mr{id}\cdot \alpha}\\
(k\cdot gh)l\ar[u]^-{\alpha\cdot \mr{id}}&&k(gh\cdot
l)\ar[ll]^-{\alpha} }
$$
$$(1gh):\qquad\xymatrix{
1\cdot gh\ar[r]^-{\alpha}\ar[d]_-{\lambda}&
1g\cdot h\ar[ld]^-{\lambda\cdot \mr{id}}\\
gh& } \qquad\qquad (g1h):\qquad \xymatrix{ g\cdot
1h\ar[r]^-{\alpha}\ar[d]_-{\mr{id}\cdot \lambda} &
g1\cdot h\ar[ld]^-{\rho\cdot \mr{id}}\\
gh
}
$$
$$(gh1):\qquad\xymatrix{
g\cdot h1\ar[r]^-{\alpha}\ar[d]_-{\mr{id}\cdot \rho}&
 gh\cdot 1\ar[ld]^-{\rho}\\
gh&
}\qquad\qquad
(gg^{-1}g):\qquad\xymatrix{
 g1\ar[r]^-{\rho}&g& 1g\ar[l]_-{\lambda}\\
g(g^{-1}\cdot g)\ar[rr]^-{\alpha} \ar[u]^-{\mr{id}\cdot \i_l}&&
(g\cdot g^{-1})g\ar[u]_-{\i_r\cdot \mr{id}} }
$$
where $k,g,h,l\in \cl{G}$ are such that the compositions make sense
and the $1$'s are the appropriate identities of $\cl{G}$.
\end{itemize}

\begin{definition} \label{defGroupoidinCFG}
A \textbf{groupoid} in $\CFG_\cl{C}$ is defined by a manifold $M$,
an object  $\cl{G}$ in $\CFG_\cl{C}$, and morphisms $\sour, \tar,
\un, \inv, \mult$ and 2-isomorphisms $\alpha, \lambda,\rho,
\i_l,\i_r$ as in ${\bf (g1), (g2), (g3), (g4)}$ above.
\end{definition}

Note that the first two identities in {\bf (g2)} imply that $\sour$ and
$\tar$ are {\em epimorphisms}.

\begin{remark}\label{rm:morphismM}
Since $M$ is a manifold (so, as a category fibred in
groupoids, it is fibred in sets), any two isomorphic morphisms into
$M$ must coincide. It follows that \eqref{eq:fibredprod}, which is
in principle a 2-fibred product, is a 1-fibred product; in
particular, \eqref{eq:fibredprod} commutes in the strict sense. For
the same reason, we require equalities (rather than
isomorphisms) in the identities in {\bf (g2)}.
\end{remark}

%\comment{remark added below}

\begin{remark}\label{rm:higher}
For the higher coherences {\bf (g4)}, we selected a set of conditions that we explicitly use throughout the paper and that generates other coherences, but is not meant to be minimal (i.e., it may contain redundancies). For more on higher coherences, see \cite{Kelly,Laplaza}.
\end{remark}

A groupoid in $\CFG_\cl{C}$ will be alternatively called a {\bf
cfg-groupoid}; we use the notation $\cl{G}\rra M$, or simply
$\cl{G}$. A {\bf cfg-group} is a cfg-groupoid for which the base
manifold $M$ is a point.

\begin{remark}\label{rem:restrict} Given a cfg-groupoid $\cl{G}\rra M$ and a smooth map $\iota:
N\to M$, we consider the fibred product
$$
\xymatrix{
 \cl{G}_N
 %\slbtimes{\sour}{M}{\tar}
 \ar[r]^-{}\ar[d]_-{} &N\times N\ar[d]^-{\mathrm{(\iota,\iota)}} \\
\cl{G} \ar[r]_-{\tar,\sour} & M\times M.}
$$
One may verify that $\cl{G}_N \rra N$ is naturally a cfg-groupoid,
with operations and higher coherences inherited from those for
$\cl{G}$.
\end{remark}

As a particular instance of Remark~\ref{rem:restrict}, we consider,
for each $x\in M$, the fibred product
\begin{equation}\label{eq:isotropy}
\xymatrix{
 \cl{G}_x
 %\slbtimes{\sour}{M}{\tar}
 \ar[r]^-{}\ar[d]_-{} &\{x\}\ar[d]^-{\mathrm{diag}} \\
\cl{G} \ar[r]_-{\tar,\sour} & M\times M, }
\end{equation}
which is is a cfg-group, called the {\bf isotropy} group of $\cl{G}$
at $x$.

One can also consider $\sour$-fibres, and analogously $\tar$-fibres,
defined for each $x\in M$ as the category fibred in groupoids
resulting from the fibred product
\begin{equation}\label{eq:sfib}
\xymatrix{
 \sour^{-1}(x)
 %\slbtimes{\sour}{M}{\tar}
 \ar[r]^-{}\ar[d]_-{} &\{x\}\ar[d]^-{} \\
\cl{G} \ar[r]_-{\sour} & M. }
\end{equation}

\begin{definition}\label{DefSLieGrpd}
 A \textbf{stacky groupoid} is a groupoid $\G\rra M$ in
 $\CFG_\cl{C}$ such that $\G$ is a stack.
A \textbf{stacky Lie groupoid} is a groupoid $\G\rra M$ in
$\CFG_\cl{C}$ such that $\G$ is a differentiable stack, and source
and target (epi)morphisms $\sour, \tar$ are  submersions.

A stacky Lie groupoid $\G\rra M$ is called {\bf \'etale} if $\G$ is
an \'etale differentiable stack.
\end{definition}

We have some important classes of examples.

\begin{example}\label{ex:BA}
A particular class of stacky Lie groups is given by the (strict)
{\em 2-groups}, defined as Lie groupoids $G\rra G_0$ where both $G$
and $G_0$ are Lie groups, and so that the Lie-group multiplication
and inversion maps define morphisms of Lie groupoids,
$$
\xymatrix{
G\times G \ar[r] \ar@<0.5ex>[d]\ar@<-0.5ex>[d] &
G\ar@<0.5ex>[d]\ar@<-0.5ex>[d]\\
G_0\times G_0\ar[r] & G_0,}\qquad\qquad \xymatrix{
G \ar[r] \ar@<0.5ex>[d]\ar@<-0.5ex>[d] & G\ar@<0.5ex>[d]\ar@<-0.5ex>[d]\\
G_0\ar[r] & G_0.}
$$
In this case $BG=[G_0/G]$ inherits the structure of a stacky Lie
group:
$$
\mult: BG\times BG \to BG, \;\;\qquad \inv: BG\to BG.
$$

As an example, consider a homomorphism of abelian Lie groups
$\varphi: A\to K$, and the action of $A$ on $K$ by $k \mapsto
k+\varphi(a)$. Then the action groupoid $A\ltimes K\rra K$ defines a
2-group, where the additional group structure on $A\times K$ is the
direct product. This is a special instance of the well-known fact
that 2-groups admit an equivalent description as {\em crossed
modules} \cite{Breen94} (see also \cite{ginotstienon,NW}). In particular, taking $K=\{e\}$, we see that $BA$ is a 2-group.
\end{example}

%\comment{remark added below}

\begin{remark} As shown in \cite{trzh}, every {\em \'etale} stacky Lie
group (connected, finite dimensional) can be strictified, i.e., it
is isomorphic to a 2-group as in the previous example (arising from
a crossed module). However, this is no longer true without the
\'etaleness condition: an example is given by the {\em string Lie
2-group}, which is a (non-\'etale) stacky Lie group obtained by a
central extension \cite{schpri} of a simply connected Lie group $G$
by $BS^1$; this stacky Lie group cannot be strictified by
finite-dimensional models, see \cite{bala:2gp}. The reader can find
more on stacky Lie groups e.g. in \cite{blohmann} (see also
\cite{bala:2gp,henriques}, and \cite{wockzhu} for infinite
dimensional examples arising from central extensions).
\end{remark}

\begin{example}
Extending the previous example, one may consider (strict)
2-groupoids; these are defined as {\em double Lie groupoids}
\cite{brown,mackenzie92} of the form
$$
\xymatrix{ G \ar@<0.5ex>[r]\ar@<-0.5ex>[r]
\ar@<0.5ex>[d]\ar@<-0.5ex>[d] &
M\ar@<0.5ex>[d]^-{\mr{id}}\ar@<-0.5ex>[d]_-{\mr{id}}\\
G_0\ar@<0.5ex>[r]\ar@<-0.5ex>[r] & M,}
$$
where the vertical groupoid on the right is the trivial groupoid; in
this case, similarly to the previous example, $BG$ inherits the
structure of a stacky Lie groupoid over $M$ (note that the source
and target maps $BG\rra M$ are submersions, as a consequence of
Prop.~\ref{prop:charactsub}(b)). For a description of 2-groupoids in
terms of crossed modules, see e.g. \cite{brown,ginotstienon}.
\end{example}

\begin{example}\label{ex:integration}
\'Etale stacky Lie groupoids naturally arise as global objects
associated with {\em Lie algebroids}. Given a Lie algebroid $A\to
M$, following \cite{tz}, the path-space construction of
\cite{cf,severa} considers the Banach manifold of $A$-paths along
with the (finite-codimensional) foliation determined by
$A$-homotopies. This leads to two natural \'etale stacky Lie
groupoids associated with $A$, depending on whether one models the
leaf space of $A$-paths by $A$-homotopies using the monodromy
groupoid or the holonomy groupoid of the foliation. These stacky Lie
groupoids are denoted by
$$
\G(A)\rra M, \;\; \mbox{ and } \; \mathcal{H}(A)\rra M,
$$
respectively.

Conversely, it is also shown in \cite{tz} that an \'etale stacky Lie
groupoid $\G \rra M$ defines a local Lie groupoid $G_{loc}\rra M$,
hence a Lie algebroid over $M$.
\end{example}

\begin{remark}\label{Fact1}
Given cfg-groupoids $\cl{G}_i\rra M_i$ for $i=1,2$, there is a
natural cfg-groupoid $ \cl{G}_1\times \cl{G}_2\rra M_1\times M_2$
that we call the {\em product groupoid} of the original groupoids.
If each $\cl{G}_i$ is a stacky Lie groupoid then so is the product
groupoid.
\end{remark}

%%%%%%%%%%%%%%%%%%%%%%%%%%%%%%%%%%%%%%%%%%%%%%%%%%%%%%%%%%%%%%%%%%%
\subsection*{Additional properties of stacky Lie groupoids}

For a stacky Lie groupoid $\G\rra M$, the $\sour$-fibres
\eqref{eq:sfib} (resp. $\tar$-fibres) are differentiable stacks.
Indeed, given an atlas $\pi: G_0\to \G$, the maps
\begin{equation}\label{eq:st0}
\sour_0 = \sour\circ \pi : G_0\to M,\;\;\; \tar_0 = \tar\circ \pi :
G_0\to M,
\end{equation}
are submersions, %(see e.g. Prop.~\ref{PropWeakSub1}(ii)),
so $\sour_0^{-1}(x)$ (resp. $\tar_0^{-1}(x)$) is a smooth manifold
for all $x\in M$, which is naturally an atlas for $\sour^{-1}(x)$
(resp. $\tar^{-1}(x)$).

\begin{lemma}\label{lem:constrank}
Let  $\G\rra M$ be an \'etale stacky Lie groupoid, and let $\pi:
G_0\to \G$ be an \'etale atlas. Then for any $x\in M$, the
restriction $\tar_0|_{\sour_0^{-1}(x)}: \sour^{-1}_0(x)\to M$ is a
constant-rank map.
\end{lemma}

\begin{proof}
The unit map $\un: M\to \G$ gives rise (by a choice of local section
of the corresponding Hilsum-Skandalis bibundle) to a map
$\epsilon_0: V_x\to G_0$, where $V_x$ is an open neighborhood of
$x$, such that $\sour_0\circ \epsilon_0 = \mathrm{id}_M$ and
$\tar_0\circ \epsilon_0 = \mathrm{id}_M$; in particular,
$\epsilon_0$ is an immersion, so we can assume it is an embedding
for $V_x$ small enough.

Let $y = \epsilon_0(x) \in G_0$ and pick any $z\in \sour_0^{-1}(x)$,
so that $\sour_0(y)=\tar_0(y)=x = \sour_0(z)$. Let
$N=\sour_0^{-1}(x)$. We will show that $d\tar_0|_{T_yN}$ and
$d\tar_0|_{T_zN}$ have the same rank. The multiplication $\mult$
defines a Hilsum-Skandalis bibundle, and a choice of local section
$\sigma$ around $(z,y)\in G_0\times_{\sour_0,M,\tar_0}G_0$ leads to
a map $m_0: U_z\times_{\sour_0,M,\tar_0} U_y \to G_0$, where $U_z$
and $U_y$ are open neighborhoods of $z$ and $y$ in $G_0$,
respectively. Moreover, this local section can be chosen in such a
way that $m_0$ satisfies
\begin{equation}\label{eq:propm0}
m_0(z',\epsilon_0(s_0(z')))=z', \,\mbox{ for all }\, z'\in U_z.
\end{equation}
For that, one uses the higher morphism $\rho$ in \textbf{(g3)},
along with the same arguments as in \cite[Thm.~5.2]{tz}.
Additionally,
\begin{equation}\label{eq:m0Lg}
m_0(z,\cdot): \tar_0^{-1}(x)\cap U_y \to \tar_0^{-1}(\tar_0(z))
\end{equation}
is a local diffeomorphism around $y$. This follows from the
observation that, for any object $g$ in $\G$, left-multiplication by
$g$ gives rise to a morphism of stacks $L_g: \tar^{-1}(\sour(g)) \to
\tar^{-1}(\tar(g))$, which is an isomorphism. Hence, for $g=\pi(z)$,
$L_g$ defines a Morita bibundle $ \tar_0^{-1}(\sour_0(z))
\stackrel{a}{\leftarrow} E_{L_g} \stackrel{b}{\rightarrow}
\tar_0^{-1}(\tar_0(z))$ between the corresponding groupoid
presentations. The local section $\sigma$ induces a local section
$\sigma'$ of the moment map $a$ around $y$, and the resulting map
$b\circ \sigma'$, which is a local diffeomorphism (since the Morita
bibundle $E_{L_g}$ is \'etale), agrees with $m_0(z,\cdot)$.

Let us choose a submanifold $\Sigma \subset U_z$, with $z\in
\Sigma$, with the property that $\sour_0|_\Sigma$ and
$\tar_0|_{\Sigma}$ are open embeddings. By possibly restricting the
open neighborhoods, we can assume that
$\sour_0(U_z)=\sour_0(\Sigma)=\sour_0(U_y)=V_x$, and $\epsilon_0:
V_x\to G_0$ is an embedding. Consider the neighborhood of
$\tar_0(z)$ in $M$ given by $V_{\tar_0(z)}= \tar_0(U_z)$. Then we
have a diffeomorphism $f_\Sigma: V_x\to V_{\tar_0(z)}$ defined by
$f_\Sigma(x')= \tar_0(z')$, where $z'\in \Sigma$ is unique so that
$\sour_0(z')=x'$. We also have a map $F_\Sigma: \tar_0^{-1}(V_x)\cap
U_y \to \tar_0^{-1}(V_{\tar_0(z)})$ given by ``left
multiplication'': $F_\Sigma(y') = m_0(z',y')$, where $z'\in \Sigma$
is unique such that $\sour_0(z')=\tar_0(y')$. This map is such that
$F_\Sigma(y)= z$, and it preserves $s_0$-fibres: $\sour_0\circ
F_\Sigma = \sour_0$. Note also that $f_\Sigma \circ \tar_0 =
\tar_0\circ F_\Sigma$, so in order to conclude that the ranks of
$d\tar_0|_{T_yN}$ and $d\tar_0|_{T_zN}$ agree, it suffices to show
that $F_\Sigma$ is a local diffeomorphism around $y$. The tangent
spaces at $y$ (resp. $z$) to the submanifolds $\epsilon_0(V_x)$ and
$\tar_0^{-1}(x)$ (resp. $\Sigma$ and $\tar_0^{-1}(\tar_0(z))$) are
such that their direct sum is $T_yG_0$ (resp. $T_zG_0$), so the
result follows from the observation that, around $y$, both
$F_\Sigma|_{\epsilon_0(V_x)}: \epsilon_0(V_x)\to \Sigma$ and
$F_\Sigma|_{\tar_0^{-1}(x)\cap U_y} = m_0(z,\cdot):
\tar_0^{-1}(x)\cap U_y \to \tar_0^{-1}(\tar_0(z))$ are local
diffeomorphisms (cf. \eqref{eq:propm0} and \eqref{eq:m0Lg}).

\end{proof}

\begin{corollary}\label{cor:onto}
For  $\G\rra M$ a stacky Lie groupoid with \'etale atlas $\pi:
G_0\to \G$, if the restriction $\tar_0|_{\sour_0^{-1}(x)}:
\sour^{-1}_0(x)\to M$ is onto, then it is a submersion.
\end{corollary}

We will refer to a stacky Lie groupoid admitting an \'etale atlas as
in Corollary~\ref{cor:onto} as {\bf transitive}.

\begin{remark}
More generally, the map in the previous corollary is a submersion
onto leaves of the underlying Lie algebroid.
\end{remark}

We can now show two key properties of \'etale stacky Lie groupoids.

\begin{proposition}\label{prop:etstackyprops}
Let $\G\rra M$ be an \'etale stacky Lie groupoid. The following
holds:
\begin{enumerate}
\item[(i)] The unit map $\un:M \to \cl{G}$ is an {\em immersion}.
\item[(ii)] $\G_x$ is an (\'etale) stacky Lie group for all $x\in M$,

\end{enumerate}
\end{proposition}

\begin{proof}

For $(i)$, it follows directly from Def.~\ref{def:weak} that it
suffices to show that, for an atlas $\pi: G_0\to \G$, the induced
map $f: M\times_{\un,\G,\pi} G_0 \to G_0$ is an immersion. Consider
the induced map $\pi_M: M\times_{\un,\G,\pi} G_0\to M$, which is a
surjective submersion. Taking any local section $\sigma$ of $\pi_M$
and using that $\sour\circ \un = \mathrm{id}_M$ and $\pi\circ f
\cong \un\circ \pi_M$, it follows that
\begin{equation}\label{eq:etale}
\sour\circ \pi (f(\sigma(x))) = x,
\end{equation}
for all $x$ in the domain of $\sigma$. The condition that $\G$ is
\'etale implies that we can take $\pi$ to be \'etale, hence $\pi_M$
is \'etale, and $\sigma$ is a local diffeomorphism. It then follows
from \eqref{eq:etale} that $f$ is an immersion.

For $(ii)$, note that by considering an \'etale atlas $\pi: G_0\to
\G$ in \eqref{eq:isotropy}, the natural map
$\sour_0^{-1}(x)\cap\tar_0^{-1}(x) \to \G_x$ defines an atlas as
long as $\sour_0^{-1}(x)\cap\tar_0^{-1}(x)$ is a manifold, which is
a direct consequence of Lemma~\ref{lem:constrank}.

\end{proof}

%%%%%%%%%%%%%%%%%%%%%%%%%%%%%%%%%%%%%%%%%%%%%%%%%%%%%

\subsection{Actions}\label{subsec:act}

Consider $\G\rra M$, a groupoid in $\CFG_\cl{C}$, as in
Definition~\ref{defGroupoidinCFG}, and let $\X \in
\mathrm{Obj}(\CFG_\cl{C})$. The definition of a (right) action of
$\G$ on $\X$ requires the following data and conditions:

\begin{itemize}
\item[{\bf (a1)}] Morphisms $\ma$ (the {\it moment map}) and $\act$ (the {\it action map}),
\begin{align*}
& \cl{X}\lmap{\ma} M, \\
& \cl{X} \times_{\ma,M,\tar}
%\slbtimes{\ma}{M}{\tar}
\cl{G}\lmap{\act} \cl{X},\;\;\;
\act(x,g)=x\cdot g=xg,
\end{align*}
where, in the definition of $\act$, we use the fibred product of
$\ma:\X\to M$ and $\tar:\G\to M$,
$$
\xymatrix{
 \cl{X} \times_{\ma,M,\tar}
 %\slbtimes{\ma}{M}{\tar}
 \cl{G}\ar[r]^-{n_\ma}\ar[d]_-{} &\cl{G}\ar[d]^-{\tar} \\
\cl{X} \ar[r]_-{\ma} & M. }
$$

\item[{\bf (a2)}] The morphisms $\ma$ and $\act$ are assumed to satisfy the following
identity:
$$
\ma\act=\sour n_\ma.
$$

We also require additional action axioms in weak form. Let us
consider the morphisms
\begin{itemize}
\item[$\bullet$] $\act (\mr{id}\times \mult): \cl{X}\slbtimes{\ma}{M}{\tar}\cl{G}\slbtimes{\sour}{M}{\tar}\cl{G}
\map \cl{X}$ and $\act(\act\times \mr{id}):
\cl{X}\slbtimes{\ma}{M}{\tar}\cl{G}\slbtimes{\sour}{M}{\tar}\cl{G}
\map \cl{X}$, encoding the two ways in which one can act on $\X$ by
two elements of $\G$,

\item[$\bullet$] $\act\langle \mr{id},\un \ma \rangle: \cl{X}\map \cl{X}$,
encoding multiplication by the identity.
\end{itemize}

\item[{\bf (a3)}] Two 2-isomorphisms
\begin{align*}
& \beta: \act(\mr{id}\times \mult)\lmap{\sim} \act(\act\times \mr{id}), \\
& \varepsilon: \act \langle \mr{id},\un \ma \rangle\lmap{\sim}
\mr{id}.
\end{align*}

These 2-isomorphisms represent weaker versions of the associativity
and multiplication by identity axioms for groupoid actions.

\item[{\bf (a4)}] The 2-isomorphisms $\beta$ and $\varepsilon$ satisfy the
\textit{higher coherence} conditions given by the commutativity of
the following diagrams, displayed with their corresponding labels on
the left (we follow the notation explained in {\bf (g4)}):
$$
(xghl):\qquad \xymatrix{ x(g\cdot
hl)\ar[r]^-{\beta}\ar[d]_-{\mr{id}\cdot \alpha}
&xg\cdot hl\ar[r]^-{\beta}& (xg\cdot h)l\\
x(gh\cdot l)\ar[rr]_-{\beta} &&(x\cdot gh)l\ar[u]^-{\beta\cdot
\mr{id}} }
$$
$$(x1g):\qquad\xymatrix{
x\cdot 1g\ar[r]^-{\beta}\ar[d]_-{\mr{id}\cdot \lambda} & x1\cdot
g\ar[ld]^-{\varepsilon\cdot \mr{id}}
\\xg&}
\qquad \qquad(xg1):\qquad \xymatrix{ x\cdot g1
\ar[r]^-{\beta}\ar[d]_-{\mr{id}\cdot \rho}
&xg\cdot 1\ar[ld]^-{\varepsilon}\\
 xg&
}
$$
where $x\in\cl{X}$ and $g,h,l\in\cl{G}$ are such that the
compositions in the diagrams make sense
and the $1$'s are the appropriate identities of $\cl{G}$.
\end{itemize}

\begin{definition} \label{defaction}
Let $\cl{G}\rra M$ be a groupoid in $\CFG_\cl{C}$ and $\cl{X} \in
\mathrm{Obj}(\CFG_{\cl{C}})$. A \textbf{right action} of $\cl{G}$ on
$\cl{X}$ is defined by morphisms $\ma$, $\act$ and 2-isomorphisms
$\beta$, $\varepsilon$ as in ${\bf (a1)}$, ${\bf (a2)}$, ${\bf
(a3)}$, ${\bf (a4)}$.
\end{definition}

%\comment{first sentence added below}

Regarding {\bf (a4)}, a similar observation to Remark~\ref{rm:higher} applies.
We will say that $\cl{G}$ {\bf acts on $\X$ along $\ma$}, or
alternatively that $\cl{G}$ acts on $\ma: \cl{X}\to M$, to make the
moment map explicit.  In this paper, we will be mostly interested in
actions of stacky Lie groupoids on differentiable stacks.

Given an action, we define its associated \textbf{action-projection map} as
\begin{equation}\label{eq:Delta}
\Delta=(\mathrm{pr}_1,\act): \cl{X} \times_{\ma,M,\tar} \cl{G}\map
\cl{X}\times\cl{X}.
\end{equation}
We will often denote this map by indicating how the functor acts on
objects: $(x,g)\mapsto (x,xg)$.

\begin{remark}[Left actions]
A left action of $\G \rra M$ on $\ma: \X\to M$ is defined
analogously: the action morphism in {\bf (a1)} is replaced by
$$
\cl{G} \times_{\sour,M,\ma}
\cl{X}\map \cl{X},
$$
in such a way that $\ma(gx)=\tar(g)$ (cf. {\bf (a2)}), and there
are 2-isomorphisms (cf. {\bf (a3)})
$$
(hg)\cdot x \lmap{\beta} h\cdot (gx), \qquad 1x\lmap{\varepsilon} x
$$
satisfying the following higher coherence conditions (cf. {\bf
(a4)}):
$$
(lhgx):\qquad \xymatrix{ (lh\cdot g)x\ar[r]^-{\beta}
&lh \cdot gx\ar[r]^-{\beta}& l(h\cdot gx)\\
(l\cdot hg)x\ar[rr]_-{\beta}\ar[u]^-{\alpha\cdot\mr{id}} &&l(hg\cdot
x)\ar[u]_-{\mr{id}\cdot \beta} }
$$
$$(g1x):\qquad\xymatrix{
g1\cdot x\ar[r]^-{\beta}\ar[d]_-{\rho\cdot\mr{id} } &g\cdot
1x\ar[ld]^-{ \mr{id}\cdot\varepsilon}
\\gx&}
\qquad \qquad(1gx):\qquad \xymatrix{ 1g\cdot
x\ar[r]^-{\beta}\ar[d]_-{\lambda\cdot\mr{id}}
&1\cdot gx\ar[ld]^-{\varepsilon}\\
gx&
}
$$
Similarly to \eqref{eq:Delta}, we have an action-projection map
$$
\Delta: \cl{G} \times_{\sour,M,\ma}
\cl{X}\map \cl{X}\times\cl{X}, \qquad\;\;\; (g,x)\mapsto (gx,x).
$$
\end{remark}

When $M$ is a point, similar notions of action (with varying levels
of strictness) have been considered e.g. in
\cite{bartels,breen,ginotnoohi,romagny,wockel}.
Categorified actions similar to Def.~\ref{defaction} are also
studied in \cite[Sec.~5]{duli} using simplicial methods.
%\comment{last sentence added}

We present here some examples, others will be discussed in
Section~\ref{subsec:stackprincipal}.

\begin{example}\label{LRMult}
A cfg-groupoid  $\cl{G}\rra M$ acts on itself on the right and on
the left by multiplication, with moment maps $\sour:\cl{G}\ra M$ and
$\tar:\cl{G}\ra M$ respectively. The associativity and identity
2-isomorphisms of these actions are induced by the corresponding
2-isomorphisms of the groupoid.
\end{example}

\begin{example}\label{ex:sfib}
Given an action of $\cl{G}$ on $\cl{X}$ along $\ma: \cl{X}\to M$ and
a smooth map $\iota: N\to M$, let $\cl{X}_N= N\times_M \cl{X}$. Then
there is natural action of $\cl{G}_N$ (see
Remark~\ref{rem:restrict}) on $\cl{X}_N$.

As a particular case, given a cfg-groupoid $\cl{G}\rra M$, let
$\cl{G}_x$ be the cfg-group defined by its isotropy at $x\in M$, as
in \eqref{eq:isotropy}. Then the groupoid multiplication restricts
to an action (on the right) of $\cl{G}_x$ on $\sour^{-1}(x)$.
\end{example}

\begin{example}\label{ActionXX}
Let the cfg-groupoid $\cl{G}\rra M$ act on the right on the
categories fibred in groupoids $\cl{X}_1$ e $\cl{X}_2$, along the
morphisms $\ma_1: \X_1\to M$ and $\ma_2: \X_2\to M$. The
2-isomorphisms associated with these actions are denoted by
$\beta_1$, $\varepsilon_1$, and $\beta_2$, $\varepsilon_2$,
respectively. Then there is an induced action over
$\cl{X}_1\times_M\cl{X}_2$ given by $ (x_1,x_2)g=(x_1g,x_2g).$ The
associativity 2-isomorphism is defined by
$$
(x_1,x_2)\cdot gh =(x_1\cdot gh,x_2\cdot
gh)\quad\lmap{\beta_1\times\beta_2}\quad (x_1g\cdot h, x_2g\cdot
h)=(x_1,x_2)g \cdot h,
$$
and the identity 2-isomorphism is defined by
$$
(x_1,x_2)1 =(x_11,x_21)\quad\lmap{\varepsilon_1\times\varepsilon_2}\quad
(x_1, x_2).
$$
\end{example}

%%%%%%%%%%%%%%%%%%%%%%%%%%%%%%%%%%%%%%%%%%%%%%%%%%%%%%%%%%%%%%%%%%%%%%%%%%
\subsection{Equivariant morphisms}\label{subsec:equiv}

For objects in $\CFG_{\cl{C}}$ endowed with actions, we consider a
natural notion of equivariant morphism.

\begin{definition}\label{DefGequivariant}
Let $\G\rra M$ be a cfg-groupoid acting  (on the right) on $\cl{X}_i
\in Obj(\CFG_{\cl{C}})$ (with action map $\act_i$), along the map
$\ma_i: \X_i\to M$, $i=1,2$. A morphism $F:\cl{X}_1 \ra\cl{X}_2$ is
\textbf{$\cl{G}$-equivariant} if $\ma_2 F = \ma_1$ and there is a
given 2-isomorphism $\delta: \act_2 \,\scirc (F\times
\mr{id})\lmap{\sim} F\,\scirc \,\act_1$ that makes the square
$$
\xymatrix{\cl{X}_1\times_M\cl{G}\ar[r]^-{\act_1}\ar[d]_-{F\times\mr{id}}
& \cl{X}_1\ar[d]^-{F}\\
\cl{X}_2\times_M\cl{G}\ar[r]_-{\act_2}& \cl{X}_2 }
$$
2-commute, and that satisfies  the higher coherence conditions
expressed by the commutativity of the following diagrams (with the
corresponding labels displayed on the left):
$$
(\delta\beta_1\beta_2):\qquad\qquad \xymatrix{ F(x)\cdot g_1g_2
\ar[r]^-{\beta_2}\ar[d]_-{\delta}& F(x)g_1\cdot
g_2\ar[r]^-{\delta\cdot\mr{id}}&
F(xg_1)g_2\ar[d]^-{\delta}\\
F(x\cdot g_1g_2)\ar[rr]_-{F(\beta_1)}&& F(xg_1\cdot g_2) }
$$
$$(\delta\varepsilon_1\varepsilon_2):\qquad \xymatrix{
 F(x)\cdot 1\ar[r]^-{\varepsilon_2}\ar[d]_-{\delta}
&F(x)\\
F(x\cdot 1)\ar[ur]_-{F(\varepsilon_1)}&}
$$
where $x\in\cl{X}_1$, $g_1,g_2\in\cl{G}$ are such that the
compositions make sense, $1$ is the appropriate groupoid identity,
$\beta_1,\beta_2$ are the associativity 2-isomorphisms of the
actions, and $\varepsilon_1, \varepsilon_2$ are the identity
2-isomorphisms of the actions (see Def.~\ref{defaction}). We will
often refer to $\delta$ as the {\em equivariance 2-isomorphism}.
\end{definition}

Clearly, the identity morphism is always $\cl{G}$-equivariant with
$\delta=\mr{id}_{\act}$. The following property is also naturally
expected.

\begin{lemma}\label{CompGequiv}
The composition of $\cl{G}$-equivariant morphisms is
$\cl{G}$-equivariant.
\end{lemma}

\begin{proof}
Given two equivariant morphims
$\cl{X}_1\lmap{F_1}\cl{X}_2\lmap{F_2}\cl{X}_3$, with corresponding
2-isomorphisms $\delta_1$ and $\delta_2$ (as in
Def.~\ref{DefGequivariant}), the composition $F_2F_1$ is equivariant
with respect to
$$
\delta=(\mr{id}_{F_2}\,\scirc\,\delta_1)* (\delta_2\,\scirc\,
\mr{id}_{F_1\times\mr{id}}).
$$
More explicitly, for any $(x,g)\in \cl{X}\times_M\cl{G}$, we have
$$
\delta: F_2F_1(x)\cdot g\lmap{\delta_2}F_2(F_1(x)\cdot g)
\lmap{F_2(\delta_1)} F_2F_1(x\cdot g).
$$
Higher coherences $(\delta\beta_1\beta_2)$ and $(\delta\varepsilon_1
\varepsilon_2)$ for $\delta$ follow from the  respective coherences
for $\delta_1$ and $\delta_2$.
\end{proof}

%%%%%%%%%%%%%%%%%%%%%%%%%%%%%%%%%%%%%%%%%%%%%%%%%%%%%%%%%%%%%%%%%%%%%
\subsection{Stacky principal bundles}\label{subsec:stackprincipal}

Let $\G\rra M$ be a cfg-groupoid acting (on the right) on $\cl{X}
\in \Obj(\CFG_{\cl{C}})$ along $\ma: \X \to M$. Let $\cl{S} \in
\Obj(\CFG_{\cl{C}})$ and $\Proj:\cl{X}\ra \cl{S}$ be a morphism. We
consider a natural compatibility between the action of $\G$ on $\X$
and the morphism $\Proj$ (see Section~\ref{subsec:LG} for a similar
notion in the context of Lie groupoids).

\begin{definition}\label{DefActOnFibers}
We say that $\cl{G}$ \textbf{acts on the fibers of} $\Proj:\X \to
\cl{S}$ if there is a given 2-isomorphism $\gamma: \Proj\,\scirc\,
\rm{pr}\ra \Proj\,\scirc\, \act$ that makes the square
\begin{equation}\label{eq:diagr}
\xymatrix{
\X\times_M\G\ar[r]^-{\act}\ar[d]_-{\rm{pr}}& \X\ar[d]^-{\Proj}\\
\X\ar[r]_-{\Proj}& \cl{S}}
\end{equation}
2-commute and that satisfies the higher coherence conditions
expressed by the commutativity of the following diagrams (with the
corresponding labels displayed on the left):
$$
(\gamma\beta):\quad\xymatrix{
\Proj(x)\ar[r]^-{\gamma}\ar[d]_-{\gamma}&\Proj(xg)\ar[d]^-{\gamma}\\
\Proj(x(gh))\ar[r]_-{\Proj(\beta)}&\Proj((xg)h),} \qquad
\qquad(\gamma\varepsilon):\quad\xymatrix{
\Proj(x)\ar[rd]^-{\mr{id}}\ar[d]_-{\gamma}&\\
\Proj(x\cdot 1)\ar[r]_-{\Proj(\varepsilon)}&\Proj(x), }
$$
where $x\in \cl{X}$ and $g,h\in\cl{G}$ are such that the
compositions make sense, $1\in\cl{G}$ denotes the appropriate
identity, $\beta$ is the associativity 2-isomorphism of the action,
and $\varepsilon$ is the identity 2-isomorphism of the action (see
${\bf (a3)}$). We will say that $\G$ acts on the fibers of $\Proj$
\emph{via} $\gamma$.
\end{definition}

A similar definition holds for left actions.

Note that, for any cfg-groupoid $\cl{G}\rra M$, the action by right
multiplication is on the fibers of $\tar:\cl{G}\ra M$, and the
action by left multiplication is on the fibers of $\sour:\cl{G}\ra
M$.

If $\cl{G}$ acts on the fibers of $\Proj : \cl{X}\to \cl{S}$, we may
consider the morphism
\begin{equation}\label{eq:actfibre}
\cl{X}\times_M \cl{G}\ra \cl{X}\times_{\cl{S}}\X,\qquad (x,g)\mapsto
(x, \gamma(x,g) ,xg).
\end{equation}

\newpage

\begin{remark} \label{RemActFibXG}\
\begin{itemize}
\item[(a)] If $\cl{S}=S$ is a manifold, then the condition for an action to
be on the fibers of $\Proj:\cl{X}\ra S$  is just the commutativity
of the square \eqref{eq:diagr}, as in this case the 2-isomorphism
$\gamma$ is necessarily equal to the identity and the higher
coherences $(\gamma\beta)$, $(\gamma\varepsilon)$ are automatically
satisfied (cf. Remark~\ref{rm:morphismM}).

\item[(b)]
If $\cl{G}$ acts on the fibers of $\Proj$ via $\gamma$ and
$\Proj':\cl{S}\ra\cl{S}'$ is a morphism, then $\cl{G}$ acts on the
fibers of $\Proj'\Proj$ via $\mr{id}_{\Proj'}\,\scirc\,\gamma$.

\end{itemize}

\end{remark}

We will be mostly interested in the context of differentiable
stacks. Let $\G \rra M$ be a stacky Lie groupoid acting (on the
right) on a differentiable stack $\cl{X}$ along the map $\ma:\X\to
M$.

\begin{definition}\label{DefPrincBundle}
We say that a morphism $\Proj:\cl{X}\ra \cl{S}$ defines a (right)
\textbf{principal $\G$-bundle} $\X$ over a differentiable stack
$\cl{S}$ if the following conditions are satisfied:
\begin{enumerate}
\item $\Proj:\cl{X}\ra \cl{S}$ is an epimorphism and a
submersion,
\item $\cl{G}$ acts on the fibers of $\Proj$,
\item The morphism $ \cl{X}\times_M \cl{G}\map \cl{X}\times_{\cl{S}}\X
$ (see \eqref{eq:actfibre}) is an isomorphism.
\end{enumerate}
Principal \textbf{left} $\cl{G}$-bundles are defined similarly.
\end{definition}

Note that if the differentiable stacks $\G$, $\X$ and $\cl{S}$ are representable,
the previous definition recovers the notion of principal bundle in the smooth category.

%\comment{added remark below (other viewpoints to principal bundles)}

\begin{remark}
Other viewpoints to smooth principal bundles naturally extend to higher settings.
For example, higher principal bundles are often studied through methods of homotopy theory  \cite{lurie, nss:1, fss}, generalizing the fact that smooth principal bundles can be defined (and classified) by maps into classifying spaces $BG$ of Lie group(oid)s by means of (homotopical) pullbacks. For 2-groups, descriptions of principal bundles in terms of local trivializations and cocycles can be found e.g. in \cite{basch,bartels,wockel} (\cite{NW} reconciles these viewpoints with Def.~\ref{DefPrincBundle} above).
\end{remark}

\begin{example}
The target morphism $\tar:\cl{G}\ra M$ of a stacky Lie groupoid is a
principal right $\cl{G}$-bundle under the action of multiplication
on the right.  The same is true for $\sour:\cl{G}\ra M$ and left
multiplication.
\end{example}

\begin{example} \label{ex:princBG}
Consider a (right) action of a Lie groupoid $G\rra G_0$ on a
manifold $X$ along $a: X\to G_0$. Then one may verify that $X$ is a
principal $G$-bundle over the quotient stack $[X/G]$.

In particular, any Lie groupoid $G\rra G_0$ is such that $G_0$ is a
principal $G$-bundle over $BG$.
\end{example}

\begin{example}\label{ex:HSprincbundle}
Consider Lie groupoids $X\rra X_0$ and $Y\rra Y_0$, and a right
principal $X$-$Y$ - bibundle $P$, i.e., a Hilsum-Skandalis bibundle
presenting a morphism $BX\to BY$. Note that the surjective
submersion $P\to X_0$ naturally gives rise to a morphism of quotient
stacks
$$
[X\bs P] \to [X\bs X_0]=BX
$$
which is a submersion and an epimorphism. One can check that $Y$
acts on $[X\bs P]$ on the fibres of this map, and makes $[X\bs P]$
into a principal $Y$-bundle over $BX$.

\end{example}

\begin{example}\label{ex:gerbe2}
When $\G$ is a 2-group as in Example~\ref{ex:BA}, principal
$\G$-bundles (with strict actions as in \cite[Sec.~6]{NW}) arise in
the description of (non-abelian) {\em gerbes}, see
\cite[Sect.~7]{NW} (cf. \cite{breen,Breen94}, see also \cite{ginotstienon}).
A simple example is the correspondence between $S^1$-gerbes and principal $BS^1$-bundles over a manifold $M$ (or, more generally, over a differentiable stack). Recall that one way to model $S^1$-gerbes is via {\em $S^1$-central extensions} (see e.g. \cite{bx}),
\[
\xymatrix{S^1\times X_0 \ar[r]^{\iota} \ar@<0.5ex>[d]\ar@<-0.5ex>[d]&
X \ar[r]
  \ar@<0.5ex>[d]\ar@<-0.5ex>[d]& Y \ar@<0.5ex>[d]\ar@<-0.5ex>[d] \\
X_0 \ar[r] & X_0 \ar[r] & X_0.}
\]
In such a case, the multiplication of $X\rightrightarrows X_0$ and
the embedding $\iota$ give us a map
\[ \xymatrix{ X\times S^1 \ar@<0.5ex>[d]\ar@<-0.5ex>[d] \ar[r] &X
  \ar@<0.5ex>[d]\ar@<-0.5ex>[d] \\
X_0 \times pt \ar[r] & X_0   }
\]
which defines an action of $BS^1$ on $BX=[X_0/X]$. The fact that $X$ is an
$S^1$-principal bundle over $Y$ allows one to verify that this $BS^1$-action makes
$BX$ into a principal $BS^1$-bundle over $BY=[X_0/Y]$, as in Def.~\ref{DefPrincBundle}.
This picture in fact extends to non-abelian gerbes \cite{slx}: in this case, a $G$-gerbe over a manifold $M$
is a principal $\G$-bundle over $M$, where $\G$ is the 2-group defined by the crossed module $G\to \mathrm{Aut}(G)$.
\end{example}

To construct more examples of principal bundles, let us consider a
cfg-groupoid $\G\rra M$ acting on $\X \in \Obj(\mathrm{CFG}_\cl{C})$
on the fibers of $\Proj: \X\to \cl{S}$, where $\cl{S}$ is in
$\Obj(\mathrm{CFG}_\cl{C})$. The next two propositions can be
directly verified.

\begin{proposition}\label{prop:princ1}
Consider a map of categories fibred in groupoids $\mathcal{S}'\to
\mathcal{S}$ and the fibred product $\X'= \X
\times_{\mathcal{S}}\mathcal{S}'$. The following holds:
\begin{enumerate}
\item There is an induced action of $\G$ on $\X'$ on the fibres
of the natural projection $\Proj': \X'\to \mathcal{S}'$ (the
``action on the first coordinate'').
\item If $\Proj$ is an epimorphism, so is $\Proj'$.
\item If the map $\G\times_M \X\to
\X\times_{\mathcal{S}}\X$  \eqref{eq:actfibre} is an isomorphism,
then so is the corresponding map $\G\times_M \X'\to
\X'\times_{\mathcal{S}'}\X'$.
\item If $\G$, $\X$, $\mathcal{S}$ and $\mathcal{S}'$ are
differentiable stacks and $\Proj$ is a  submersion, then so is
$\Proj'$.
\end{enumerate}
It follows that if $\Proj:\X\to \mathcal{S}$ is a principal
$\G$-bundle, then so is $\Proj':\X'\to \mathcal{S}'$.
\end{proposition}

\begin{example}\label{ex:restric1}
For a stacky Lie groupoid $\G\rra M$, it follows from the previous
proposition that left multiplication restricts to a principal
$\G$-bundle $ \sour^{-1}(x)\to \{x\}$, for each $x\in M$.
\end{example}

\begin{example}\label{ex:universalPB}
Consider  a morphism $F: \X \to \cl{Y}$ of differentiable stacks,
 and assume that $\cl{Y}$ is presented by $Y\rra Y_0$. Following
Example~\ref{ex:princBG}, we know that $Y_0$ is principal $Y$-bundle
over $\cl{Y}$, so it follows from the previous proposition that $F$
induces a principal $Y$-bundle over $\X$.
\end{example}

\begin{proposition}\label{prop:princ2}
Consider a smooth map $f: N\to M$, let $\G_N$ and $\cl{X}_N$ be as
in Example~\ref{ex:sfib}. The following holds:
\begin{enumerate}
\item The action of $\G_N$ on $\X_N$ is on the fibres of the natural map $\Proj_N: \X_N\to
\mathcal{S}$.
\item  If the map $\G\times_M \X\to
\X\times_{\mathcal{S}}\X$ \eqref{eq:actfibre} is an isomorphism,
then so is the corresponding map $\G_N \times_N \X_N\to \X_N
\times_{\mathcal{S}}\X_N$.
\item If $\X\to \mathcal{S}$ is a principal $\G$-bundle, $\G_N$
is a differentiable stack and $\Proj_N$ is a  submersion and an
epimorphism, then $\Proj_N:\X_N\to \mathcal{S}$ is a principal
$\G_N$-bundle.
\end{enumerate}
\end{proposition}

\begin{example}\label{ex:restric2}
For an \'etale stacky Lie groupoid $\G\rra M$, fix $x\in M$ and
consider the stacky Lie group $\G_x$. Suppose, for simplicity, that
$\G$ is transitive (see Cor.~\ref{cor:onto}). It follows from the
previous proposition that right multiplication on $\G$ restricts to
a principal $\G_x$-bundle $\tar: \sour^{-1}(x)\to M$.

\end{example}

Our last result in this section gives the expected relation between
Examples~\ref{ex:HSprincbundle} and \ref{ex:universalPB} above.

\begin{proposition}\label{prop2BXfibprodY0}
Let $\cl{X}$ and $\cl{Y}$ be differentiable stacks, let $X\rra X_0$
and $Y\rra Y_0$ be Lie groupoids presenting them, and let $F:\X\ra
\cl{Y}$ be a morphism. Let $P$ be a right principal $X$-$Y$-bibundle
presenting $F$. Then $[X\bs P]$ fits into a canonical 2-cartesian
diagram
$$\xymatrix{
[X\bs P] \ar[r]\ar[d] &Y_0\ar[d]\\
\X\ar[r]_-{F}&\cl{Y}, }
$$
in such a way that the resulting isomorphism $[X\bs P] \cong
\X\times_{\cl{Y}}Y_0$ is an identification of $Y$-principal bundles
over $\X$.
\end{proposition}

\begin{proof}
We start by replacing $X\rra X_0$ with a Morita equivalent Lie
groupoid, denoted by $X\ltimes P\rtimes Y \rra P$, whose arrows are
triples  $(x,z,y)$ in $X\times P \times Y$ such that $s(x)=a_0(z)$
and $t(y)=b_0(z)$; source and target maps are
$$
s(x,z,y)=z,\qquad t(x,z,y)=xzy,
$$
and multiplication is defined by
$$
(x,z,y)(\bar{x},\bar{z},\bar{y})= (x\bar{x},\bar{z},\bar{y}y),
$$
where $z=\bar{x}\bar{z}\bar{y}$ is the composability condition. A
Morita $(X\ltimes P\rtimes Y)$-$X$ bibundle is given by
$P\times_{X_0}X$, with left action $(x,z,y)\cdot(z,\bar{x})=(x z y,
x\bar{x})$ and right action $(z,\bar{x}) x = (z,\bar{x}x)$.
Moreover, the composition of this Morita bibundle with $P$ gives
rise to a right principal $(X\ltimes P\rtimes Y)$-$Y$ bibundle,
which in fact corresponds to a morphism of groupoids,
$$
\xymatrix{ X\ltimes P\rtimes Y \ar[r]^-{b}
\ar@<0.5ex>[d]\ar@<-0.5ex>[d]&
Y\ar@<0.5ex>[d]\ar@<-0.5ex>[d]\\
P\ar[r]_-{b_0}&Y_0,}
$$
explicitly given by $b(x,z,y)=y^{-1}$. In other words, this map $b$
is a presentation for $F:\X\to \cl{Y}$ upon the identification
$B(X\ltimes P\rtimes Y)\cong  BX = \cl{X}$.

We will now compute the (weak) fibred product of groupoids
\begin{equation}\label{eq:W}
\xymatrix{
W \ar[r]\ar[d] &Y_0\ar[d]\\
X\ltimes P\rtimes Y \ar[r]_-{b}&Y ,}
\end{equation}
where $Y_0$ denotes the trivial groupoid $Y_0\rra Y_0$, and prove
that there is an identification $BW=B(X\ltimes P)$. Since
$B(X\ltimes P)=[X\bs P]$, this leads to an isomorphism $[X\bs
P]\cong \X\times_{\cl{Y}} Y_0$ (by Prop.~\ref{prop:Bprod}).

The fibred product $W$ in \eqref{eq:W} is described by objects
$(z,y,y_0)\in P\times Y\times Y_0$ such that $s(y)=b_0(z)$ and
$t(y)=y_0$, and arrows from $(z,y,y_0)$ to
$(\bar{z},\bar{y},\bar{y}_0)$ given by pairs $(x,y_1) \in X\times Y$
such that
$$
s(x)=a_0(z),\quad t(y_1)=b_0(z),\quad xzy_1=\bar{z},\quad
\bar{y}=yy_1,
$$
assuming that $\bar{y}_0=y_0$ (otherwise there are no arrows). The
multiplication between $(\bar{x},\bar{y}_1):
(z,y,y_0)\ra(\bar{z},\bar{y},y_0)$ and
$(x,y_1):(\bar{z},\bar{y},y_0)\ra (\tilde{z},\tilde{y},y_0)$ is
defined componentwise:
$$
(x,y_1)(\bar{x},\bar{y}_1)=(x\bar{x},\bar{y}_1y_1).
$$

We note that $W$ admits an alternative description: objects are
pairs $(z,y)$ such that $s(y)=b_0(z)$, arrows are given by
$(x,z,y,y_1)$ such that $s(x)=a_0(z)$, $t(y_1)=b_0(z)=s(y)$, and
source and target maps are:
$$
s(x,z,y,y_1)=(z,y),\qquad t(x,z,y,y_1)=(xzy_1,yy_1).
$$
For multiplication, we have
$$
(x,z,y,y_1)(\bar{x},\bar{z},\bar{y},\bar{y}_1)=
(x\bar{x},\bar{z},\bar{y},\bar{y}_1y_1),
$$
where $(\bar{x}\bar{z}\bar{y}_1,\bar{y}\bar{y}_1)=(z,y)$ is the
condition for composability.

Using this alternative description of $W$, we define a groupoid
morphism
\begin{equation}\label{eq:c}
c:W\map X\ltimes P
\end{equation}
as follows: on objects, we let
$$
c_0(z,y)=zy^{-1},
$$
and on arrows we set
$$
c(x,z,y, y_1)=(x,zy^{-1}).
$$
We leave to the reader to check that $c_0$ is a surjective
submersion (which follows from $P$ being right principal), and that
$c$, as groupoid morphism, is fully faithful. This shows that $c$ is
a weak equivalence, hence induces the desired identification.

To see that the identification
\begin{equation}\label{eq:inducedc}
BW \cong \X\times_{\cl{Y}}Y_0\stackrel{\sim}{\to} [X\bs P]
\end{equation}
induced by $c$ is $Y$-equivariant, we first notice that the
$Y$-action on $BW \cong \X\times_{\cl{Y}}Y_0$ is presented by the
groupoid morphism $W\times_{Y_0} Y \to W$ given on objects and
arrows by
$$
((z,y), \bar{y})\mapsto (z,\bar{y}^{-1}y),\;\;\;
((x,z,y,y_1),\bar{y})\mapsto (x,z,\bar{y}^{-1}y, y_1),
$$
respectively. Similarly, the $Y$-action on $[X\bs P]$ is presented
by the groupoid morphism $(X\ltimes P) \times_{Y_0} Y \to (X\ltimes
P)$ where $Y$ acts on the $P$ factor only. (Here  $W\times_{Y_0} Y$
and $(X\ltimes P) \times_{Y_0} Y$ denote the fibred product of
groupoids viewing $Y$ as the trivial groupoid $Y\rra Y$). One may
directly check that $c$ \eqref{eq:c} is $Y$-equivariant, which
implies the  same property for \eqref{eq:inducedc}.
\end{proof}

%%%%%%%%%%%%%%%%%%%%%%%%%%%%%%%%%%%%%%%%%%%%%%%%%%%%%%%%%%%%%%%%%%%%%%%%%%%%%%%%%%%%%%%%%%%%%
\section{Quotients}\label{sec:Quotients}

%\subsection{Quotients by groupoid actions}\label{subsecQuotient}

Let a cfg-groupoid $\G\rra M$ act on $\X \in \Obj(\CFG_{\cl{C}})$.
Our goal in this section is to make appropriate sense of the
quotient of $\X$ by the $\G$-action. We start by defining a natural
category fibred in groupoids $\X\pq\G$ associated with the action,
referred to as the \textit{prequotient}; its stackification will
then be defined as the {\it quotient} stack $\X/\G$.

%%%%%%%%%%%%%%%%%%%%%%%%%%%%%%%%%%%%%%%%%%%%%%%%%%%%%%%%%%%%%%%%
\subsection{The (pre)quotient of an action}\label{subsecQuotient}

The definition of the prequotient given below is a natural
generalization of the construction given in \cite{breen} when $M$ is
a point.

For a (right) action of a cfg-groupoid $\G\rra M$ on a category
fibred in groupoids $\X$ along $\ma: \X\to M$, we will define a
category $\X\pq\G$ and a functor $\pi_{\X\pq\G}:\X\pq\G\ra \cl{C}$
to the category of manifolds as follows. The objects of $\cl{X}\pq
\cl{G}$ are given by
\begin{equation}\label{eq:objpre}
\Obj(\cl{X}\pq\cl{G}):= \Obj(\cl{X}).
\end{equation}
For $x,y\in \Obj(\cl{X})$, the set of morphisms in $\cl{X}\pq\cl{G}$
between $x$ and $y$ is
\begin{equation}\label{eq:morpre}
\mr{Hom}_{\cl{X}\pq\cl{G}}(x,y):=\big\{(g,b): g\in\cl{G},\;
\ma(x)=\tar(g),\; b:x\cdot g\rightarrow y \mbox{ in }\cl{X}\big\}
/\sim
\end{equation}
where the equivalence relation $\sim$ results from identifying two
pairs $(g,b)$ and $(\bar{g},\bar{b})$ if there exists $j:
g\stackrel{\sim}{\rightarrow} \bar{g}$ in $\cl{G}$ such that
$\tar(j)=\mr{id}_{\ma(x)}$ and the following triangle commutes:
\begin{equation}\label{eq:triang}
\xymatrix{
x\cdot g\ar[rd]^-{b}\ar[d]_-{\mr{id}_x\cdot j}&\\
x\cdot \bar{g}\ar[r]_-{\bar{b}}&y.
}
\end{equation}
Considering the structure functors $\pi_\cl{G}:\cl{G}\to \cl{C}$ and
$\pi_{\cl{X}}:\cl{X}\to \cl{C}$, note that
$\pi_\cl{G}(g)=\pi_\cl{G}(\bar{g})=\pi_\cl{X}(x)$ and
$\pi_\cl{G}(j)=\mr{id}_{\pi_{\cl{X}}(x)}$. We denote the equivalence
class of $(g,b)$ in $\mr{Hom}_{\cl{X}\pq\cl{G}}(x,y)$ by $[g,b]$.
The composition of morphisms,
$$
x\lmap{[g,b]}y\lmap{[h,c]} z,
$$
is defined as follows. Let $\mu:U\rightarrow V$ be the morphism in
$\cl{C}$ given by $\pi_\X (b)$.  Since $\pi_\G(h)=V$, we may take a
morphism $\mu_h:\mu^*h\map h$ in $\cl{G}$ over $\mu$ representing
the pull back of $h$ along $\mu$ (this involves a choice, unique up
to canonical isomorphism). Since $\ma(y)=\tar(h)$ and $M$ is fibred
in sets, it follows that $\tar(\mu^*h)=\sour(g)$ and
$\ma(b)=\tar(\mu_h)$. Hence it makes sense to compose $[h,c]$ and
$[g,b]$ by
\begin{equation}\label{eq:comppre}
[h,c][g,b]:=[g\cdot \mu^*h, \;c\circ (b\cdot \mu_h)\circ \beta(x,g,
\mu^*h)],
\end{equation}
where $\beta$ is defined in {\bf (a3)}  of Section~\ref{subsec:act},
and this definition does not depend on the choices of
representatives $(g,b)$ and $(h,c)$, nor on the choices of $\mu^*h$
and $\mu_h$. For each $x\in \Obj(\cl{X})$, there is an associated
identity $\mr{id}_x\in \mr{Hom}_{\cl{X}\pq\cl{G}}(x,x)$ given by
\begin{equation}\label{eq:idpre}
\mr{id}_x:=[\un_{\ma(x)},\; \varepsilon(x)],
\end{equation}
with $\varepsilon$ defined in {\bf (a3)}.
Proposition~\ref{PropPreqCfg} below asserts that $\cl{X}\pq\cl{G}$
is indeed a category, with objects \eqref{eq:objpre} and morphisms
\eqref{eq:morpre}. Moreover, the maps $\Obj(\cl{X}\pq\cl{G}) \to
\Obj(\cl{C})$, sending objects $x$ to $\pi_\cl{X}(x)$, and
$\mr{Hom}_{\cl{X}\pq\cl{G}}(x,y) \to
\mr{Hom}_{\cl{C}}(\pi_\cl{X}(x),\pi_\cl{X}(y))$, given by
\begin{equation}\label{eq:func}
[g,b] \mapsto \pi_\X(b)
\end{equation}
(the reader may verify that this is well defined), define a functor
\begin{equation}\label{eq:projpre}
\pi_{\cl{X}\pq\cl{G}}: \cl{X}\pq\cl{G}\map \cl{C}.
\end{equation}

\begin{proposition}\label{PropPreqCfg}
The definitions in \eqref{eq:objpre}, \eqref{eq:morpre},
\eqref{eq:comppre} and \eqref{eq:idpre}  make $\X\pq\G$ into a
category and $\pi_{\cl{X}\pq\cl{G}}$ into a functor, in such a way
that $(\X\pq\G, \pi_{\cl{X}\pq\cl{G}})$ is a category fibred in
groupoids over $\cl{C}$.
\end{proposition}

The proof of this proposition is discussed in
Appendix~\ref{app:preq}.

\begin{definition}\label{def:pquot}
The category fibred in groupoids $\X\pq\G$ is the
\textbf{prequotient} of $\X$ by the action of $\G$.
%\begin{definition}\label{def:quot}
The \textbf{quotient} $\cl{X}/\cl{G}$ of $\cl{X}$ by the  action of
$\cl{G}$ is the stackification of the prequotient $\X\pq\G$.
%\end{definition}
\end{definition}

There is a natural projection functor
\begin{equation}\label{eq:quotproj}
\q: \cl{X}\map \cl{X}\pq\cl{G}
\end{equation}
defined by
%\begin{align*}
\begin{equation}\label{eqnProjExpl}
\q(x):=x,\qquad \q(b):=[\un_{\ma(x)},b\scirc \varepsilon (x)],
%\end{align*}
\end{equation}
for $x\in \Obj(\cl{X})$ and $b:x\to y$ morphism in $\cl{X}$. The
fact that $\q$ respects composition follows from $(xg1)$ in ${\bf
(a4)}$, while $\q$ respects the identity by definition. Since
$\pi_\X(\varepsilon)=\mr{id}$, it follows that $\q$ commutes with
the projections to $\cl{C}$. This shows the next result.

\begin{proposition}
The functor $\q$ is a morphism of categories fibred in groupoids.
\end{proposition}

We keep the notation
\begin{equation}\label{eq:quotmap2}
\q: \X \to \X/\G
\end{equation}
for the morphism induced by \eqref{eq:quotproj} and the
stackification $\X\pq\G \to \X/\G$ (see
Prop.~\ref{PropStackification}).

\begin{remark}\label{ex:actfib}
Note that any action of $\cl{G}$ on $\X$ is on the fibers of the
prequotient map $\cl{X}\ra \cl{X}\pq\cl{G}$ via $\gamma_0$ defined
by
$$
\gamma_0(x,g):=[g,\mr{id}_{xg}]: x\ra xg.
$$
Here the higher coherences $(xg1)$ and $(x1g)$ of ${\bf (a4)}$ are
used to verify that $\gamma_0$ is indeed a natural transformation,
and $(xg1)$ is used to prove  conditions $(\gamma\beta)$ and
$(\gamma\varepsilon)$ of Def.~\ref{DefActOnFibers}. It follows from
Remark~\ref{RemActFibXG}(b) that $\cl{G}$ also acts on the fibers of
$\cl{X}\ra\cl{X}/\cl{G}$.
\end{remark}

\begin{example}[The representable case]\label{ex:quotientst}
Suppose that $\G$ and $\X$ are representable, i.e., isomorphic to
manifolds $G$ and $X$. Then the action of $\G$ on $\X$ boils down to
an ordinary Lie groupoid action of $G\rra M$ on $X$, along $a: X\to
M$. Let us denote $M$ by $G_0$.

In this case, the prequotient $X\pq G$ is isomorphic in
$\CFG_{\cl{C}}$ to the category fibred in groupoids $[X/G]_p$, whose
objects are trivial principal $G$-bundles equipped with an
equivariant map to $X$, and whose morphisms are morphisms of bundles
commuting with the map to $X$. Indeed, an object of $[X/G]_p$ is
equivalent to a manifold $U$ (thought of as the base of the bundle)
equipped with a morphism $c:U \ra X$. More explicitly, the
corresponding $G$-bundle is
$$
U \times_{a c, G_0, t} G \stackrel{\mathrm{pr}}{\ra} U,
$$
and the map $U\times_{G_0} G\ra X$ is given by $(z,g)\mapsto
c(z)\cdot g$. Viewing $X$ as a category fibred in groupoids, the
pair $(U,c)$ is the same as an object in $X$, and hence an object in
$X\pq G$ (since $\Obj(X\pq G)=\Obj(X)$).

Similarly, a morphism in $[X/G]_p$ between the objects $x$, $x'$,
corresponding to $c: U\ra X$ and $c': U'\ra X$, is equivalent to a
pair of smooth maps $\mu: U\ra U'$ and $\nu: U\ra G$ such that $a  c
= t   \nu$ and $c \cdot \nu = c'  \mu$ (here $c\cdot \nu$ is defined
by the $G$-action on $X$). More explicitly, the bundle map $U
\times_{G_0} G\ra U'\times_{G_0} G$ is given by $(z,g)\mapsto
(\mu(z), \nu(z)^{-1}g)$. We must show that giving such $\mu$ and
$\nu$ is the same as giving a morphism in $X\pq G$ between $x$ and
$x'$. This follows from the definition of prequotient. Indeed,  an
object $g$ of $G$ (viewed in $\CFG_\cl{C}$) is the same as a smooth
map $\nu: \tilde{U}\ra G$, and by \eqref{eq:morpre} we have the
condition $a c=t  \nu$, which implies that $\tilde{U}= U$. The
object $x \cdot g$ of $X$ is given by $c \cdot \nu: U\ra X$, and a
morphism $b: x\cdot g\ra x'$ in $X$ is the same as a smooth map
$\mu: U\ra U'$ such that $c'  \mu= c\cdot \nu$.

Following Def.~\ref{def:pquot}, the quotient of $X$ by $G$, given by
the stackification of $[X/G]_p$, is just the usual quotient stack of
an action of a Lie groupoid, as described in
Section~\ref{subsec:LG}.
\end{example}

%%%%%%%%%%%%%%%%%%%%%%%%%%%%%%%%%%%%%%%%%%%%%%%%%%%%%%%%%%%%%%%%%%%%%%%%%%%%%%%%%%%%%
\subsection{Conditions for the prequotient to be a prestack}\label{subsec:StackQuotient}

Let a cfg-groupoid $\G\rra M$ act on a category fibred in groupoids
$\X$ (on the right, along $\ma:\X\to M$), and consider the
prequotient $\X\pq\G \in \Obj(\CFG_{\cl{C}})$. We will be
particularly interested in the case where $\X\pq\G$ is a prestack.
The following are simple consequences of this fact:

\begin{proposition}\label{propXtoX/Gepi/Gprestackthen}
Suppose that the prequotient $\X\pq \G \in \Obj(\CFG_{\cl{C}})$ is a
prestack. Then
\begin{itemize}
\item[(a)] the  morphism $\q: \cl{X}\map \cl{X}/\cl{G}$ in \eqref{eq:quotmap2} is an epimorphism;

\item[(b)] the natural morphism $ \cl{X}\times_{\cl{X}\pq\cl{G}}\cl{X}\map \cl{X}\times_{\cl{X}/\cl{G}}\cl{X}$
is an isomorphism.
\end{itemize}
\end{proposition}

\begin{proof}
The morphism $\cl{X}\ra \cl{X}\pq\cl{G}$ is surjective at the level
of objects, hence it is an epimorphism. The morphism
$\cl{X}\pq\cl{G}\ra \cl{X}/\cl{G}$ is an epimorphism, see
Prop.~\ref{PropStackification} (iii), and part $(a)$ follows. Part
$(b)$ is  standard.
\end{proof}

We will now discuss conditions guaranteeing that the prequotient
$\X\pq\G$ is a prestack.

\begin{definition}\label{def1free}
The action of $\cl{G}$ on $\cl{X}$ is said to be \textbf{1-free} if,
for all $x\in \Obj(\cl{X})$, the section functor of the action,
$$
\act(x,-): \tar^{-1}(\ma(x))\map \cl{X},
$$
is faithful. (Here $\tar^{-1}(\ma(x))$ is the fiber of the functor
$\tar$ over the object $\ma(x)$.)
\end{definition}
In other words, the action is 1-free if for all $j,j':g\rightarrow
\bar{g}$ in $\cl{G}$ with $\tar(j)=\tar(j')=\mr{id}_{\ma(x)}$, we
$$
\mr{id}_x\cdot j=\mr{id}_x\cdot j'\quad \then \quad j=j'
$$
(note the formal analogy with the usual set-theoretic notion of freeness, but now applied to 1-morphisms, hence the terminology). It is straightforward to verify that if the action-projection map
$\Delta$ \eqref{eq:Delta} of the action is faithful, then the action
is 1-free.

The following is the main result of this section.

\begin{proposition}\label{prop1freeThenPrestack}
If $\G$ is a stacky groupoid, $\X$ is a prestack, and the action of
$\G$ on $\X$ is 1-free (in particular, if the action-projection map
is faithful), then $\X\pq\G$ is a prestack.
\end{proposition}

\begin{proof}
Consider a manifold $U$, an open cover $(U_\alpha\map U)_\alpha$,
and consider objects $x,x'\in (\cl{X}\pq\cl{G})_U=\cl{X}_U$. We must
verify the following:
\begin{itemize}
\item[(i)] Given morphisms in $(\cl{X}\pq\cl{G})_U$,
$$
[g,b],[\bar{g},\bar{b}]: x\map x',
$$
such that  $[g,b]_{|U_\alpha}= [\bar{g},\bar{b}]_{|U_\alpha}$ for
each $\alpha$,
 then $[g,b]=[\bar{g},\bar{b}]$.

\item[(ii)] Given
 $[g_\alpha, b_\alpha]: x_{|U_\alpha}\rightarrow
x'_{|U_\alpha}$ morphisms in $(\cl{X}\pq\cl{G})_{U_\alpha} $ such
that $[g_\alpha, b_\alpha]_{|U_{\alpha\beta}} =[g_\beta,
b_\beta]_{|U_{\alpha\beta}}$ for all $\alpha$ and $\beta$, then
there exists $[g,b]:x\rightarrow x'$ in $(\cl{X}\pq\cl{G})_U$ such
that, for all $\alpha$, $[g,b]_{|U_\alpha}=[g_\alpha, b_\alpha]$.
\end{itemize}

Verifying (i) amounts to proving that there exists $j: g\lmap{\sim}
\bar{g}$ in $\cl{G}$ such that $\tar(j)=\mr{id}_{\ma(x)}$ and
$\bar{b}\circ (\mr{id}_x\cdot j)=b$. Here $g,\bar{g}\in \cl{G}$ and
$\ma(x)=\tar(g)=\tar(\bar{g})$; moreover, $b:x\cdot g\rightarrow x'$
and $\bar{b}:x\cdot \bar{g} \rightarrow x'$ are arrows in $\cl{X}$
over the identity of $U$, hence isomorphisms. Notice also that
$[g,b]_{|U_\alpha} =[g_{|U_\alpha}, b_{|U_\alpha}]$ by Remark
\ref{rempullbackmorphisminXG}.

The assumption that $[g,b]_{|U_\alpha}=
[\bar{g},\bar{b}]_{|U_\alpha}$ says that, for all $\alpha$, there is
a $j_\alpha: g_{|U_\alpha}\lmap{\sim} \bar{g}_{|U_\alpha}$ such that
$\tar(j_\alpha)=\mr{id}_{\ma(x_{|U_\alpha})}$ and
\begin{eqnarray}\label{eqn1}
\bar{b}_{|U_\alpha}\circ (\mr{id}_{x_{|U_\alpha}}\cdot j_\alpha)=
b_{|U_\alpha}.
\end{eqnarray}
Restricting this last equality to $U_{\alpha\beta}$ and using the
fact that $\bar{b}$ is an isomorphism, it follows that, for all
$\alpha$ and $\beta$,
$$
\mr{id}_{x_{|U_{\alpha\beta}}}\cdot {j_\alpha}_{|U_{\alpha\beta}}=
\mr{id}_{x_{|U_{\alpha\beta}}}\cdot {j_\beta}_{|U_{\alpha\beta}}.
$$
Since the action is 1-free, it follows that
${j_\alpha}_{|U_{\alpha\beta}}={j_\beta}_{|U_{\alpha\beta}}$, and
since $\cl{G}$ is a prestack it follows that there exists $j:
g\rightarrow \bar{g}$ in $\cl{G}$ such that $\pi_\G(j)=\mr{id}_U$
and $j_{|U_\alpha}=j_\alpha$. Note that $j$ is an isomorphism. Using
that $\pi_\G(j)=\mr{id}_U$ and $M$ is fibred in sets, we see that
$\tar(j)=\mr{id}_{\ma(x)}$. By \eqref{eqn1} and
 $j_{|U_\alpha}=j_\alpha$, we have that
$$
\big(\bar{b}\circ (\mr{id}_x\cdot j)\big)_{|U_\alpha}=b_{|U_\alpha}.
$$
Since $\cl{X}$ is a prestack, it follows that $\bar{b}\circ
(\mr{id}_x\cdot j)=b$, and we conclude the proof of (i).

To prove (ii), we notice that the $b_\alpha$'s are isomorphisms and
that $[g_\alpha, b_\alpha]_{|U_{\alpha\beta}}=
[{g_\alpha}_{|U_{\alpha\beta}}, {b_\alpha}_{|U_{\alpha\beta}}]$.
Since $[g_\alpha, b_\alpha]_{|U_{\alpha\beta}} =[g_\beta,
b_\beta]_{|U_{\alpha\beta}}$, there are $j_{\alpha\beta}
:{g_\beta}_{|U_{\alpha\beta}}\lmap{\sim}
{g_\alpha}_{|U_{\alpha\beta}}$ in $\cl{G}$ such that
$\tar(j_{\alpha\beta})=\mr{id}_{\ma(x_{|U_{\alpha\beta}})}$ and
$$
{b_\alpha}_{|U_{\alpha\beta}}\circ (\mr{id}_{x_{|U_{\alpha\beta}}}
\cdot j_{\alpha\beta} )={b_\beta}_{|U_{\alpha\beta}}.
$$
Restricting the above equality to $U_{\alpha\beta\gamma}$ and using
the fact that the $b_\alpha$'s are isomorphisms, we conclude that
\begin{eqnarray*}
\mr{id}_{x_{|U_{\alpha\beta\gamma}}}\cdot
 {j_{\alpha\gamma}}_{|U_{\alpha\beta\gamma}}&=&
{b_\alpha}^{-1}_{|U_{\alpha\beta\gamma}}\circ
{b_\gamma}_{|U_{\alpha\beta\gamma}} \\ &=&
 {b_\alpha}^{-1}_{|U_{\alpha\beta\gamma}}\circ
{b_\beta}_{|U_{\alpha\beta\gamma}} \circ
{b_\beta}^{-1}_{|U_{\alpha\beta\gamma}}\circ
{b_\gamma}_{|U_{\alpha\beta\gamma}} \\ &=&
\mr{id}_{x_{|U_{\alpha\beta\gamma}}}\cdot(
 {j_{\alpha\beta}}_{|U_{\alpha\beta\gamma}} \circ
{j_{\beta\gamma}}_{|U_{\alpha\beta\gamma}}) .
\end{eqnarray*}
It follows from the action being 1-free that
${j_{\alpha\gamma}}_{|U_{\alpha\beta\gamma}} =
{j_{\alpha\beta}}_{|U_{\alpha\beta\gamma}} \circ
{j_{\beta\gamma}}_{|U_{\alpha\beta\gamma}}$, and using that $\cl{G}$
is a stack it follows that there exists $g\in \cl{G}_U$ and
isomorphisms
$$
\varphi_\alpha: g_{|U_\alpha} \lmap{\sim} g_\alpha
$$
in $\cl{G}_{U_\alpha}$ such that
$$
j_{\alpha\beta}\circ {\varphi_\beta}_{|U_{\alpha\beta}}=
{\varphi_\alpha}_{|U_{\alpha\beta}}.
$$
Notice that $\tar(\varphi_\alpha)= \mr{id}_{\ma(x_{|U_\alpha})}$. We
now use the fact that $M$, viewed as a category fibred in groupoids,
is automatically a prestack to prove that
$$
\tar(g)=\ma(x).
$$
Indeed, for all $\alpha$ we have $(\tar
g)_{|U_\alpha}=\tar(g_{|U_\alpha})=\tar(g_\alpha)=
\ma(x_{|U_\alpha}) =(\ma x)_{|U_\alpha} $, where in the second
equality we used the fact that $M$ is fibred in sets.

Since $\tar(\varphi_\alpha)= \mr{id}_{\ma(x_{|U_\alpha})}$, we can
define an isomorphism in $\cl{X}_{U_\alpha}$ by
$$
\overline{b}_\alpha:=b_\alpha\circ (\mr{id}_{x_{|U_\alpha}}\cdot
\varphi_\alpha): (x\cdot g)_{|U_\alpha}= x_{|U\alpha}\cdot
g_{|U\alpha}\lmap{\sim} x'_{|U_\alpha}.
$$
We have that
\begin{eqnarray*}
 {\overline{b}_\alpha}_{|U_{\alpha\beta}}&=&
{b_\alpha}_{|U_{\alpha\beta}}\circ (\mr{id}_{x_{|U_{\alpha\beta}}}
\cdot {\varphi_\alpha}_{|U_{\alpha\beta}})\\ &=&
{b_\alpha}_{|U_{\alpha\beta}}\circ
\big(\mr{id}_{x_{|U_{\alpha\beta}}}\cdot (j_{\alpha\beta}\circ
{\varphi_\beta}_{|U_{\alpha\beta}})\big)\\&=&
{b_\alpha}_{|U_{\alpha\beta}}\circ (
\mr{id}_{x_{|U_{\alpha\beta}}}\cdot j_{\alpha\beta})\circ
(\mr{id}_{x_{|U_{\alpha\beta}}}\cdot
 {\varphi_\beta}_{|U_{\alpha\beta}})\\&=&
{b_\beta}_{|U_{\alpha\beta}}\circ
(\mr{id}_{x_{|U_{\alpha\beta}}}\cdot
{\varphi_\beta}_{|U_{\alpha\beta}})\\&=&
{\overline{b}_\beta}_{|U_{\alpha\beta}}.
\end{eqnarray*}

Since $\cl{X}$ is a prestack, there exists a morphism in $\cl{X}_U$,
$$
b:x\cdot g\map x',
$$
such that  $b_{|U_\alpha}=\overline{b}_\alpha$, for all $\alpha$.
This produces a morphism in $(\cl{X}/\cl{G})_U$:
$$
[g,b]: x\map x'.
$$
The existence of the isomorphisms $\varphi_\alpha$ implies that
$[g,b]_{|U_\alpha}= [g_\alpha,b_\alpha]$, and this concludes  the
proof of (ii).
\end{proof}

%%%%%%%%%%%%%%%%%%%%%%%%%%%%%%%%%%%%%%%%%%%%%%%%%%%%%%%%%%%%%%%%%%%%%%%%%%%%%%%%%%%%
\subsection{The universal property of
(pre)quotients}\label{subsec:universal}

Let $\G\rra M$ be a cfg-groupoid acting (on the right) on $\X \in
\Obj(\CFG_{\cl{C}})$.

\begin{proposition}\label{LemmaXGtoS}
The quotient map $\q: \X \to \X\pq \G$  has the following universal
property:
\begin{itemize}
\item[(i)] Suppose that $\G$ acts on $\X$ on the fibers of a morphism
$\Proj:\cl{X}\ra \cl{S}$ in $\CFG_{\cl{C}}$ via $\gamma$ (see
Def.~\ref{DefActOnFibers}). Then there exists a pair
$(\Phi,\varphi)$, where
$$
\Phi:\cl{X}\pq\cl{G}\ra\cl{S},
$$
$\varphi: \Phi\q\lmap{\sim} \Proj$,
and such that
$$
\gamma*(\varphi\,\scirc\,\mr{id}_{\mathrm{pr}})=
(\varphi\,\scirc\,\mr{id}_{\act})*(\mr{id}_\Phi\,\scirc\,\gamma_0),
$$
where $\mathrm{pr}: \X\times_M \G \to \X$ is the natural projection,
$\act: \X\times_M \G \to \X$ is the action map, and $\gamma_0$
defined in Rem.~\ref{ex:actfib}. In other words, this last condition
means that for any $(x,g)\in \cl{X}\times_M\cl{G}$ the following
diagram of morphisms in $\cl{S}$ commutes:
$$
(\varphi\gamma):\qquad \xymatrix{
\Phi(x)\ar[r]^-{\varphi}\ar[d]_-{\Phi[g,\mr{id}_{xg}]}&\Proj(x)\ar[d]^-{\gamma}\\
\Phi(xg)\ar[r]_-{\varphi}&\Proj(xg)}
$$
(We notice that the left vertical map is $\Phi(\gamma_0)$.)

\item[(ii)] The pair $(\Phi,\varphi)$ is unique up to canonical
isomorphism, in the following sense. Let $\cl{G}$ act on the fibers
of another morphism $\bar{\Proj}:\cl{X}\ra\cl{S}$ via
$\bar{\gamma}$, and let $\varrho: \Proj \ra\bar{\Proj}$ be a
2-isomorphism such that for any $(x,g)\in \cl{X}\times_M\cl{G}$ the
following diagram of morphisms in $\cl{S}$ commutes:
$$
(\varrho\gamma):\qquad \xymatrix{
\Proj(x)\ar[r]^-{\varrho}\ar[d]_-{\gamma}&\bar{\Proj}(x)\ar[d]^-{\bar{\gamma}}\\
\Proj(xg)\ar[r]_-{\varrho}&\bar{\Proj}(xg).}
$$
Let $(\overline{\Phi},\overline{\varphi})$ be a pair where
$\overline{\Phi}:\cl{X}\pq\cl{G}\ra\cl{S}$, $\overline{\varphi}:
\overline{\Phi} \q \lmap{\sim} \overline{\Proj}$, and such that
condition $(\varphi\gamma)$ is satisfied. Then there exists a unique
$\psi: \Phi\lmap{\sim} \overline{\Phi}$ such that, for any
$x\in\cl{X}$, the following square commutes:
$$
(\psi\varphi\varrho):\qquad\qquad \xymatrix{
\Phi(x)\ar[r]^-{\psi}\ar[d]_-{\varphi}&\overline{\Phi}(x)\ar[d]^-{\bar{\varphi}}\\
\Proj(x)\ar[r]_-{\varrho}&\bar{\Proj}(x). }
$$
\end{itemize}
\end{proposition}

\begin{proof}
We start with the existence of $(\Phi,\varphi)$.

We define the morphism $\Phi:\cl{X}\pq\cl{G}\ra \cl{S}$ as follows.
For any $x\in\mr{Obj}(\cl{X}\pq\cl{G}) =\mr{Obj}(\cl{X})$, we set
$\Phi(x):=\Proj(x)$. If $[g,b]:x\ra y$ is a morphism in
$\cl{X}\pq\cl{G}$ then we define $\Phi[g,b]:=\Proj(b)\gamma$,
\begin{equation}\label{eq:defgb}
\Proj(x)\lmap{\gamma}\Proj(xg)\lmap{\Proj(b)} \Proj(y).
\end{equation}
Using the naturality of $\gamma$ one shows  that $\Phi[g,b]$ is
independent of the representative chosen for $[g,b]$. The fact that
$\Phi$ sends composition of morphisms to the composition of the
images follows by the higher coherence $(\gamma\beta)$ of
Definition~\ref{DefActOnFibers}. Similarly, $\Phi(\mr{id})=\mr{id}$
as a consequence of the higher coherence $(\gamma\varepsilon)$.
Hence $\Phi$ is a functor that is a morphism of categories fibred in
groupoids. It follows from the definition of $\Phi$ and the higher
coherence $(\gamma\varepsilon)$ that $\Phi \q=\Proj$. Hence we can
choose $\varphi:=\mr{id}_\Proj$, and $(\varphi\gamma)$ is satisfied
because, by definition of $\Phi$, we have
$\Phi[g,\mr{id}_{xg}]=\gamma$. This defines $(\Phi,\varphi)$.

We now prove the uniqueness of $(\Phi,\varphi)$. First, condition
$(\psi\varphi\varrho)$ in the statement of the lemma defines $\psi$
uniquely. We must verify that $\psi: \Phi \to \overline{\Phi}$
defined in this way is indeed a natural transformation. Let
$[g,b]:x\ra y$ be a morphism in $\cl{X}\pq\cl{G}$, and consider the
following diagram:
$$
\xymatrix{
\Phi(x)\ar[rr]^-{\Phi[g,b]}\ar[dd]_-{\varphi}\ar[rd]_-{\Phi[g,\mr{id}]}&&
\Phi(y)\ar[dd]^-{\varphi}\\
&\Phi(xg)\ar[d]_-{\varphi}\ar[ur]_-{\Phi(\q(b))}&\\
\Proj(x)\ar[r]_-{\gamma}\ar[d]_-{\varrho}&\Proj(xg)\ar[r]^-{\Proj(b)}\ar[d]_-{\varrho}&
\Proj(y)\ar[d]^-{\varrho}\\
\bar{\Proj}(x)\ar[r]_-{\bar{\gamma}}&\bar{\Proj}(xg)\ar[r]^-{\bar{\Proj}(b)}&\bar{\Proj}(y)\\
&\bar{\Phi}(xg)\ar[u]_-{\bar{\varphi}}\ar[rd]^-{\bar{\Phi}(\q(b))}&\\
\bar{\Phi}(x)\ar[rr]_-{\bar{\Phi}[g,b]}\ar[uu]^-{\bar{\varphi}}\ar[ru]^-{\bar{\Phi}[g,\mr{id}]}&&
\bar{\Phi}(y)\ar[uu]_-{\bar{\varphi}}.\\
}
$$
Since $\psi=(\overline{\varphi})^{-1}\varrho \varphi$, we have to
prove that the outer rectangle commutes. Commutativity of the right
trapezia follows from naturality of $\varphi$ (the upper trapezium)
and of $\bar{\varphi}$ (the lower one). Commutativity of the left
trapezia follows from condition $(\varphi\gamma)$ applied to
$\varphi$ (the upper one) and to $\overline{\varphi}$ (the lower
one). Commutativity of the upper and bottom triangles follows from
$[g,b]=\q(b)[g,\mr{id}]$, which in turn follows from the higher
coherence $(xg1)$ of ${\bf (a4)}$, Definition~\ref{defaction}. The
left middle square commutes by condition $(\varrho\gamma)$, and the
right middle one by naturality of $\varrho$. This completes the
proof of the uniqueness of $(\Phi,\varphi)$.
\end{proof}

The universal property described in Prop.~\ref{LemmaXGtoS} still
holds when $\X\pq \G$ is replaced by its stackification $\X/\G$, as
long as $\cl{S}$ is a stack.

\begin{corollary}\label{CorolXGtoS}
Suppose that a cfg-groupoid $\G$ acts on $\X \in \Obj(\CFG_\cl{C})$
on the fibers of a morphism $\Proj:\X \to \cl{S}$ such that $\cl{S}$
is a stack. Then there exists a pair $(\Phi,\varphi)$, where $\Phi:
\X/\G \to \cl{S}$ and $\varphi: \Phi\q\lmap{\sim} \Proj$, with $\q:
\X \to \X/\G$ defined in \eqref{eq:quotmap2}, satisfying the
properties described in $(i)$ and $(ii)$ of the previous
proposition\footnote{To formulate properties $(i)$ and $(ii)$ of
Prop.~\ref{LemmaXGtoS} in the present context, we note that in
diagram $(\varphi\gamma)$ one has to substitute $\Phi(x)$ and
$\Phi(xg)$ with $\Phi(\q(x))$ and $\Phi(\q(xg))$ respectively, and
$\Phi[g,\mr{id}_{xg}]$ by $\Phi(\gamma)$, where $\gamma$ is the
2-isomorphism that makes $\cl{G}$ act on the fibers of $\q$.
Moreover, in the diagram $(\psi\varphi\varrho)$, one has to
substitute $\Phi(x)$ and $\overline{\Phi}(x)$ by $\Phi(\q(x))$ and
$\overline{\Phi}(\q(x))$, respectively.}.
\end{corollary}

\begin{proof}
The proof follows from Prop.~\ref{LemmaXGtoS} and the properties of
stackification, see Prop.~\ref{PropStackification} (one has to
replace $\gamma_0$ by the 2-isomorphism $\gamma$ that makes $\G$ act
on the fibers of $\X \to \X / \G$, as in Remark~\ref{ex:actfib}).
\end{proof}

\begin{corollary}\label{cor:principal}
If $\X$ is a principal $\G$-bundle over a differentiable stack
$\cl{S}$, then $\cl{S}$ is canonically isomorphic to $\X/\G$.
\end{corollary}

\begin{proof}
Let $\Proj: \cl{X}\to \cl{S}$ be the projection map of the principal
bundle (see Def.~\ref{DefPrincBundle}), so that $\cl{G}$ acts on its
fibres. As mentioned in the previous corollary, there is an induced
morphism $\cl{X}/\cl{G}\ra \cl{S}$; we take it as the morphism
induced by the morphism $\Phi: \cl{X}\pq\cl{G}\ra \cl{S}$
constructed in the proof of Prop.~\ref{LemmaXGtoS}. We must show
that $\cl{X}/\cl{G}\ra \cl{S}$ is an isomorphism.

Note that the fact that the map $\X\times_M\G \to
\X\times_{\cl{S}}\X$ in \eqref{eq:actfibre} is an isomorphism
(condition $3.$ in Def.~\ref{DefPrincBundle}) implies that the
$\G$-action on $\X$ is 1-free (indeed, since the natural map
$\X\times_{\cl{S}}\X \to \X\times \X$ is faithful, so is the
action-projection map, cf. \eqref{eq:factor} below, which implies
1-freeness), so $\cl{X}\pq \cl{G}$ is a prestack
(Prop.~\ref{prop1freeThenPrestack}). Hence, by
Prop.~\ref{PropStackification} (iv), it is enough to show that
$\Phi$ is a monomorphism and an epimorphism.

The fact that $\Phi$ is an epimorphism follows directly from the
fact that $\Proj$ is. We will show that $\Phi$ is a monomorphism,
i.e., that for any manifold $U$ the fiber of $\Phi$ at $U$, i.e.,
the functor
 $\Phi_U: (\cl{X}\pq\cl{G})_U\ra \cl{S}_U$,
is fully faithful. In what follows, all the objects and morphisms
are understood to be over $U$; in particular, all morphisms are
isomorphisms.

Let $x,y\in\mr{Obj}(\cl{X}\pq\cl{G})=\mr{Obj}(\cl{X})$, and let
$a:\Proj(x)\ra \Proj(y)$  be a morphism in $\cl{S}$ (recall that
$\Phi(x)=\Proj(x)$ and $\Phi(y)=\Proj(y)$). We have to show that
there exists a unique morphism $[g,b]:x\ra y$ in $\cl{X}\pq\cl{G}$
such that $\Phi[g,b]=a$.

\textit{Existence of $[g,b]$}: The triple $(x,a,y)$ is an object of
$\cl{X}\times_\cl{S}\cl{X}$. Since we assume that the canonical
morphism $ \cl{X}\times_M \cl{G}\ra \cl{X}\times_{\cl{S}}\cl{X}$ is
an isomorphism, there exists $(\bar{x},g)\in\cl{X}\times_M \cl{G}$
and an isomorphism in $\cl{X}\times_{\cl{S}}\X$ between
$(\bar{x},\bar{\gamma},\bar{x}g)$ and $(x,a,y)$, where
$\bar{\gamma}$ is induced by the isomorphism that makes $\cl{G} $
act on the fibres of $\Proj$; in other words, we have isomorphisms $
c:\bar{x}\ra x$ and ${b}':\bar{x}g\ra y$ in $\cl{X}$ such that the
square
$$
\xymatrix{\Proj(\bar{x})\ar[r]^-{\Proj(c)}\ar[d]_-{\bar{\gamma}}
& \Proj(x)\ar[d]^-{a}\\
\Proj(\bar{x}g)\ar[r]_-{\Proj({b}')}& \Proj(y) }
$$
commutes. Defining $b:= {b}'\,\scirc\,(c\cdot \mr{id})^{-1}: xg\ra
y$, we conclude that $\Phi[g,b]=a$ (cf. \eqref{eq:defgb}).

\textit{Uniqueness of $[g,b]$}: Let $[\bar{g},\bar{b}]:x\ra y $ be
another morphism such that $\Phi[\bar{g},\bar{b}]=a$. Then
$(\mr{id},\bar{b}^{-1}b): (x,\gamma,xg)\ra
(x,\bar{\gamma},x\bar{g})$ is a morphism in
$\cl{X}\times_\cl{S}\cl{X}$ between the images of
$(x,g),(x,\bar{g})\in \Obj(\cl{X}\times_M\cl{G})$; notice that this
holds since $\Proj(\bar{b})\bar{\gamma}=\Proj(b)\gamma$, as both
compositions are equal to $a$. Since $ \cl{X}\times_M \cl{G}\ra
\cl{X}\times_{\cl{S}}\cl{X}$ is an isomorphism, there exists a
morphism $(c,j):(x,g)\ra (x,\bar{g})$ in $\cl{X}\times_M\cl{G}$
whose image in $\cl{X}\times_\cl{S}\cl{X}$ is
$(\mr{id},\bar{b}^{-1}b)$. In particular, $c=\mr{id}$ and
$(\mr{id}\cdot j)=\bar{b}^{-1}b$. We conclude that
$[g,b]=[\bar{g},\bar{b}]$.
\end{proof}

As a consequence of the last corollary, we note that any principal
$\G$-bundle $\X$ over a differentiable stack $\cl{Y}$, with
projection $\q_\mathcal{Y}: \X\to \mathcal{Y}$, satisfies the same
universal property as the quotient $\q: \cl{X}\ra\cl{X}/\cl{G}$
(cf. Corollary~\ref{CorolXGtoS}):

\begin{corollary}\label{propUnivPropPrincBund}
Let $\X$ be a principal $\G$-bundle with projection $\q_\cl{Y}:
\cl{X}\ra \cl{Y}$. Suppose that $\G$ acts on $\X$ on the fibers of a
morphism $\Proj: \X\to \cl{S}$ via $\gamma$, where $\cl{S}$ is a
differentiable stack. Then:
\begin{itemize}
\item[(i)] There exists a pair $(\Phi,\varphi)$, where $\Phi: \cl{Y}\to
\cl{S}$ and $\varphi: \Phi \q_\cl{Y} \stackrel{\sim}{\to} \Proj$ are
as in $(i)$ of Prop.~\ref{LemmaXGtoS} (suitably adapted to the
present context, similarly to the footnote of
Cor.~\ref{CorolXGtoS}).

%\footnote{To formulate properties $(i)$ and $(ii)$ of
%Prop.~\ref{LemmaXGtoS} in the present context, we note that in
%diagram $(\varphi\gamma)$ one has to substitute $\Phi(x)$ and
%$\Phi(xg)$ with $\Phi(\q_{\cl{Y}}(x))$ and $\Phi(\q_{\cl{Y}}(xg))$
%respectively, and $\Phi[g,\mr{id}_{xg}]$ by $\Phi(\gamma_0)$, where
%$\gamma_0$ is the 2-isomorphism that makes $\cl{G}$ act on the
%fibers of $\q_\cl{Y}:\cl{X}\ra\cl{Y}$. Moreover, in the diagram
%$(\psi\varphi\varrho)$, one has to substitute $\Phi(x)$ and
%$\overline{\Phi}(x)$ by $\Phi(\q_{\cl{Y}}(x))$ and
%$\overline{\Phi}(\q_{\cl{Y}}(x))$, respectively.}.

\item[(ii)] The pair $(\Phi,\varphi)$ is unique in the sense of $(ii)$ of
the same proposition.
\end{itemize}
\end{corollary}

%%%%%%%%%%%%%%%%%%%%%%%%%%%%%%%%%%%%%%%%%%%%%%%%%%%%%%%%%%%%%%%%%%%%%
\subsection{Equivariant maps and (pre)quotients}

This section discusses the relation between equivariant maps and
maps of (pre)quotients.

Consider a cfg-groupoid $\cl{G}$ acting on $\cl{X}_i \in
\Obj(\CFG_\cl{C})$, $i=1,2$, and let $F:\cl{X}_1\ra\cl{X}_2$ be a
$\G$-equivariant morphism, with associated equivariance
2-isomorphism $\delta$ (Def.~\ref{DefGequivariant}). Let
$\gamma_0^i$ be the 2-isomorphism associated with the
$\cl{G}$-action on the fibers of $\q_i: \cl{X}_i\to \cl{X}_i\pq
\cl{G}$, as in Remark~\ref{ex:actfib}.

The next proposition makes precise the fact that $F$ induces a
morphism $\cl{X}_1\pq \cl{G} \to \cl{X}_2\pq \cl{G}$.

\begin{proposition}\label{PropfoverG}
The following holds:

\begin{itemize}
\item[(i)] There is a morphism $\Phi:\cl{X}_1\pq \cl{G} \ra\cl{X}_2\pq \cl{G}$ and a 2-isomorphism
$\varphi: \Phi\q_1\lmap{\sim}\q_2 F$ satisfying the following higher
coherence condition: for any $(x,g)\in\cl{X}_1\times_M\cl{G}$, the
diagram
$$
(\varphi\gamma\delta):\qquad\qquad \xymatrix{
\Phi\q_1(x)\ar[rr]^-{\Phi(\gamma^1_0)}\ar[d]_-{\varphi}
&&\Phi\q_1(xg)\ar[d]^-{\varphi}\\
\q_2F(x)\ar[r]_-{\gamma^2_0}&\q_2(F(x)g)\ar[r]_{\q_2({\delta})}&\q_2
F(xg)}
$$
commutes.

\item[(ii)] Let $F,\overline{F}:\cl{X}_1\ra\cl{X}_2$ be $\G$-equivariant morphisms,
with associated equivariance 2-isomorphisms $\delta$ and
$\bar{\delta}$, respectively, and let $(\Phi,\varphi)$ and
$(\overline{\Phi},\overline{\varphi})$ be corresponding pairs as in
$(i)$. Then, for any  $\varsigma: F\lmap{\sim} \overline{F}$ such
that, for all $(x,g)\in\cl{X}_1\times_M\cl{G}$, the diagram
$$
(\varsigma \delta):\qquad \qquad \xymatrix{
F(x)g\ar[r]^-{\delta}\ar[d]_-{\varsigma\cdot\mr{id}}
&F(xg)\ar[d]^-{\varsigma}\\
\overline{F}(x)g\ar[r]_-{\bar{\delta}}&\overline{F}(xg)}
$$
commutes, there exists a unique $\psi: \Phi
\lmap{\sim}\overline{\Phi}$ such that, for all $x\in\cl{X}_1$, the
following diagram commutes:
$$
(\psi\varsigma):\qquad \qquad \xymatrix{ \Phi
\q_1(x)\ar[r]^-{\psi}\ar[d]_-{\varphi}
&\overline{\Phi}\q_1(x)\ar[d]^-{\bar{\varphi}}\\
\q_2F(x)\ar[r]_-{\q_2(\varsigma)}&\q_2\overline{F}(x). }
$$
\end{itemize}
\end{proposition}

\newpage

\begin{proof}

To prove $(i)$, let us define
$\widetilde{\gamma}_2:=(\mr{id}_{\q_2}\scirc \delta)*
(\gamma^2_0\scirc \mr{id}_{F\times\mr{id}})$; more explicitly, for
$(x,g)\in\cl{X}_1\times_M\cl{G}$ we have $\widetilde{\gamma}_2 :=
\widetilde{\gamma}_2(x,g)=\q_2(\delta)\scirc \gamma^2_0(F(x),g)$,
$$
\widetilde{\gamma}_2: \q_2(F(x))\lmap{\gamma_0^2} \q_2(F(x)
g)\lmap{\q_2(\delta)} \q_2(F(xg)).
$$
One may check that  $\widetilde{\gamma}_2: \q_2 F \mathrm{pr}_{\X_1}
\to \q_2 F  {\act_1}$  makes $\cl{G}$ act on the fibers of $\q_2 F$,
i.e., that $\widetilde{\gamma}_2$ satisfies the higher coherences
described in Definition~\ref{DefActOnFibers}. Indeed, one directly
checks that condition $(\gamma\beta)$ for $\widetilde{\gamma}_2$
follows from the same condition for $\gamma^2_0$ and condition
$(\delta\beta_1\beta_2)$ of Definition~\ref{DefGequivariant}.
Moreover, condition $(\gamma\varepsilon)$ for $\widetilde{\gamma}_2$
follows from the same condition for $\gamma^2_0$ and condition
$(\delta\varepsilon_1\varepsilon_2)$.

Next, we apply the universal property of the quotient map
$\cl{X}_1\ra \cl{X}_1\pq \cl{G}$ to the morphism $\q_2 F$ using
Prop.~\ref{LemmaXGtoS}. It follows that there exists
$(\Phi,\varphi)$, where $\Phi:\cl{X}_1\pq \cl{G}\ra\cl{X}_2\pq
\cl{G}$ and $\varphi: \Phi \q_1\ra \q_2 F$, satisfying the higher
coherence given in part $(i)$ of Prop.~\ref{LemmaXGtoS}. This
condition is equivalent to condition $(\varphi\gamma\delta)$, so
$(i)$ follows.

The assertion in $(ii)$ is a restatement of part $(ii)$ of
Prop.~\ref{LemmaXGtoS} (with $\varrho = \mr{id}_{\q_2}\scirc
\varsigma: \q_2 F \to \q_2 \overline{F}$).
\end{proof}

The construction taking a $\G$-equivariant morphism $F: \X_1\to
\X_2$ to a morphism $\Phi: \cl{X}_1\pq \cl{G} \to \cl{X}_2\pq
\cl{G}$ preserves compositions and takes the identity to the
identity. We state the precise result about compositions, leaving
the proof to the reader.

\begin{proposition}\label{PropEquivComp}
Let $\G$ be a cfg-groupoid acting on $\cl{X}_i\in
\Obj(\CFG_\cl{C})$, for $i=1,2,3$,
and let $F_1:\cl{X}_1\ra\cl{X}_2$ and $F_2:\cl{X}_2\ra\cl{X}_3$ be
$\G$-equivariant morphisms, with associated equivariance
2-isomorphisms $\delta_1$ and $\delta_2$. Let $(\Phi_1,\varphi_1)$
and $(\Phi_2,\varphi_2)$ be pairs as in
Proposition~\ref{PropfoverG}$(i)$ corresponding  to $F_1$ and $F_2$,
respectively. Then the pair
$$
(\Phi_2\Phi_1, (\varphi_2\hc\mr{id}_{F_1})
*(\mr{id}_{\Phi_2}\hc \varphi_1))
$$
corresponds to $F_2 F_1$.
\end{proposition}

\begin{remark}\label{RemIsoIso}
\
\begin{itemize}
\item[(a)] The above functoriality property of the induced map implies that
if $F$ is an isomorphism then so is $\Phi$.
%\end{remark}

%\begin{remark}\label{RemfoverGquot}
\item[(b)] It is immediate (using Corollary \ref{CorolXGtoS}) that
Propositions \ref{PropfoverG}, \ref{PropEquivComp}, and part (a)
above still work if we consider the quotient $\cl{X}/\cl{G}$ instead
of the prequotient.
\end{itemize}
\end{remark}

We finally observe that passing from equivariant maps to maps of
quotients preserves the monomorphism and epimorphism conditions.

\begin{proposition}\label{prop:equivmonoepi}
Let $F:\X_1\ra\cl{X}_2$ be a $\G$-equivariant map between cfg's,
where $\G$ is a cfg-groupoid. Let $\Phi:\X_1\pq \G\ra \cl{X}_2\pq\G$
be the induced map on the prequotients. If $F$ is a monomorphism
(resp. an epimorphism) then so is $\Phi$.
\end{proposition}

\begin{proof}
Since the canonical projections $\cl{X}_i\ra \cl{X}_i\pq\G$ are the
identity on objects, it follows immediately that if $F$ is an
epimorphism then so is $\Phi$. For the remaining part of the proof
we fix a manifold $U$, and we assume that the morphisms considered
are over $\mr{id}_U$.

We first prove that if $F_U$ is faithful, than so is $\Phi_U$. Take
two morphisms $[g,b],[\bar{g},\bar{b}]: x\ra x'$ in $\X_1\pq\G$, and
assume that $\Phi[g,b]=\Phi[\bar{g},\bar{b}]$. Recall that $b:x\cdot
g\ra x'$ is a morphism in $\X_1$. Since $\Phi[g,b]=[g,F(b)]$ (after
the identification of $F(x)\cdot g$ with $F(x'\cdot g)$ by the
equivariance of $F$) then there exists an isomorphism $j:g\ra
\bar{g}$ such that $F(\bar{b})\circ (\mr{id}_{F(x)}\cdot j) =F(b)$.
Then, again using the equivariance of $F$, $F(\bar{b}\circ
(\mr{id}_{x}\cdot j))=F(b)$. Since $F_U$ is faithful then
$\bar{b}\circ (\mr{id}_{x}\cdot j)=b$, hence
$[g,b]=[\bar{g},\bar{b}]$.

We now verify that if $F_U$ is full, than so is $\Phi_U$. Take
objects $x,x'\in\Obj(\X_1\pq\G)=\Obj(\X_1)$ and a morphism $[g,c]:
F(x)\ra F(x')$ in $\cl{X}_2\pq \G$. By the equivariance of $F$ we
can interpret $c: F(x\cdot g)\ra F(x')$ as a morphism between the
images via $F$ of objects in $\X_1$. Since $F_U$ is full, there
exists $b: x\cdot g\ra x'$ in $\cl{X}_1$ such that $F(b)=c$. We
conclude $\Phi[g,b]=[g,c]$, and we are done.

\end{proof}

%%%%%%%%%%%%%%%%%%%%%%%%%%%%%%%%%%%%%%%%%%%%%%%%%%%%%%%%%%%%%%%%%%%%%%%%%%%%%
\subsection{Prequotients and the action-projection map}\label{subsec:actproj}

Let us consider a cfg-groupoid $\G\rra M$ acting on $\X \in
\Obj(\CFG_{\mathcal{C}})$ (on the right), and let $\Delta: \cl{X}
\times_{\ma,M,\tar} \cl{G}\map \cl{X}\times\cl{X}$,
$\Delta(x,g)=(x,xg)$, be the associated action-projection map. Using
the projection $\q: \X \to \X\pq \G$, we see that $\Delta$ induces a
morphism (cf. \eqref{eq:actfibre})
\begin{equation}\label{eq:actproj}
\Q :\cl{X} \times_{\ma,M,\tar}
\cl{G}\map \cl{X} \times_{\X\pq\G}
\cl{X},
\end{equation}
such that the diagram
\begin{equation}\label{eq:factor}
\xymatrix{
 \cl{X}\times_{M} \cl{G}\ar[r]^-{\Q}\ar[rd]_-{\Delta}&
\cl{X}\times_{\cl{X}\pq\cl{G}}\cl{X}\ar[d]^-{}\\
& \cl{X}\times\cl{X} }
\end{equation}
2-commutes.

\begin{proposition}\label{propQFaithful} The following holds:
\begin{enumerate}
 \item[(a)] $\Q$ is full and essentially surjective.
 \item[(b)] If $\Delta$ is faithful, then so is $\Q$.
\end{enumerate}
Hence, if $\Delta$ is faithful then $\Q$ is an isomorphism.
\end{proposition}

\begin{proof}
The claim in $(b)$ follows from the factorization \eqref{eq:factor}.

Let us prove $(a)$. To verify that $\Q$ is full, let us consider
objects
$$
(x,g),(\overline{x},\overline{g})\in  \cl{X}\times_{M}\cl{G},
$$
and a morphism between the images of these objects in
$\cl{X}\times_{\cl{X}\pq\cl{G}}\cl{X}$:
$$
(c_1,c_2): (x,  [g, \mr{id}_{x\cdot g}], x\cdot g)
\map(\overline{x}, [\overline{g}, \mr{id}_{\overline{x} \cdot
\overline{g}}], \overline{x}\cdot \overline{g}).
$$
We have to prove the existence of $(b,j):(x,g)\rightarrow
(\overline{x},\overline{g})$ in $\cl{X}\times_{M} \cl{G}$ such that
$(b, b\cdot j)=(c_1, c_2)$. Since necessarily $b=c_1$ then, given
$(c_1, c_2)$, we are looking for a morphism in $\cl{G}$,
$$
j:g\map\overline{g},
$$
such that $\tar(j)=\ma(c_1)$ and $c_2=c_1\cdot j$.

In order to proceed we observe that $c_1:x\rightarrow \overline{x}$
and $c_2: x\cdot g\rightarrow \overline{x}\cdot\overline{g}$ are
morphisms in $\cl{X}$ with
$$
\pi_\X(c_1)=\pi_\X(c_2)=: \mu,
$$
and that the square
$$
\xymatrix{ x\ar[d]_-{[\un_{\ma(x)}, c_1\sscirc \varepsilon
(x)]}\ar[r]^-{[g, \mr{id}_{x\cdot g}]}
&x\cdot g\ar[d]^-{[\un_{\ma(x\cdot g)}, c_2\sscirc \varepsilon (x\cdot g)]} \\
\overline{x}\ar[r]_-{[\overline{g}, \mr{id}_{\overline{x} \cdot
\overline{g}}]}
 & \overline{x}\cdot \overline{g}
}
$$
commutes in $\cl{X}\pq\cl{G}$. This implies that there exists
$$
j': g\cdot \un_{\ma(x\cdot g)} \lmap{\sim} \un_{\ma(x)}\cdot
\mu^*\overline{g}
$$
such that $\tar(j')=\mr{id}_{\ma(x)}$ and the diagram
\begin{equation}\label{eq:diag}
\xymatrix{ x\cdot(g\cdot \un_{\ma(x\cdot g)})\ar[d]_-{\mr{id}\cdot
j'}\ar[r]^-{\beta}&
(x\cdot g)\cdot \un_{\ma(x\cdot g)}\ar[d]^-{\varepsilon}&\\
x\cdot (\un_{\ma(x)}\cdot \mu^*\overline{g})\ar[d]_-{\beta}& x\cdot
g\ar[r]^-{c_2}& \overline{x}\cdot
\overline{g}\\
(x\cdot \un_{\ma(x)})\cdot
\mu^*\overline{g}\ar[r]_-{\varepsilon\cdot \mr{id}} & x\cdot
\mu^*\overline{g}\ar[ru]_-{c_1\cdot \mu_{\overline{g}}}& }
\end{equation}
commutes, where $\mu_{\overline{g}}:\mu^*\overline{g}\rightarrow
\overline{g}$ is a cartesian arrow in $\cl{G}$ over $\mu$. Since
$\un_{\ma(x\cdot g)}=\un_{\sour(g)}$ and
$\un_{\ma(x)}=\un_{\tar(\mu^*\overline{g})}$, we can consider the
following diagram of isomorphisms in $\cl{G}$:
$$
g\lmap{(\rho g)^{-1}} g\cdot \un_{\ma(x\cdot g)}\lmap{j'}
\un_{\ma(x)}\cdot \mu^*\overline{g}
\lmap{\lambda(\mu^*\overline{g})}\mu^*\overline{g}\lmap{\mu_{\overline{g}}}\overline{g},
$$
and define $j$ to be the above composition:
$$
j:=\mu_{\overline{g}}\circ \lambda(\mu^*\overline{g})\circ j'\circ
(\rho g)^{-1}.
$$
Since $\tar(\mu_{\overline{g}})=\ma(c_1)$,
$\tar(\lambda(\mu^*\overline{g}))= \tar(j')=\tar(\rho
g)^{-1}=\mr{id}_{\ma(x)}$, then $\tar(j)=\ma(c_1)$. We are left with
checking that $c_2= c_1\cdot j$, and this follows by completing the
diagram in \eqref{eq:diag} to
$$
\xymatrix{ x\cdot(g\cdot \un_{\ma(x\cdot g)})\ar[d]_-{\mr{id}\cdot
j'}\ar[r]^-{\beta} \ar[rd]_-{\mr{id}\cdot \rho}&
(x\cdot g)\cdot \un_{\ma(x\cdot g)}\ar[d]^-{\varepsilon}&\\
x\cdot (\un_{\ma(x)}\cdot
\mu^*\overline{g})\ar[d]_-{\beta}\ar[rd]^-{\mr{id}\cdot \lambda}&
x\cdot g\ar[r]^-{c_2}& \overline{x}\cdot
\overline{g}\\
(x\cdot \un_{\ma(x)})\cdot
\mu^*\overline{g}\ar[r]_-{\varepsilon\cdot \mr{id}} & x\cdot
\mu^*\overline{g}\ar[ru]_-{c_1\cdot \mu_{\overline{g}}}& }
$$
noticing that the two triangles commute by the higher coherences
$(xg1)$ and $(x1g)$ in ${\bf (a4)}$, all the arrows but possibly
$c_2$ and $(c_1\cdot \mu_{\overline{g}})$ are isomorphisms, and
$c_2= (c_1\cdot \mu_{\overline{g}})\circ (\mr{id}\cdot \lambda)\circ
(\mr{id}\cdot j')\circ (\mr{id}\cdot\rho^{-1})= c_1\cdot j$. This
completes the proof that $\Q$ is full.

It remains to show that $\Q$ is essentially surjective. Given an
object
$$
(x_1, [g,b], x_2)\in \Obj(\cl{X}\times_{\cl{X}\pq\cl{G}}\cl{X}),
$$
we have to show that there exists an object
$$
(x,h)\in \Obj(\cl{X}\times_{M}\cl{G})
$$
and an isomorphism
$$
(c_1,c_2): (x, [h, \mr{id}_{x\cdot h}], x\cdot h)\lmap{\sim}(x_1,
[g,b], x_2)\;\; \mbox{ in }\;\;
\cl{X}\times_{\cl{X}\pq\cl{G}}\cl{X}.
$$
Let
$$
x:=x_1, \;\;\; h:=g,
$$
and
$$
c_1:=\mr{id}_{x_1}:x_1\map x_1 \qquad c_2:=b:x_1\cdot g\map x_2
$$
in $\X$. The first condition to be verified is that the diagram
$$
\xymatrix{ x_1\ar[d]_-{[\un_{\ma(x_1)}, \varepsilon
(x_1)]}\ar[r]^-{[g, \mr{id}_{x_1\cdot g}]}
&x_1\cdot g\ar[d]^-{[\un_{\ma(x_1\cdot g)}, b\sscirc \varepsilon (x_1\cdot g)]} \\
x_1\ar[r]_-{[g,b ]}
 & x_2
}
$$
commutes in $\cl{X}\times_{\cl{X}\pq\cl{G}}\cl{X}$. To do that, we
have to show that there exists an isomorphism in $\cl{G}$,
$$
j: \un_{\ma(x_1)}\cdot g\lmap{\sim} g\cdot \un_{\ma(x_1\cdot g)},
$$
such that $\tar(j)=\mr{id}_{\ma(x_1)}$ and
\begin{equation}\label{eq:diag2}
\xymatrix{ x_1\cdot (\un_{\ma(x_1)}\cdot
g)\ar[r]^-{\beta}\ar[dd]_-{\mr{id}\cdot j}&
(x_1\cdot \un_{\ma(x_1)})\cdot g\ar[d]^-{\varepsilon\cdot \mr{id}}&\\
&x_1\cdot g\ar[r]^-{b}& x_2\\
x_1\cdot (g\cdot \un_{\ma(x_1\cdot g)})\ar[r]_-{\beta}&
(x_1\cdot g)\cdot \un_{\ma(x_1\cdot g)}\ar[u]_-{\varepsilon}&\\
}
\end{equation}
commutes. Since $\un_{\ma(x_1)}=\un_{\tar(g)}$ and
$\un_{\ma(x_1\cdot g)}=\un_{\sour(g)}$, we can define
$$
j:= (\rho g)^{-1}\cdot \lambda g: \un_{\ma(x_1)}\cdot g\lmap{\sim}
g\lmap{\sim} g\cdot \un_{\ma(x_1\cdot g)}.
$$
Since $\rho (g)$ and $\lambda (g)$ are sent to the identity by
$\pi_\G$, it follows that $\tar(j)=\mr{id}_{\ma(x_1)}$. The
commutativity of the diagram \eqref{eq:diag2}  is shown by the
diagram
 $$
 \xymatrix{
x_1\cdot (\un_{\ma(x_1)}\cdot
g)\ar[r]^-{\beta}\ar[dd]_-{\mr{id}\cdot j} \ar[rd]_-{\mr{id}\cdot
\lambda}&
(x_1\cdot \un_{\ma(x_1)})\cdot g\ar[d]^-{\varepsilon\cdot \mr{id}}&\\
&x_1\cdot g\ar[r]^-{b}& x_2\\
x_1\cdot (g\cdot \un_{\ma(x_1\cdot
g)})\ar[r]_-{\beta}\ar[ru]^-{\mr{id}\cdot \rho}&
(x_1\cdot g)\cdot \un_{\ma(x_1\cdot g)}\ar[u]_-{\varepsilon}&\\
}
$$
noticing that the upper and lower triangles commute by the higher
coherences $(x1g)$ and $(xg1)$ of ${\bf (a4)}$ in
Def.~\ref{defaction}. This proves that $(c_1,c_2)$ is a morphism in
$\cl{X}\times_{\cl{X}\pq\cl{G}}\cl{X}$, hence an isomorphism since
$c_1$ and $c_2$ are over an identity map in $\mathcal{C}$. This
concludes the proof that $\Q$ is essentially surjective.
\end{proof}

%%%%%%%%%%%%%%%%%%%%%%%%%%%%%%%%%%%%%%%%%%%%%%%%%%%%%%%%%%%%%%%%%%%%%%%%%%%%%%%%%%%%%%%%

\section{Principal actions of stacky Lie
groupoids}\label{sec:principal}

This section concerns actions that give rise to stacky principal bundles.

\begin{definition}\label{def:principal}
If $\G$ is a stacky Lie groupoid and $\X$ is a differentiable stack,
we call a $\G$-action on $\X$ {\bf principal} if $\X/\G$ is a
differentiable stack and $\X$ is a principal $\G$-bundle over
$\X/\G$.
\end{definition}

Since we know that $\G$ always acts on $\X$ on the fibres of the
quotient map $\q: \X\ra \cl{X}/\cl{G}$ (see Remark~\ref{ex:actfib}),
the previous definition amounts to showing that $\X/\G$ is a
differentiable stack such that $\q$ is an epimorphism and a
submersion, and that the natural map \eqref{eq:actfibre},
\begin{equation}\label{eq:apquot}
\cl{X} \times_{M} \cl{G}\map \cl{X} \times_{\X/\G} \cl{X},
\end{equation}
is an isomorphism.

%\comment{extended paragraph below (further explanations were added to the introduction)}

We will provide in this section a simple characterization of
principal actions, which could be thought of as parallel to the characterization of principal actions in the
smooth category by free and proper actions.

%%%%%%%%%%%%%%%%%%%%%%%%%%%%%%%%%%%%%%%%%%%%%%%%%%%%%%%%%%%%%%%%
\subsection{The main theorem: characterization of principal actions}\label{subsec:main}

Recall the notion of weak representability from
Def.~\ref{def:weakep}. The following is our main result.

\begin{theorem}\label{TheDiagonalThenPrincipal}
Let $\G\rra M$ be a stacky Lie groupoid acting (on the right) on a
differentiable stack $\X$ along $\ma: \X\to M$.
\begin{itemize}
\item[(a)] If the action-projection map $\Delta: \X\times_M\G\ra \X\times \X$
is weakly representable, then the $\G$-action on $\X$ is principal.

\item[(b)] Conversely, suppose $\X$ is a principal $\G$-bundle over
a differentiable stack $\cl{S}$. Then the action-projection map of
the action is weakly representable, and $\cl{S}$ is canonically
isomorphic to $\cl{X}/\cl{G}$.
\end{itemize}
\end{theorem}

\begin{proof}

\noindent{\bf Part (a)}: Notice that by Prop.~\ref{propwrf} we know
that the action-projection map $\Delta$ is automatically faithful,
so the quotient map $\q: \X \to \X/\G$ is an epimorphism as a
consequence of Propositions~\ref{prop1freeThenPrestack} and
\ref{propXtoX/Gepi/Gprestackthen}(a), and the map \eqref{eq:apquot}
is an isomorphism as a result of Propositions~ \ref{propQFaithful}
and \ref{propXtoX/Gepi/Gprestackthen}(b).
So, to conclude part (a), it remains to show that $\X/\G$ is a
differentiable stack so that $\q: \X \to \X/\G$ is a submersion.

Let $ G\rra G_0$ and $X\rra X_0$ be Lie groupoids presenting $\G$
and $\X$. We will prove that the morphism
\begin{equation}\label{eq:f}
X_0\ra \cl{X}/\cl{G},
\end{equation}
given by the composition $X_0\ra \X\ra \cl{X}/\cl{G}$, is
representable. Assuming this fact, and recalling that the quotient
map $\q: \cl{X}\ra \cl{X}/\cl{G} $ is an epimorphism, we see that
\eqref{eq:f} is an epimorphism, so it is an atlas for
$\cl{X}/\cl{G}$, i.e., this quotient is differentiable. Moreover,
Prop.~\ref{prop:charactsub}, part (b), implies that $\q$ is a
submersion.

So we are left with verifying the representability of the morphism
\eqref{eq:f}, and according to Lemma \ref{LemmaBX} it suffices to
prove that:
\begin{enumerate}
\item The fibred product $X_0\times_{\cl{X}/\cl{G}} X_0$ is represented by a
manifold, say $V$.
\item The induced map $V\ra X_0$ is a submersion. (There are two such induced
maps; it is enough to prove the statement for one of them.)
\end{enumerate}

The proofs of conditions $1.$ and $2.$ above follow from Proposition
\ref{propGEFiberedProduct} and Lemmas~\ref{lemmaXGXEFiberedProduct}
and \ref{lemmaIsoXtimesGE}, presented in Sections
\ref{subsecTechnicalLemmas} and \ref{SubsecCartSquare} below. We
outline the steps.

To prove $1.$, let $E$ be the (total space of the) Hilsum-Skandalis
bibundle associated with the action map $\X\times_M \G \to \X$.
Recall that the fact that the action-projection map  is faithful
implies that the action is 1-free. So we can use Proposition
\ref{propGEFiberedProduct} below to conclude that $E$ carries a left
action of the Lie groupoid $G \rra G_0$ in such a way that the
quotient stack $[G\backslash E]$ is described as a fibred product:
$$
\xymatrix{
[G\backslash E]\ar[r]\ar[d]&X_0\ar[d]\\
X_0 \ar[r] &\cl{X}/\cl{G}. }
$$
The conclusion in $1.$ follows if we show that $[G\bs E]$ is
representable.

Lemma \ref{lemmaXGXEFiberedProduct} shows that we have a 2-cartesian
diagram
$$
\xymatrix{[X\times_M G\bs X\times_{X_0} E]\ar[r]\ar[d]& X_0\times X_0\ar[d]\\
\cl{X}\times_M \cl{G}\ar[r]_-{\Delta}& \cl{X}\times\cl{X}, }
$$
while Lemma \ref{lemmaIsoXtimesGE} gives an isomorphism
$$
[X\times_M G\bs X\times_{X_0} E]\lmap{\sim} [G\bs E].
$$
From the assumption that $\Delta$ is weakly representable, it
follows that $[X\times_M G\bs X\times_{X_0} E]$ is representable,
and hence so is $[G\bs E]$. This concludes the proof of $1.$

In order to prove $2.$, we will see in
Section~\ref{subsecTechnicalLemmas} that we have a commutative
diagram
$$
\xymatrix{
E\ar[rd]_-{}\ar[r]& [G\bs E]\ar[d]^-{}\\
& X_0 }
$$
in which the map $E\to X_0$ is a submersion, see Remark
\ref{remtpasubmersion}. Since the horizonal map is surjective (as a
map between manifolds), it follows that the vertical map is a
submersion, as desired.

%%%%

\smallskip

\noindent{\bf Part (b):} In order to verify that the
action-projection map $\Delta: \X \times_M \G \to \X\times \X$ is
weakly representable, let $U$ be a manifold, and let $U\ra \cl{X}$
be a representable morphism (e.g. an atlas). We have a cartesian
diagram
$$
\xymatrix{U\times_\cl{S} U\ar[r]\ar[d]& U\times U\ar[d]\\
\cl{X}\times_\cl{S} \cl{X}\ar[r]_-{}& \cl{X}\times\cl{X}. }
$$
Since $\cl{X}\ra \cl{S}$ is assumed to be a submersion, the
composition $U\ra \cl{S}$ is representable, so $Z:= U\times_\cl{S}
U$ is representable. Using that the map $ \cl{X}\times_M\cl{G}\ra
\cl{X}\times_\cl{S} \cl{X}$ (induced by $\Delta$) is an isomorphism,
we see that $Z$ fits into a cartesian diagram
$$
\xymatrix{Z\ar[r]\ar[d]& U\times U\ar[d]\\
\cl{X}\times_M \cl{G}\ar[r]_-{\Delta}& \cl{X}\times\cl{X}, }
$$
and we conclude that $\Delta$ is weakly representable.

The assertion that $\cl{S}$ is canonically identified with $\X/\G$
is proven in Cor.~\ref{cor:principal}.
\end{proof}

An important instance of Theorem~\ref{TheDiagonalThenPrincipal} is
when the stacky Lie groupoid $\G\rra M$ is representable, i.e.,
isomorphic to an ordinary Lie groupoid. In this case the conditions
in the theorem are automatically satisfied:

\begin{corollary}\label{cor:ginot1}
Let $\G\rra M$ be a Lie groupoid, and suppose that it acts on a
differentiable stack $\X$. Then $\X/\G$ is a differentiable stack
and the $\G$-action on $\X$ is principal over it.
\end{corollary}

\begin{proof}
We must check that the action-projection map is weakly
representable. Note that by Lemmas~\ref{lemmaXGXEFiberedProduct} and
\ref{lemmaIsoXtimesGE} (cf. proof above), it suffices to show that
$[G\backslash E]$ is representable, where $E$ is the
Hilsum-Skandalis bibundle presenting the action map $\X\times_\G
\G\to \X$, and $G\rra G_0$ presents $\G$. By the assumption on $\G$
being representable, we may take $G=G_0=\G$ (the trivial groupoid),
in which case $[G\backslash E] = E$.
\end{proof}

When the Lie groupoid in the previous corollary is a Lie group, this
recovers \cite[Thm.~0.2]{ginotnoohi}. When both $\G$ and $\X$ are
representable, we are in the situation of an ordinary Lie groupoid
action of $G\rra M$ on a manifold $X$, in which case $\X/\G$ is the
quotient stack $[X/G]$, and we recover the well-known fact that $X$
is a principal $G$-bundle over it (see Example~\ref{ex:princBG}).

%%%%%%%%%%%%%%%%%%%%%%%%%%%%%%%%%%%%%%%%%%%%%%%%%%%%%%%%%%%%%%%%%%%%%%%%%%%%%%%%%
\subsection{Proofs of lemmas}\label{subsecTechnicalLemmas}

We now present the lemmas used to prove
Theorem~\ref{TheDiagonalThenPrincipal}.

Let a stacky Lie groupoid $\cl{G}\rra M$ act (on the right) on the
differentiable stack $\X$. As in Section~\ref{subsec:main}, let
$G\rra G_0$ and $X\rra X_0$ be presentations of $\cl{G}$ and
$\cl{X}$, respectively. We identify $\cl{G}$ with $BG$ and $\cl{X}$
with $BX$, and use the following notation: the action of $BG$ on
$BX$ is along $\p : BX\to M$, and $\sour,\tar : BG \to M$ are the
source and target maps.
The compositions of these structure maps with the atlas maps
$X_0\rightarrow BX$ and $G_0\rightarrow BG$  give rise to groupoid
morphisms
\begin{equation}\label{eq:grpm}
\xymatrix{
X\ar[r]^-{\bar{p}}\ar@<0.5ex>[d]^-{t_X}\ar@<-0.5ex>[d]_-{s_X}&
M\ar@<0.5ex>[d]\ar@<-0.5ex>[d]&&
G\ar[r]^-{\bar{s}}\ar@<0.5ex>[d]^-{t_G}\ar@<-0.5ex>[d]_-{s_G}&
M\ar@<0.5ex>[d]\ar@<-0.5ex>[d]&&
G\ar[r]^-{\bar{t}}\ar@<0.5ex>[d]^-{t_G}\ar@<-0.5ex>[d]_-{s_G}&
M\ar@<0.5ex>[d]\ar@<-0.5ex>[d]\\
X_0 \ar[r]_-{p} &M&& G_0\ar[r]_-{s}&M&&G_0\ar[r]_-{t}&M
}
\end{equation}
where $M\rra M$ is the trivial groupoid, and the following
equalities hold:
$$
\bar{p}=p \circ s_X=p \circ t_X,\qquad \bar{s}=s \circ s_G=s \circ
t_G,\qquad \bar{t}=t\circ s_G=t\circ t_G.
$$

The fibred product
$$
BX\times_M BG:=BX \times_{\p, M, \tar} BG
$$
is presented by the groupoid $X \times_{\bar{p}, M, \bar{t}} G\rra
X_0 \times_{p, M, t} G_0$ (the ``1-categorical'' fibred product, as
in \cite{MM}), that we write more simply as
$$
X\times_M G \rra  X_0\times_M G_0.
$$
The action morphism
$$
BX\times_M BG\map BX
$$
is presented by a right principal bibundle that we denote by
\begin{equation}\label{eq:Eab}
\xymatrix{
X\times_M G \ar@<0.5ex>[d]\ar@<-0.5ex>[d] &&X\ar@<0.5ex>[d]\ar@<-0.5ex>[d]\\
X_0\times_M G_0 &E\ar[l]^-{a}\ar[r]_-{b}&X_0. }
\end{equation}

Let $p_t$ and $t_p$ denote the following natural maps:
\begin{equation}\label{eq:pttp}
\xymatrix{
 X_0\times_M G_0\ar[r]^-{p_t}\ar[d]_-{t_p}& G_0\ar[d]^t\\
X_0\ar[r]_-{p}& M. }
\end{equation}
We will consider a $G$-action on $E$ along the map $p_t  a:E\to
G_0$, given by
\begin{equation}\label{eq:GE}
g\cdot e:=(1_{t_pa(e)},g)\cdot e,
\end{equation}
for $g\in G$ and $e\in E$ with $s_G(g)=p_ta(e)$, where on the
right-hand side we use the $(X\times_M G)$-action on $E$.

Since $G$ acts on the fibers of $b:E\rightarrow X_0$ and
$t_pa:E\rightarrow X_0$, we have induced morphisms between stacks,
$$ \tilde{b}:[G\backslash E]\map X_0,\;\;\;
\widetilde{t_pa}:[G\backslash E]\map X_0,
$$
fitting into the following commutative triangles:
$$
\xymatrix{
E\ar[rd]_-{b}\ar[r]& [G\bs E]\ar[d]^-{\tilde{b}}\\
& X_0 }\qquad\qquad \xymatrix{
E\ar[rd]_-{t_pa}\ar[r]& [G\bs E]\ar[d]^-{\widetilde{t_pa}}\\
& X_0.
}
$$
More explicitly, the induced morphisms are defined as follows. Given
an object $(P,c)$ in $[G\bs E]$, i.e., a left principal $G$-bundle
$P$ and a $G$-equivariant morphism $c:P\map E$,
$$ \xymatrix{
P\ar[d]\ar[r]^-{c}&E\\
U& }$$ then $\tilde{b}(P,c):U\rightarrow X_0$ is the unique morphism
that makes the following diagram commute:
$$ \xymatrix{
P\ar[d]\ar[r]^-{c}&E\ar[d]^-{b}\\
U\ar[r]& X_0 . }
$$
Given a morphism in $[G\bs E]$,
$$
\xymatrix{
P\ar[d]\ar[r]&P'\ar[d]\\
U\ar[r]_-{\mu}&U' }
$$
its image under $\tilde{b}$ is given by $\mu$. The morphism
$\widetilde{t_pa}$ is defined similarly.

\begin{remark}\label{remtpasubmersion}
Note that $t_pa:E\ra X_0$ is a submersion since both $a$ and $t_p$
are (as a result of $E$ being right principal and  $t_p$ being a
base change of $t:G_0\ra M$, which is a submersion since $\tar
:\cl{G}\ra M$ is a submersion).
\end{remark}

The next proposition is the first key result in the proof of part
(a) of Theorem~\ref{TheDiagonalThenPrincipal}.

\begin{proposition}\label{propGEFiberedProduct}
If  the action of $BG$ on $BX$ is 1-free, then there is a canonical
2-isomorphism making the following square 2-cartesian:
$$\xymatrix{
[G\backslash E]\ar[r]^-{\tilde{b}}\ar[d]_-{\widetilde{t_pa}}&X_0\ar[d]\\
X_0 \ar[r] &BX/BG.
}$$
%(more precisely, there is a canonical
%2-isomorphism that makes the square 2-cartesian).
\end{proposition}

Since the proof of this result is lengthy, we will present it
separately in Section~\ref{SubsecCartSquare}.

%\begin{proof}
%The proof is given in Subsection \ref{SubsecCartSquare}.
%\end{proof}

We now consider a (left) action of the Lie groupoid $X\times_M G\rra
X_0\times_M G_0$ on
$$
X\times_{X_0} E:= X\times_{t_X, X_0, t_p a} E
$$
along the map
$$
X\times_{X_0} E \map  X_0\times_{M} G_0, \qquad (x,e)\mapsto a(e),
$$
as follows:
$$
(x',g)\cdot(x,e)=(x' x, (x',g)e).
$$
The space $X\times_{X_0} E$ also carries a right action of $X\times
X$, along the map $(s_X,b)$, by $(x,e)\cdot (x_1,x_2)=(xx_1, ex_2)$,
making it into a bibundle
\begin{equation}\label{eq:bibundle}
\xymatrix{
X\times_M G \ar@<0.5ex>[d]\ar@<-0.5ex>[d] &&X\times X\ar@<0.5ex>[d]\ar@<-0.5ex>[d]\\
X_0\times_M G_0 &X\times_{X_0}E\ar[l]^-{}\ar[r]_-{}&X_0\times X_0. }
\end{equation}

\begin{lemma}\label{lemmaHSofDiagonal}
The bibundle \eqref{eq:bibundle} defines a Hilsum-Skandalis map
corresponding to the action-projection map $\Delta:\cl{X}\times_M
\cl{G}\rightarrow \cl{X}\times\cl{X}$.
%The Hilsum-Skandalis bibundle of the diagonal of the action:
%$\Delta:\cl{X}\times_M \cl{G}\rightarrow \cl{X}\times\cl{X}$, is
%given by:
%$$ \xymatrix{
%X\times_M G \ar@<0.5ex>[d]\ar@<-0.5ex>[d] &&X\times X\ar@<0.5ex>[d]\ar@<-0.5ex>[d]\\
%X_0\times_M G_0 &X\times_{X_0}E\ar[l]^-{}\ar[r]_-{}&X_0\times X_0
%}$$
%where the action of $X\times_M G$ is described above and the action of
%$X\times X$ is given by: $(x,e)\cdot (x_1,x_2)=(xx_1, ex_2)$.
\end{lemma}

\begin{proof}
The proof follows from the fact that the Hilsum-Skandalis bibundle
corresponding to the projection $\X\times_M \G\ra \X$ is
$$ \xymatrix{
X\times_M G \ar@<0.5ex>[d]\ar@<-0.5ex>[d] &&X\ar@<0.5ex>[d]\ar@<-0.5ex>[d]\\
X_0\times_M G_0 &G_0\times_{M}X\ar[l]^-{}\ar[r]_-{}&X_0, }
$$
as described in Lemma~\ref{lem:HSfacts1}. The result is now an
application of Lemma~\ref{lem:HSfacts2}.
\end{proof}

\begin{lemma}\label{lemmaXGXEFiberedProduct}
There is a canonical 2-cartesian square
$$
\xymatrix{[X\times_M G\bs X\times_{X_0} E]\ar[r]\ar[d]& X_0\times X_0\ar[d]\\
\cl{X}\times_M \cl{G}\ar[r]_-{\Delta}& \cl{X}\times\cl{X}.
}$$
\end{lemma}

\begin{proof}
The proof follows from Proposition ~\ref{prop2BXfibprodY0} and
Lemma~\ref{lemmaHSofDiagonal}.
\end{proof}

\begin{lemma}\label{lemmaIsoXtimesGE}
There is an isomorphism
$$
[X\times_M G\bs X\times_{X_0} E]\lmap{\sim}  [G\bs E].
$$
\end{lemma}

\begin{proof}
Both quotient stacks are presented by the translation groupoids  of
the corresponding actions. The translation groupoid presenting
$[X\times_M G\bs X\times_{X_0} E]$ is
$$
(X\times_M G)\underset{X_0\times_M G_0}{\times}
(X\times_{X_0} E)\rra X\times_{X_0} E
$$
where source and target maps are given, respectively, by
$$
(x',g,x,e)\stackrel{}\mapsto (x,e),\;\;\;
(x',g,x,e)\stackrel{}\mapsto (x' x,(x',g)e)
$$
and multiplication by
$$
(x_1',g_1,x_1,e_1)\cdot (x_2',g_2,x_2,e_2)=(x_1'
x_2',g_1g_2,x_2,e_2).
$$
The translation groupoid of $[G\bs E]$ is $ G\times_{G_0}
 E\rra  E$, with source and target maps given, respectively, by
$$
(g,e)\stackrel{}\mapsto e,\;\;\; (g,e)\stackrel{}\mapsto (1,g)e,
$$
and multiplication
$$
(g,e)\cdot (g',e')=(gg',e').
$$
The assertion in the lemma  follows from the fact that there is a
weak equivalence (as in \cite[Sec.~5.4]{MM}, see
also comments below Eq. (\ref{eq:HSbb}))
between the two translation groupoids,
$$
(X\times_M G)\underset{X_0\times_M G_0}{\times} (X\times_{X_0} E)
\to G\times_{G_0} E,
$$
given on objects by $(x,e)\mapsto (x^{-1},1)e$ (using the action of
$X\times_M G$ on $E$), and on arrows by $(x',g,x,e)\mapsto (g,
(x^{-1},1)e)$.
%, which is defined by (from LHS to RHS):
%$$ \begin{array}{ll}
%\mbox{Objects:}& (x,e)\mapsto (x^{-1},1)e\\\noalign{\medskip}
%\mbox{Morphisms:}& (y,g,x,e)\mapsto (y, (x^{-1},1)e).
%\end{array}$$
\end{proof}

%%%%%%%%%%%%%%%%%%%%%%%%%%%%%%%%%%%%%%%%%%%%%%%%%%%%%%%%%%%%%%%%%%%%%%%%%%%%%%%%%%%%%%%

\subsection{Proof of Proposition~\ref{propGEFiberedProduct}}\label{SubsecCartSquare}

%This subsection is dedicated to  the proof of Proposition \ref{propGEFiberedProduct}.
%We use the notation of Subsection \ref{subsecTechnicalLemmas}.

For the proof of Prop.~\ref{propGEFiberedProduct}, we keep the
notation of Section~\ref{subsecTechnicalLemmas}.

Since we assume that the action of $\G=BG$ on $\X=BX$ is 1-free, it
follows from Propositions~\ref{propXtoX/Gepi/Gprestackthen}$(b)$ and
\ref{prop1freeThenPrestack} that
$$
X_0\times_{BX/BG} X_0=X_0\times_{BX\pq BG} X_0.
$$
We will work with the latter stack, for which we have a more
explicit description: objects are triples $(\widehat{a},
[Q,\varphi], \widehat{b})$, where $\widehat{a}, \widehat{b}: U\to
X_0$, $U$ is a manifold, and $[Q,\varphi]: \widehat{a}^*X \to
\widehat{b}^*X$ is a morphism over $U$ in the prequotient $BX\pq BG$
(as in \eqref{eq:morpre}); recall, in particular, that $Q$ is a
principal $G$-bundle over $U$ such that $\widehat{a}^*X\cdot Q$ is
defined ($\widehat{a}^*X\cdot Q$ is the $X$-bundle defined by the
action $BX\times_M BG\to BX$), and $\varphi: \widehat{a}^*X\cdot Q
\to \widehat{b}^*X$ is an isomorphism of principal $X$-bundles over
$U$. We keep the notation $\widehat{a}^*X$ for the pullback by
$\widehat{a}$ of $X$ viewed as a principal $X$-bundle through
multiplication on the right (similarly for $\widehat{b}^*X$).
Morphisms in $X_0\times_{BX\pq BG} X_0$ are pairs
$$
(\mu_a,\mu_b): (\widehat{a}, [Q,\varphi], \widehat{b}) \to
(\widehat{a}_1, [Q_1,\varphi_1], \widehat{b}_1)
$$
where $\mu_a, \mu_b: U\to U_1$ are such that
$\widehat{a}=\widehat{a}_1 \mu_a$, $\widehat{b}=\widehat{b}_1 \mu_b$
and the diagram
$$
\xymatrix{
\hat{a}^*X\ar[rr]^-{[Q,\varphi]}\ar[d]&&\hat{b}^*X\ar[d]\\
\hat{a}_1^*X\ar[rr]_-{[Q_1,\varphi_1]}&&\hat{b}_1^*X }
$$
commutes in $BX\pq BG$ (the vertical maps in the diagram are those
naturally induced by $\mu_a$ and $\mu_b$).

Consider the Hilsum-Skandalis bibundle $E$ corresponding to the
action map, see Section~\ref{subsecTechnicalLemmas}. Let $[G\bs
E]_p$ be the subcategory of $[G\bs E]$ as in
Example~\ref{ex:quotientst}, which is a prestack whose
stackification is $[G\bs E]$.

 The proof of
Prop.~\ref{propGEFiberedProduct} consists of defining categories
fibred in groupoids and morphisms,
$$
\cl{A}_1\map\cl{A}_2\map \cl{A}_3:= X_0\times_{BX\pq BG} X_0,
$$
such that
\begin{itemize}
\item $\cl{A}_1$ is canonically isomorphic to $[G\bs E]_p$,
\item the morphism $\cl{A}_1\ra\cl{A}_2$ is a (strict) isomorphism,
\item $\cl{A}_2\ra \cl{A}_3$ is a monomorphism and an epimorphism.
\end{itemize}
By Prop.~\ref{PropStackification} (iv), it follows that the
stackification $[G\bs E]\ra X_0\times_{BX\pq BG} X_0$ of the induced
morphism $[G\bs E]_p\ra X_0\times_{BX\pq BG} X_0$ is an isomorphism,
and this proves the proposition.

\subsubsection*{The category $\cl{A}_1$}

The objects of $\cl{A}_1$ are triples $(U,f,c)$, where $U$ is a
manifold and $f:U\ra G_0$ and $c:U\ra E$ are morphisms such that
$p_t  a  c=f$, for $a$ in \eqref{eq:Eab} and $p_t$ in
\eqref{eq:pttp}. A morphism
\begin{equation}\label{eq:hv1}
(\mu,\nu):(U,f,c)\ra (U_1,f_1,c_1)
\end{equation}
is defined by maps $\mu: U\ra U_1$ and $\nu: U\ra G$ such that $s_G
\nu = f_1\mu$, $t_G \nu =f$ and $\nu \cdot(c_1 \mu)=c$. Here
``$\cdot$" denotes the action of $G$ on $E$ (the same notation is
used for multiplication in $G$). The composition of morphisms is
given by
$$
(\mu_1,\nu_1)(\mu, \nu)=(\mu_1\mu, \nu\cdot (\nu_1\mu)).
$$
The identity of an object $(U,f,c)$ is given by $(\mr{id}_U, 1_G
f)$. The category $\cl{A}_1$ is isomorphic (as a fibred category) to
 $[G\bs E]_p$ (cf. Example~\ref{ex:quotientst}, replacing $X$ by $E$ and $a$ by $p_t  a$).

\subsubsection*{The category $\cl{A}_2$}
%\noindent
%\textbf{The category $\cl{C}$.}

An object in $\cl{A}_2$ is of the form
$(U,\widehat{a},\widehat{b},f,\varphi)$, where $U$ is a manifold,
$\widehat{a},\widehat{b}:U\ra X_0$, $f:U\ra G_0$ are morphisms such
that $p  \widehat{a}=t f$, and $\varphi: \hat{a}^*X\cdot f^*G\ra
\hat{b}^*X$ is an isomorphism in $BX(U)$. A morphism in $\cl{A}_2$
from $A=(U,\hat{a},\hat{b},f,\varphi)$ to
$A_1=(U_1,\hat{a}_1,\hat{b}_1,f_1,\varphi_1)$ is a pair
\begin{equation}\label{eq:hv2}
(\mu, \nu): A\ra A_1
\end{equation}
where $\mu:U\ra U_1$, $\nu:U\ra G$ and $\widehat{a}_1 \mu
=\widehat{a}$, $\widehat{b}_1 \mu=\widehat{b}$, $s_G \nu =f_1 \mu$,
$t_G \nu =f$, and
\begin{equation}\label{eq:diagmorp}
\xymatrix{
\widehat{a}^*X\cdot f^*G\ar[r]^-{\varphi}\ar[d]& \hat{b}^*X\ar[d]\\
\widehat{a}_1^*X\cdot f_1^*G\ar[r]_-{\varphi_1}& \hat{b}_1^*X }
\end{equation}
commutes, where the map $f^*G = U\times_{G_0} G \ra f_1^*G =
U_1\times_{G_0}G$ is given by $(u,g)\mapsto (\mu(u),\nu(u)^{-1}g)$.
The maps $\hat{a}^*X\to \hat{a}_1^*X$ and $\hat{b}^*X\to
\hat{b}_1^*X$ are naturally induced by $\mu$. Composition and
identities in the category $\cl{A}_2$ are defined similarly to
$\cl{A}_1$.

%\noindent
%\textbf{Functor $\cl{B}\ra\cl{C}$: objects.}

\subsubsection*{The functor $\cl{A}_1\ra\cl{A}_2$}
%\begin{itemize}
%\item {\bf Objects:}

At the level of objects, we consider the following assignment
$\Obj(\cl{A}_1)\to \Obj(\cl{A}_2)$:
\begin{equation}\label{eq:objbij}
(U,f,c)\mapsto (U, \widehat{a}, \widehat{b}, f, \varphi),
\end{equation}
where $\widehat{a}=t_pac$ and $\widehat{b}=bc$; we now describe how
$\varphi$ is defined.
%This assignment turns out to be a bijection.

In order to define $\varphi:  \hat{a}^*X\cdot f^*G \to \hat{b}^*X$,
we use the fact that $\hat{a}^*X\cdot f^*G$ can be described as a
fibred product as follows. Let
$$
Y_0:= X_0\times_{p,M,t}G_0.
$$

\begin{lemma}
The principal (right) $X$-bundle $\hat{a}^*X\cdot f^*G$ over $U$ can
be identified with the left vertical arrow of the fibred product
$$
\xymatrix{ U\times_{Y_0} E \ar[r]^-{}\ar[d]
& E\ar[d]^-{a}\\
 U\ar[r]_-{(\hat{a},f)}& Y_0,
}
$$
where $U\times_{Y_0} E$ is equipped with an $X$-action along the map
$U\times_{Y_0} E\rightarrow X_0$, $(u,e)\mapsto b(e)$, given by
$(u,e)x=(u,ex)$.
\end{lemma}

\begin{proof}
The pair of principal bundles
$$
(\hat{a}^*X, f^*G)\in BX(U)\times_{M(U)}BG(U)
$$
is identified, via the isomorphism
$BX(U)\times_{M(U)}BG(U)=B(X\times_M G) (U)$ (see
Prop.~\ref{prop:Bprod}), with the principal $(X\times_M G)$-bundle
$\widehat{a}^*X\times_U f^*G$. Since $\widehat{a}^*X=U\times_{X_0}X$
and $f^*G=U\times_{G_0}G$, one can check the identification
$\widehat{a}^*X\times_U f^*G=U\times_{Y_0}Y$, where $Y=X\times_M G$.
Hence
$$
\widehat{a}^*X\cdot f^*G=(\widehat{a}^*X\times_U f^*G)\otimes_Y
E=((U\times_{Y_0}Y)\times_{Y_0} E)/Y=U\times_{Y_0} E.
$$
The reader may verify that the structure maps of the resulting
$X$-bundle are as in the statement.
\end{proof}

Since $\widehat{b}^*X$ is a trivial principal $X$-bundle, the
isomorphism $\varphi^{-1}: \widehat{b}^*X \to \hat{a}^*X\cdot f^*G$
is determined by a section $\sigma$ of $\hat{a}^*X\cdot f^*G$ over
$U$ such that $b d  \sigma=\widehat{b}$, where
$$
d: \hat{a}^*X\cdot f^*G = U\times_{Y_0} E  \to E
$$
is the natural projection (note that $bd$ is the moment map of
$\hat{a}^*X\cdot f^*G$). We now define $\varphi$ by declaring that
$\varphi^{-1}$ is determined by the unique section $\sigma$ of
$\hat{a}^*X\cdot f^*G$ over $U$ such that
$$
d \sigma=c.
$$

This last equation shows that $c$ and $\sigma$ uniquely determine
each other, hence \eqref{eq:objbij} is a bijection.

Given two objects $(U,f,c)$ and $(U_1,f_1,c_1)$ of $\cl{A}_1$, we
now consider a map from the set of morphisms between these two
objects and the set of morphisms between their images,
$(U,\hat{a},\hat{b},f,\varphi)$ and
$(U_1,\hat{a}_1,\hat{b}_1,f_1,\varphi_1)$, in $\cl{A}_2$.
The map sends a morphism $(\mu,\nu)$ in $\cl{A}_1$ (as in
\eqref{eq:hv1}) to $(\mu,\nu)$ in $\cl{A}_2$ (as in \eqref{eq:hv2}).
To check that the map is well defined and that it is a bijection,
let $(\mu,\nu)$ be such that $\mu:U\ra U_1$, $\nu: U\ra G$, and $s_G
\nu =f_1\mu$, $t_G \nu =f$. We claim that $\nu \cdot(c_1\mu)=c$ if
and only if $\widehat{a}_1\mu=\widehat{a}$,
$\widehat{b}_1\mu=\widehat{b}$, and the diagram \eqref{eq:diagmorp}
commutes. To verify that the condition $\nu\cdot(c_1\mu)=c$ implies
that $\widehat{a}_1\mu=\widehat{a}$, first notice that the
$G$-action on $E$ (see \eqref{eq:GE}) is on the fibers of $t_p a$.
Then we have
$$
\widehat{a}_1 \mu = t_p a c_1 \mu = t_p a (\nu \cdot c_1\mu)=t_p a
c=\widehat{a}.
$$
The fact that the condition $\nu\cdot(c_1\mu)=c$ implies that
$\hat{b}_1\mu=\hat{b}$ follows, similarly, from the fact that the
$G$-action on $E$ is on the fibers of $b$. It remains to prove that
if $\widehat{a}_1\mu=\widehat{a}$ and $\widehat{b}_1\mu=\widehat{b}$
hold, then $\nu \cdot(c_1\mu)=c$ is equivalent to the commutativity
of \eqref{eq:diagmorp}. We do that by writing the two maps from
$\hat{b}^*X$ to $\hat{a}_1^*X\cdot f_1^*G$ in diagram
\eqref{eq:diagmorp} in terms of the sections $\sigma$ and $\sigma_1$
that induce $\varphi^{-1}$ and $\varphi_1^{-1}$. Using the
identifications $\widehat{a}^*X\cdot f^*G=U\times_{Y_0} E$ and
$\widehat{a}_1^*X\cdot f_1^*G=U_1\times_{Y_0} E$, the two maps read:
$$
(u,x)\mapsto (\mu(u), c_1(\mu(u))x),\;\;\;\; (u,x)\mapsto
(\mu(u),\nu(u)^{-1}c(u)x),
$$
from which the claim follows. Here we use that the maps on the
square \eqref{eq:diagmorp} are given as follows: the vertical right
map is $(u,x)\mapsto (\mu(u),x)$, the map $\varphi^{-1}$ is
$\varphi^{-1}(u,x)=\sigma(u)x= (u,c(u)x)$, the map $\varphi^{-1}_1$
is $\varphi^{-1}_1(u_1,x)=\sigma_1(u_1)x= (u_1,c_1(u_1)x)$, and the
vertical left map is $(u,e)\mapsto (\mu(u), \nu(u)^{-1}e)$.

The conclusion is that the assignments on objects and morphisms just
described define a functor $\cl{A}_1\to \cl{A}_2$ which is a strict
isomorphism in $\CFG_\cl{C}$.

\subsubsection*{The functor $\Phi:\cl{A}_2\ra\cl{A}_3$}

Recall that $\cl{A}_3 =X_0\times_{BX\pq BG}X_0$. The functor
$\Phi:\cl{A}_2\ra\cl{A}_3$, at the level of objects, is defined by
%We want to define a functor $\Phi: \cl{C}\ra \cl{D}$ over the
%category of manifolds.
%On the objects $\Phi$ is defined by:
$$
\Phi(U,\hat{a},\hat{b},f,\varphi)=(\hat{a},[f^*G,\varphi],\hat{b}).
$$
At the level of morphisms, $\Phi$ sends a morphism $(\mu,\nu)$ in
$\cl{A}_2$, from $(U,\widehat{a},\widehat{b},f,\varphi)$ to
$(U_1,\widehat{a}_1,\widehat{b}_1,f_1,\varphi_1)$, to the pair
$(\mu,\mu)$; to verify that $(\mu,\mu)$ is indeed a morphism between
the corresponding images in $\cl{A}_3$, we need to show the
commutativity of the diagram
$$
\xymatrix{
\hat{a}^*X\ar[rr]^-{[f^*G,\varphi]}\ar[d]&&\hat{b}^*X\ar[d]\\
\hat{a}_1^*X\ar[rr]_-{[f_1^*G,\varphi_1]}&&\hat{b}_1^*X
}
$$
in $BX\pq BG$, where the vertical maps are induced by $\mu$. This
follows from the commutativity of \eqref{eq:diagmorp} and the higher
coherences $(x1g)$ and $(xg1)$ of the action.

One may check that $\Phi$ is a morphism of categories fibred in
groupoids, and we now verify that it is a monomorphism and an
epimorphism.

\begin{proposition} The following holds:
\begin{itemize}
\item[(a)] The morphism $\Phi:\cl{A}_2\ra \cl{A}_3$ is an epimorphism.

\item[(b)] If the action of $BG$ on $BX$ is 1-free then $\Phi:\cl{A}_2\ra
\cl{A}_3$ is a monomorphism, i.e., its restriction to fibers
$\Phi:\cl{A}_2(U)\ra \cl{A}_3(U)$ over any manifold $U$ is fully
faithful.

\end{itemize}
\end{proposition}

\begin{proof}
To show that $\Phi$ is an epimorphism, let
$(\hat{a},[Q,\varphi],\hat{b})$ be an object of $\cl{A}_3$ over $U$.
Then $Q$ is a principal right $G$-bundle over $U$ along some $\rho:
Q\to G_0$. Let $(\iota_\alpha: U_\alpha\ra U)_\alpha$  be an open
cover of $U$ trivializing $Q$. There are trivializations of $Q$
induced by maps $\sigma_\alpha: U_\alpha\ra Q$ such that $\pi_Q
\sigma_\alpha=\iota_\alpha$, where $\pi_Q : Q\ra U$ is the
projection. Define $f_\alpha:=\rho \sigma_\alpha:U_\alpha\ra G_0$.
The pullback of $\varphi:\hat{a}^*X\cdot Q\ra \hat{b}^*X$ via
$\iota_\alpha$ defines a morphism in $BX(U_\alpha)$,
$$
\varphi_{|U_\alpha}:(\hat{a}\iota_\alpha)^*X\cdot f_\alpha^*G\ra
(\hat{b}\iota_\alpha)^*X,
$$
and the restriction of $(\hat{a},[Q,\varphi],\hat{b})$ to $U_\alpha$
is $(\hat{a}\iota_\alpha,[f_\alpha^*G,
\varphi_{|U_\alpha}],\hat{b}\iota_\alpha)=
\Phi(U_\alpha,\hat{a}\iota_\alpha, \hat{b}\iota_\alpha, f_\alpha,
\varphi_{|U_\alpha})$.

To prove part $(b)$, let $A=(U,\hat{a},\hat{b},f,\varphi)$ and
$A_1=(U_1,\hat{a}_1,\hat{b}_1,f_1,\varphi_1)$ be objects of
$\cl{A}_2$ over $U$. The functor $\Phi$ induces a map
$$
\mr{Hom}_{\cl{A}_2(U)}(A,A_1)\map
\mr{Hom}_{\cl{A}_3(U)}(\Phi(A),\Phi(A_1)),
$$
and we must verify that this is a bijection. Note that $\cl{A}_3(U)$
is a set, i.e., all its morphisms are the identities. So it is
enough to prove the following two assertions:
\begin{itemize}
\item[(i)] If $\Phi(A)\neq\Phi(A_1)$, then there are no morphisms $A\ra
A_1$ over $U$.
\item[(ii)] If $\Phi(A)=\Phi(A_1)$, then there exists a unique morphism
$A\ra A_1$ over $U$.
\end{itemize}

The claim in (i) is straightforward: if there is a morphism $A\ra
A_1$ over $U$ then its image via $\Phi$ is a morphism
$\Phi(A)\ra\Phi(A_1)$ over $U$. Hence $\Phi(A)=\Phi(A_1)$, since
$\cl{A}_3(U)$ is a set.

We now prove (ii). Assume that $\Phi(A)=\Phi(A_1)$. This is
equivalent to having $\hat{a},\hat{b}:U\ra X_0$ and $f,f_1:U\ra G_0$
such that $tf=p\hat{a}=tf_1$, and morphisms $\varphi:\hat{a}^*X\cdot
f^*G\ra\hat{b}^*X$, $\varphi_1:\hat{a}^*X\cdot f_1^*G\ra\hat{b}^*X$
in $BX(U)$ such that $[f^*G,\varphi]=[f_1^*G,\varphi_1]$ as
morphisms $\hat{a}^*X\ra\hat{b}^*X$ in $BX\pq BG(U)$.

Since a morphism $A\ra A_1$ over $U$ is of the form
$(\mr{id}_U,\nu)$, we have to prove that there exists a unique
$\nu:U\ra G$ such that $s_G \nu=f_1$, $t_G \nu=f$, and the diagram
\begin{equation}\label{eq:diagbeta}
\xymatrix{ \hat{a}^*X\cdot
f^*G\ar[r]^-{\varphi}\ar[d]_-{\mr{id}\cdot \eta}
& \hat{b}^*X\ar[d]^-{\mr{id}}\\
\hat{a}^*X\cdot f_1^*G\ar[r]_-{\varphi_1}& \hat{b}^*X
}
\end{equation}
commutes, where  $\eta:f^*G\ra f_1^*G$ is the map described in the
definition of $\cl{A}_2$: $\eta(u,g)=(u,\nu(u)^{-1}g)$. The fact
that $[f^*G,\varphi]=[f_1^*G,\varphi_1]$ implies the existence of a
morphism $\eta:f^*G\ra f_1^*G$ of principal bundles over $U$ making
\eqref{eq:diagbeta} commute, and having such $\eta$ is equivalent to
having $\nu$ with the required properties, so that the existence of
$\nu$ is proved. For uniqueness, given $\nu, \nu_1$ with the desired
properties it follows that $\mr{id}\cdot
\eta=\varphi_1^{-1}\varphi=\mr{id}\cdot \eta_1$. Hence, since the
$BG$ action is 1-free, we deduce that $\eta_1=\eta$, and we conclude
that $\nu_1 = \nu$.
\end{proof}

%%%%%%%%%%%%%%%%%%%%%%%%%%%%%%%%%%%%%%%%%%%%%%%%%%%%%%%%%%%%%%%%
\subsection{Application: tensor product of stacky bundles}\label{SubsecTensorBundles}

As seen in Section~\ref{subsec:LG}, if $G$ is a Lie groupoid, $P$ is
a manifold carrying a principal right $G$-action, and $Q$ is
manifold carrying a left $G$-action, we can construct a new manifold
$P\otimes_G Q$. This construction is key in the definition of Morita
equivalence of Lie groupoids; in particular, it is used in showing
that Morita bibundles can be composed. We now extend this discussion
to the context of stacky Lie groupoids.

Let us consider the diagram
\begin{equation}\label{eq:Gaction}
\xymatrix{
& \cl{G}\ar@<0.5ex>[d]\ar@<-0.5ex>[d]&\\
\cl{X}\ar[r]&M &\ar[l]\cl{Y}, }
\end{equation}
where $M$ is a manifold, and $\cl{G}$ is a stacky Lie groupoid that
acts on the right on a differentiable stack $\cl{X}$ and on the left
on a differentiable stack $\cl{Y}$.  We assume that one of the maps
$\X\ra M$ or $\cl{Y}\ra M$ is a submersion, so that
$\X\times_M\cl{Y}$ is also a differentiable stack. Following
Example~\ref{ActionXX} (cf. Prop.~ \ref{prop:invact}) there is an
induced right action of $\cl{G}$ on $\cl{X}\times_M \cl{Y}$ given
(on objects) by
\begin{equation}\label{eq:inducedaction}
(x,y)\cdot g=(xg, g^{-1}y).
\end{equation}

Let us now discuss the property of weak representability of these
actions.

\begin{lemma}\label{lem:HSfacts3}
Let $E$ and $E'$ be Hilsum-Skandalis bibundles for the action maps
$\X \times_M \G\to \X$ and $\cl{Y}\times_M \G \to \cl{Y}$,
respectively. Then $E\times_{G_0}E'$ is a Hilsum-Skandalis bibundle
for the diagonal action $(\X \times_M \cl{Y})\times_M \G \to
\X\times_M \cl{Y}$.
\end{lemma}

\begin{proof}
Let $\X=BX$, $\cl{Y}=BY$, $\G=BG$. A Hilsum-Skandalis bibundle  for
the projection $BX\times_M BY \times_M BG \to BX\times_M BG$ is
$X\times_M G \times_M Y_0$, and a Hilsum-Skandalis bibundle for the
projection $BX\times_M BY \times_M BG \to BY\times_M BG$ is
$X_0\times_M G \times_M Y$ (see Lemma~\ref{lem:HSfacts1}). Hence a
Hilsum-Skandalis bibundle for the composition with the action map,
$$(\X \times_M \cl{Y})\times_M \G \to \X\times_M \cl{G} \to \X,
$$
is $(X\times_M G\times_M Y_0)\otimes_{X\times_M G} E$; similarly, a
Hilsum-Skandalis bibundle for $(\X \times_M \cl{Y})\times_M \G \to
\cl{Y}\times_M \cl{G} \to \cl{Y}$ is $(X_0\times_M G\times_M
Y)\otimes_{Y\times_M G} E'$. By Lemma~\ref{lem:HSfacts2}, a
Hilsum-Skandalis bibundle for the diagonal action is
$$
\big ((X\times_M G\times_M Y_0)\otimes_{X\times_M G} E\big )
\times_{X_0\times_M G_0 \times_M Y_0} \left ( (X_0\times_M G\times_M
Y)\otimes_{Y\times_M G} E'\right ) \cong E\times_{G_0}E',
$$
via $([(1,1,y_0),e],[(x_0,1,1),e'])\mapsto (e,e')$. In the last
formula it is understood that given elements $[(x,g,y_0), e]\in
(X\times_M G\times_M Y_0)\otimes_{X\times_M G} E$ and
$[(x_0,g,y),e']\in (X_0\times_M G\times_M Y)\otimes_{Y\times_M G}
E'$ we can always find representatives of the form $[(1,1,y_0),e]$
and $[(x_0,1,1),e']$, respectively.
\end{proof}

\begin{lemma}\label{LemmaProductWeakRep}
If the action-projection map of the action on $\cl{X}$ is weakly
representable, or the action-projection map of the action on
$\cl{Y}$ is weakly representable, then so is the action-projection
map of the action on $\cl{X}\times_M\cl{Y}$.
\end{lemma}

\begin{proof}
Since the inversion map on $\G$ is an isomorphism, so is the map
$\cl{Y}\times_{M,\tar}\cl{G}\map \cl{G}\times_{\sour,M}\cl{Y}$,
$(y,g)\mapsto (g^{-1},y)$. So the action-projection map $\Delta_l:
\G \times_M \cl{Y}\to \cl{Y}\times \cl{Y}$, $\Delta_l(g,y)=(gy, y)$,
of the left action on $\cl{Y}$ is weakly representable if and only
if the action-projection map $\Delta_r$, $(y,g)\mapsto (y,g^{-1}y)$,
of the induced right action is weakly representable as a consequence
of the commutativity of the following diagram:
$$
\xymatrix{ \cl{Y}\times_{M}\cl{G}\ar[r]\ar[d]_-{\Delta_r}&
\cl{G}\times_{M}\cl{Y}\ar[d]^-{\Delta_l}\\
\cl{Y}\times\cl{Y}\ar[r]& \cl{Y}\times\cl{Y}, }
$$
where the lower map is the isomorphism defined by $(y,y')\mapsto
(y',y)$.

Hence it is enough to prove the assertion in the lemma with the weak
representability assumption for $\cl{X}$. In this case, by Lemmas
\ref{lemmaXGXEFiberedProduct} and \ref{lemmaIsoXtimesGE}, we have
that $[G\bs E]$ is representable, where $E$ is a Hilsum-Skandalis
bibundle representing  the action map on $\cl{X}$, and $G\rra G_0$
is a groupoid presenting $\cl{G}$. In particular the action of $G$
on $E$ is free and proper. If $E'$ is a Hilsum-Skandalis bibundle of
the morphism $\cl{Y}\times_M\cl{G} \rightarrow \cl{Y}$ that sends
$(y,g)$ to $g^{-1}y$, then a Hilsum-Skandalis bibundle of the action
on $\cl{X}\times_M \cl{Y}$ is given by $E\times_{G_0}E'$ (see
Lemma~\ref{lem:HSfacts3}). The induced $G$-action on $E\times_{G_0}
E'$ is the diagonal one, and since the action on the factor $E$ is
free and proper, the induced action is free and proper as well.
Hence $[G\bs (E\times_{G_0}E')]$ is representable and the proof is
complete once we apply Lemmas~\ref{lemmaXGXEFiberedProduct} and
\ref{lemmaIsoXtimesGE} to the diagonal action.
\end{proof}

We now define the tensor product of stacky bundles.

\begin{proposition}\label{prop:tensor}
Let $\cl{G}\rra M$ be a stacky Lie groupoid acting on the
differentiable stacks $\cl{X}$ (on the right) and $\cl{Y}$ (on the
left). If the $\G$-action on $\X$ (or $\cl{Y}$) is principal (see
Def.~\ref{def:principal})
then the $\G$-action on $\cl{X}\times_M\cl{Y}$ makes it into a
principal $\G$-bundle over the quotient
$$
\cl{X}\otimes_\cl{G}\cl{Y}:=(\cl{X}\times_M \cl{Y})/\cl{G}.
$$
\end{proposition}

\begin{proof}
It directly follows from Lemma \ref{LemmaProductWeakRep} and Theorem
\ref{TheDiagonalThenPrincipal}.
\end{proof}

%%%%%%%%%%%%%%%%%%%%%%%%%%%%%%%%%%%%%%%%%%%%%%%%%%%%%%%%%%%%%%%%%%%%%%%%%%%%%%%%%%%

\section{Morita equivalence of stacky Lie groupoids}\label{sec:morita}

%\comment{introduced next paragraph}

As recalled in Section~\ref{subsec:LG}, Morita equivalence plays a key role in groupoid theory; e.g. it identifies groupoids that present isomorphic stacks, and a similar role should be played by Morita equivalence of higher groupoids in connection with higher stacks. In this section, we show how our results on principal actions find applications in the theory of Morita equivalence of stacky Lie groupoids, which we develop through stacky bibundles (a parallel theory for 2-groupoids is presented in \cite[Sec.~6]{duli}, relying on combinatorial and simplicial methods).

%%%%%%%%%%%%%%%%%%%%%%%%%%%%%%%%%%%%%%%%%%%%%%%%%%%%%%%%%%%%%%%%%%%%%%%%
\subsection{Stacky bibundles and Morita equivalence}

In order to define Morita equivalence of stacky Lie
groupoids, we start with the notion of bibundle.

\begin{definition}\label{DefBibundle}
Let $\cl{G}_i\rra M_i$, for $i=1,2$, be cfg-groupoids. A
$\cl{G}_1$-$\cl{G}_2$-\textbf{bibundle} is an object $\cl{X}$ in
$\CFG_{\cl{C}}$ equipped with maps $\ma_1:\X\to M_1$ and $\ma_2:
\X\to M_2$,
$$
\xymatrix{
\cl{G}_1\ar@<0.5ex>[d]\ar@<-0.5ex>[d]&&  \cl{G}_2\ar@<0.5ex>[d]\ar@<-0.5ex>[d]\\
M_1 &\cl{X} \ar[l]^-{\ma_1}\ar[r]_-{\ma_2}& M_2, }
$$
together with a left action $\act_1$ of $\cl{G}_1$ on  the fibers of
$\ma_2$, a right action $\act_2$ of $\cl{G}_2$ on the fibers of
$\ma_1$, and a 2-isomorphism $\tau: \act_1(\mr{id}\times
\act_2)\lmap{\sim} \act_2(\act_1\times \mr{id})$ that makes the two
actions commute. Moreover, $\tau$ is assumed to satisfy the
following higher coherence conditions:
$$ (g_1g_1'xg_2):\qquad\qquad\xymatrix{
((g_1g_1')x)g_2&(g_1g_1')(xg_2)\ar[l]_-{\tau}
& g_1(g_1'(xg_2))\ar[l]_-{\beta_1}\ar[d]^-{\mr{id}\cdot \tau}\\
(g_1(g_1'x))g_2\ar[u]^-{\beta_1\cdot
\mr{id}}&&g_1((g_1'x)g_2)\ar[ll]^-{\tau} }$$
$$(g_1xg_2g_2'):\qquad\qquad\xymatrix{
((g_1x)g_2)g_2'&(g_1x)(g_2g_2')\ar[l]_-{\beta_2}
& g_1(x(g_2g_2'))\ar[l]_-{\tau}\ar[d]^-{\mr{id}\cdot \beta_2}\\
(g_1(xg_2))g_2'\ar[u]^-{\tau\cdot
\mr{id}}&&g_1((xg_2)g_2')\ar[ll]^-{\tau} }$$
$$(1xg_2):\qquad \xymatrix{
 1(xg_2)\ar[r]^-{\tau}\ar[d]_-{\varepsilon_1}
&(1x)g_2\ar[ld]^-{\varepsilon_1\cdot \mr{id}}\\xg_2&}\qquad
(g_1x1):\qquad \xymatrix{
 g_1(x1)\ar[r]^-{\tau}\ar[d]_-{\mr{id}\cdot \varepsilon_2}
 &(g_1x)1\ar[ld]^-{\varepsilon_2}\\
 g_1x&
}
$$
where $g_1,g_1' \in \cl{G}_1$, $x \in \cl{X}$, $g_2,g_2' \in
\cl{G}_2$ are such that the compositions make sense, $1$ is the
appropriate groupoid identity, $\beta_1,\beta_2$ are the
associativity 2-isomorphisms of the two actions, and $\varepsilon_1,
\varepsilon_2$ are the identity 2-isomorphisms of the two actions
(see Def.~\ref{defaction}).
\end{definition}

\begin{definition} A $\cl{G}_1$-$\cl{G}_2$-bibundle $\cl{X}$ is \textbf{biprincipal}
if $\cl{G}_1$ and $\cl{G}_2$ are stacky Lie groupoids, $\cl{X}$ is a
differentiable stack, $\cl{X}\ra M_2$ is a principal left
$\cl{G}_1$-bundle, and $\cl{X}\ra M_1$ is a principal right
$\cl{G}_2$-bundle.
\end{definition}

\begin{remark}\label{rem:sHS}
Clearly, one could also consider just right-principal bibundles as
generalizations of Hilsum-Skandalis bibundles to the context of
stacky Lie groupoids (cf. Section~\ref{subsec:LG}).
\end{remark}

\begin{definition}\label{def:ME}
Two stacky Lie groupoids $\cl{G}_1$ and $\cl{G}_2$ are
\textbf{Morita equivalent} if there exists a biprincipal
$\cl{G}_1$-$\cl{G}_2$-bibundle.
\end{definition}

We may also refer to a biprincipal bibundle as a {\bf Morita
bibundle}.

%\comment{remark added}

\begin{remark} In groupoid theory, another approach to Morita equivalence, alternative to bibundles, is through
``zig-zags'' of weak equivalences (see e.g. \cite[Sec.~5.4]{MM}). The extension of this theory to higher groupoids in various categories (though not encompassing higher Lie groupoids) has been recently presented in \cite{bg}, and a parallel construction for higher Lie (Banach) groupoids is considered in \cite{RZ}. The first steps towards
checking the equivalence of both approaches  have been taken in \cite[Sect 6.4]{duli}, where it is shown that, for Lie 2-groupoids, a Morita bibundle arising from a strict morphism is exactly a weak equivalence, hence Morita
equivalence via zig-zag of weak equivalences gives rise to Morita equivalence via bibundles, as in this paper.
\end{remark}

\begin{example}\label{ExmGGG}
For any stacky Lie groupoid $\cl{G}$, we have that $\cl{G}$ itself
is a biprincipal $\cl{G}$-$\cl{G}$-bibundle. In particular, $\cl{G}$
is Morita equivalent to itself.
\end{example}

One can find interesting examples of Morita equivalence by
considering \'etale stacky Lie groupoids $\G\rra M$ associated with
non-integrable Lie algebroids, see Example~\ref{ex:integration}.

\begin{example}\label{ex:transitive}
Let us assume, for simplicity, that $\G$ is an \'etale stacky Lie
groupoid that is transitive (in the sense that it satisfies the
property in Corollary~\ref{cor:onto}). In this case,
Examples~\ref{ex:restric1} and \ref{ex:restric2} imply that, for
$x\in M$, $\sour^{-1}(x)$ is a biprincipal
$\cl{G}$-$\cl{G}_x$-bibundle (the commutativity of left and right
actions follows from the commutativity of left and right
multiplications on $\G$):
\begin{equation}\label{eq:ME}
\xymatrix{
\cl{G}\ar@<0.5ex>[d]\ar@<-0.5ex>[d]&&  \cl{G}_x\ar@<0.5ex>[d]\ar@<-0.5ex>[d]\\
M &\sour^{-1}(x) \ar[l]^-{\tar}\ar[r]_-{}& \{x\}. }
\end{equation}
Hence $\G$ is Morita equivalent to the stacky Lie group $\G_x$,
\end{example}

The previous example extends the well-known property that transitive
Lie groupoids are Morita equivalent to their isotropy groups.

The following section presents a concrete Morita equivalence of the
type given by \eqref{eq:ME}.

%%%%%%%%%%%%%%%%%%%%%%%%%%%%%%%%%%%%%%%%%%%%%%%%%%%%%%%%%%%%%%%
\subsection{Example from non-integrable Lie algebroids}\label{subsec:ex}

We consider here stacky Lie groupoids arising from the simplest
examples of non-integrable Lie algebroids.

Let $\omega\in \Omega^2(M)$ be a closed 2-form, and consider the Lie
algebroid $A_\omega = TM\oplus \mathbb{R}$, with anchor given by the
natural projection on $TM$ and bracket on $\Gamma(A_\omega) =
\mathcal{X}(M)\oplus C^\infty(M)$ given by
$$
[(X,f),(Y,g)]=([X,Y], \mathcal{L}_Xg - \mathcal{L}_Yf +
\omega(X,Y)).
$$
We assume for simplicity that $M$ is connected and simply-connected.

The integrability of this Lie algebroid is measured by the group of
periods
$$
\mathrm{Per}(\omega) := \left \{ \int_\sigma \omega, \; \sigma \in
\pi_2(M) \right \} \subset \mathbb{R}.
$$
As proven in \cite{cf}, $A_\omega$ is integrable if and only if this
group is discrete. Assuming that this is the case and setting
$S_\omega:=\mathbb{R}/\mathrm{Per}(\omega)$, it is shown in
\cite{crainic} that $\omega$ determines a principal
$S_\omega$-bundle $P_\omega$ over $M$ (when
$\mathrm{Per}(\omega)=\mathbb{Z}$, this is the $S^1$-bundle known as
the {\em prequantization} of $\omega$); the associated gauge
groupoid $(P_\omega\times P_\omega)/S_\omega$ is the
source-simply-connected Lie groupoid integrating $A_\omega$, denoted
by $G_\omega$, and we have a Morita bibundle
\begin{equation}\label{eq:gaugeME}
\xymatrix{
G_\omega\ar@<0.5ex>[d]\ar@<-0.5ex>[d]&&  S_\omega\ar@<0.5ex>[d]\ar@<-0.5ex>[d]\\
M &P_\omega \ar[l]^-{}\ar[r]_-{}& \{\ast\} . }
\end{equation}

Following \cite{crainic}, we recall the approach to integrate
$A_\omega$ via path spaces \cite{cf,severa}. We consider the Banach
manifold $P(M)$ of paths (of class $C^2$), and in the product
$P(M)\times \mathbb{R}$ we define an equivalence relation
$(\gamma,r) \sim (\overline{\gamma}, \overline{r})$ by the condition
that there is a homotopy $\sigma$ (with fixed end points) from
$\gamma$ to $\overline{\gamma}$ such that
$$
\overline{r}-r = \int_\sigma \omega.
$$
This equivalence relation defines a regular foliation ${F}_\omega$
on $P(M)\times \mathbb{R}$ whose leaf space is a smooth manifold
exactly when $A_\omega$ is integrable; moreover,
\begin{equation}\label{eq:Gw}
G_\omega = (P(M)\times\mathbb{R})/\sim.
\end{equation}
To complete the diagram \eqref{eq:gaugeME} from this perspective, we
fix $x\in M$ and consider the (Banach) submanifolds $P(M)_x\times
\mathbb{R}$ and $L(M)_x\times \mathbb{R}$ of $P(M)\times
\mathbb{R}$, where $P(M)_x$ is given by paths starting at $x$ and
$L(M)_x$ is given by loops based on $x$. Both submanifolds  are
saturated by the leaves of the foliation ${F}_\omega$; when
$A_\omega$ is integrable, the corresponding leaf spaces are smooth
and
\begin{equation}\label{eq:PL}
P_\omega = (P(M)_x\times \mathbb{R})/\sim\,,\;\;\; S_\omega =
(L(M)_x\times \mathbb{R})/\sim\, ,
\end{equation}
and the left and right actions on \eqref{eq:gaugeME} are induced by
concatenation of paths.

We will see how this picture extends to the non-integrable case,
i.e., we no longer assume that $\mathrm{Per}(\omega)\subset
\mathbb{R}$ is discrete (so it can be dense). In this case one may
still view the leaf space \eqref{eq:Gw} as the differentiable stack
$$
\G_\omega := B \mathrm{Hol}_\omega,
$$
where $\mathrm{Hol}_\omega \rra P(M)\times \mathbb{R}$ is the
holonomy groupoid of the foliation ${F}_\omega$ (more precisely, one
should restrict the holonomy groupoid to a complete transversal of
$F_\omega$, so as to obtain a finite-dimensional Lie groupoid);
moreover, it is shown in \cite{tz} that $\mathcal{G}_\omega$ is an
\'etale stacky Lie groupoid over $M$.

The fact that $M$ is connected implies that $\mathcal{G}_\omega$ is
transitive, so we are in the situation of
Example~\ref{ex:transitive}. By fixing $x\in M$, we have a stacky
Morita bibundle
\begin{equation}\label{eq:stackbb}
\xymatrix{
\mathcal{G}_\omega\ar@<0.5ex>[d]\ar@<-0.5ex>[d]&&  \mathcal{S}_\omega\ar@<0.5ex>[d]\ar@<-0.5ex>[d]\\
M &\mathcal{P}_\omega \ar[l]^-{}\ar[r]_-{}& \{x\}, }
\end{equation}
where $\mathcal{P}_\omega := \sour^{-1}(x)$ and $\mathcal{S}_\omega:
= (\mathcal{G}_\omega)_x$ is the isotropy stacky Lie group at $x$.
Note that when $A_\omega$ is integrable, all differentiable stacks
in the previous diagram are representable, and we recover
\eqref{eq:gaugeME}. We will now provide an explicit description of
the stacky Lie group $\cl{S}_\omega$.

\begin{lemma}
The holonomy groups of the foliation ${F}_\omega$ are trivial, i.e.,
any two paths in a leaf joining the same points have the same
holonomy.
\end{lemma}

\begin{proof}
Consider the submersion $q: P(M)\to M\times M$, $\gamma\mapsto
(\gamma(0),\gamma(1))$, and let $\gamma_\alpha \in P(M)$, with
endpoints $q(\gamma_\alpha)=(x_\alpha, y_\alpha)$. We may consider a
neighborhood $V_\alpha$ of $\gamma_\alpha$ in $P(M)$ which is
diffeomorphic to the product $U^0_\alpha \times U^1_\alpha \times
W_\alpha$, where $U^0_\alpha$ and $U^1_\alpha$ are open balls in $M$
centered at $x_\alpha$ and $y_\alpha$, and $W_\alpha$ is a convex
subset of a Banach space, in such a way that upon this
identification $q$ becomes the natural projection $U^0_\alpha \times
U^1_\alpha \times W_\alpha \to U^0_\alpha \times U^1_\alpha$.

Given $\gamma\in V_\alpha$, we can always find a free homotopy
$\sigma^\gamma_\alpha(s,t)$ from $\gamma(t)$ to $\gamma_\alpha(t)$
in such a way that the paths spanned by the endpoints,
$\sigma^\gamma_\alpha(s,0)$ and $\sigma^\gamma_\alpha(s,1)$, follow
the straight lines in $U_\alpha^0$ and $U_\alpha^1$ linking
$\gamma(0)$ to $\gamma_\alpha(0)$ and $\gamma(1)$ to
$\gamma_\alpha(1)$, respectively. Note that given another such
homotopy $\bar{\sigma}^\gamma_\alpha(s,t)$, with the same boundary
conditions, by the Stokes theorem we have that
$$
\int_{\sigma^\gamma_\alpha}\omega =
\int_{\bar{\sigma}^\gamma_\alpha}\omega.
$$

We now consider the submersion
$$
\psi_\alpha: V_\alpha \times \mathbb{R} \to M \times M\times
\mathbb{R}, \;\; (\gamma,r)\mapsto
(\gamma(0),\gamma(1),r+\int_{\sigma^\gamma_\alpha} \omega),
$$
where $\sigma^\gamma_\alpha$ is a homotopy from $\gamma$ to
$\gamma_\alpha$ as above.

\smallskip

\noindent{\bf Claim:} Two points in $V_\alpha \times \mathbb{R}$ are
in the same fiber of $\psi_\alpha$ if and only if they lie in the
same connected component of the intersection of a leaf of $F_\omega$
with $V_\alpha \times \mathbb{R}$.

\smallskip

To verify the claim, we first observe  that the condition for two
points $(\gamma_0,r_0)$ and $(\gamma_1, r_1)$ to belong to the same
connected component of the intersection of a leaf of $F_\omega$ with
$V_\alpha \times \mathbb{R}$ is equivalent to the existence of a
path $s\mapsto (\gamma_s,r_s)$ in $V_\alpha\times \mathbb{R}$ from
$(\gamma_0,r_0)$ to $(\gamma_1,r_1)$ such that
\begin{equation}\label{eq:rs}
r_s = r_0 + \int_{\sigma(\gamma_s)} \omega,
\end{equation}
where, for each $s$, $\sigma(\gamma_s)$ denotes the homotopy (with
fixed endpoints) from $\gamma_0$ to $\gamma_s$ defined by $s'\mapsto
\gamma_{s'}$, for $0\leq s'\leq s$. In this case, noticing that
$$
\int_{\sigma_\alpha^{\gamma_0}}\omega =
\int_{\sigma(\gamma_s)}\omega +
\int_{\sigma_\alpha^{\gamma_s}}\omega,
$$
we see that $\psi_\alpha(\gamma_s,r_s)$ is constant. On the other
hand, $\psi_\alpha(\gamma_0,r_0)=\psi_\alpha(\gamma_1,r_1)$ if and
only if $\gamma_0$ and $\gamma_1$ have the same endpoints (so they
are homotopic through a homotopy $\sigma$ lying in $V_\alpha$) and
$$
r_1 = r_0 + \int_{\sigma^{\gamma_0}_\alpha}\omega -
\int_{\sigma^{\gamma_1}_\alpha}\omega = r_0 + \int_\sigma \omega.
$$
Defining the path $(\gamma_s,r_s)$ by $\gamma_s(t)=\sigma(s,t)$ and
$r_s=r_0 + \int_{\sigma(\gamma_s)}\omega$, we see that
$(\gamma_0,r_0)$ and $(\gamma_1,r_1)$ are in the same connected
component of the intersection of a leaf of $F_\omega$ with
$V_\alpha\times \mathbb{R}$. This completes the proof of the claim.

From the previous claim, we conclude that, for a given $r_\alpha\in
\mathbb{R}$, we may view $\psi_\alpha$ as the projection on the
transverse direction of a foliated chart of $F_\omega$ around
$(\gamma_\alpha, r_\alpha)$. Taking two such foliated charts, around
two nearby points $(\gamma_\alpha,r_\alpha)$ and
$(\gamma_\beta,r_\beta)$ on the same leaf, the corresponding
holonomy transformation is given by the transverse component of the
transition function from chart $\alpha$ to chart $\beta$, which we
explicitly compute to be
$$
(x,y, r) \mapsto (x,y, r + \int_{\sigma^\gamma_\beta}\omega -
\int_{\sigma^\gamma_\alpha}\omega ) = (x, y, r+
\int_{\sigma}\omega),
$$
where $\sigma$ is any homotopy (with fixed endpoints) from
$\gamma_\alpha$ to $\gamma_\beta$ in $V_\alpha$. If now $s\mapsto
\xi(s)=(\gamma_s,r_s)$ is a path in a leaf of $F_\omega$, from
$(\gamma_0,r_0)$ to $(\gamma_1,r_1)$,  by subdividing $\xi$ and
iterating the holonomy transformation of nearby points above we see
that its holonomy transformation is
$$
(x,y,r) \mapsto (x,y, r+\int_{\sigma(\gamma_1)}\omega),
$$
recalling that $\sigma(\gamma_1)$ is the homotopy with fixed
endpoints from $\gamma_0$ to $\gamma_1$ defined by $s\mapsto
\gamma_s$. If $\xi$ is a loop, then $r_1=
r_0+\int_{\sigma(\gamma_1)}\omega$ must equal $r_0$, so
$\int_{\sigma(\gamma_1)}\omega = 0$. Hence, the holonomy
transformation is trivial.
\end{proof}

It follows from the previous lemma that the holonomy groupoid
$\mathrm{Hol}_\omega$ is given by pairs $(\xi,\bar{\xi}) \in
(P(M)\times\mathbb{R})\times (P(M)\times\mathbb{R})$ such that $\xi
\sim \bar{\xi}$ (with the natural pair groupoid structure).

As a differentiable stack, $\mathcal{S}_\omega$ is presented by the
restriction of $\mathrm{Hol}_\omega$ to $L(M)_x\times \mathbb{R}$,
\begin{equation}\label{eq:restrict}
\{(\xi,\bar{\xi}) \in (L(M)_x\times \mathbb{R})\times (L(M)_x\times
\mathbb{R})\,|\, \xi\sim \bar{\xi} \} \rra L(M)_x\times \mathbb{R}.
\end{equation}
Since $M$ is simply-connected, we see that the map $\mathbb{R} \to
L(M)_x\times \mathbb{R}$, $r\mapsto (c_x,r)$, where $c_x$ is the
constant loop based on $x$, is a complete transversal to the
foliation in $L(M)_x\times \mathbb{R}$, whose leaves are the orbits
of \eqref{eq:restrict}. Upon restriction of \eqref{eq:restrict} to
this transversal, we obtain a Morita equivalent Lie groupoid
presenting $\cl{S}_\omega$, given by
$$
\{(r,s)\in \mathbb{R}^2\;|\; (c_x,r)\sim (c_x,s)\} \rra \mathbb{R}.
$$
Noticing that $(c_x,r)\sim (c_x,s)$ if and only if there exists
$\sigma \in \pi_2(M,x)$ such that $r-s = \int_\sigma \omega$, we see
that the previous Lie groupoid agrees with the action groupoid
$\mathrm{Per}(\omega)\ltimes \mathbb{R}\rra \mathbb{R}$. Hence, as a
differentiable stack, $\cl{S}_\omega$ is the stack quotient
$[\mathbb{R}/\mathrm{Per}(\omega)]$. Its stacky Lie group structure
is the one induced by the inclusion of abelian groups
$\mathrm{Per}(\omega)\to \mathbb{R}$, as in Example~\ref{ex:BA}. We
summarize the previous discussion in the following

%\comment{conclusion added to statement of Prop.}

%\comment{higher prequantization of nonintegral 2-form}

\begin{proposition}
The stacky Lie group $\cl{S}_\omega$ is the strict 2-group
$[\mathbb{R}/\mathrm{Per}(\omega)]$ defined by the action groupoid
$\mathrm{Per}(\omega)\ltimes \mathbb{R}\rra \mathbb{R}$ and the
direct product of abelian groups $\mathrm{Per}(\omega)\times
\mathbb{R}$.

It follows (see \eqref{eq:stackbb}) that the stacky Lie
groupoid $\G_\omega = B\mathrm{Hol}_\omega$ arising from the
integration construction for the Lie algebroid $A_\omega$ is Morita
equivalent to the 2-group $[\mathbb{R}/\mathrm{Per}(\omega)]$.
\end{proposition}

%\comment{group structure?}

Hence this last proposition extends the smooth situation depicted in
\eqref{eq:gaugeME}. In particular, $\mathcal{P}_\omega$ is a principal $[\mathbb{R}/\mathrm{Per}(\omega)]$-bundle
over $M$ that generalizes the prequantum $S^1$-bundle when $\omega$ is not integral.

%\comment{issue of $P_0A$ (of tseng-zhu) versus $PA$ (used in
%\cite{crainic} for the construction above)?}

%%%%%%%%%%%%%%%%%%%%%%%%%%%%%%%%%%%%%%%%%%%%%%%%%%%%%%%%%%%%%%%%%%%%%%%%%%%%%
\subsection{Properties of Morita equivalence}

We now verify two key properties of Morita equivalence of stacky Lie
groupoids: (1) It is an equivalence relation; (2) Representability
of stacky Lie groupoids is a Morita invariant (in particular, when
restricted to Lie groupoids, Def.~\ref{def:ME} agrees with the usual
notion of Morita equivalence in this setting).

%%%%%%%%%%%%%%%%%%%%%%%%%%%%%%%%%%%%%%%%%%%%%%%%%%%%%%%%%%%%%%%%%%

\subsection*{Morita equivalence is an equivalence relation}

\begin{theorem}\label{CorMorEqRel}
Morita equivalence between stacky Lie groupoids is an equivalence
relation.
\end{theorem}

Since Example~\ref{ExmGGG} shows that Morita equivalence is a
reflexive relation, the result will follow from Lemma
\ref{LemmaMorEquivSym} and Prop.~\ref{PropMorEqRel} below, which
show that it is also symmetric and transitive. Some technical work
developed in Appendix~\ref{app:stages} will be used to verify
transitivity.

\begin{lemma}\label{LemmaMorEquivSym}
Morita equivalence between stacky Lie groupoids is a symmetric
relation.
\end{lemma}

\begin{proof}
Let $\cl{X}$ be a biprincipal $\cl{G}_1$-$\cl{G}_2$-bibundle. We
will turn it into a $\cl{G}_2$-$\cl{G}_1$-bibundle. By inverting the
actions, as in Prop.~\ref{prop:invact}, we make $\cl{G}_1$ act on
$\cl{X}$ on the right, and $\cl{G}_2$ act on $\cl{X}$ on the left.
By Proposition~\ref{PropLeftPrincRightPrinc} we see that these new
actions are principal as well. We are left to checking that the new
actions commute.

Let $\tau$ be the 2-isomorphism associated with the commutativity of
the two original actions, as in Definition~\ref{DefBibundle},
$$
\tau(g_1,x,g_2): g_1(xg_2)\lmap{\sim}(g_1x)g_2,
$$
where $x\in\cl{X}$ and $g_i\in \cl{G}_i$, for $i=1,2$, are suitably
composable. We define a 2-isomorphism $\bar{\tau}(g_2,x,g_1):
g_2(xg_1)\lmap{\sim}(g_2x)g_1$ by
$$
\bar{\tau}(g_2,x,g_1):=\tau(g_1^{-1},x,g_2^{-1})^{-1}.
$$
The higher coherences $(g_1g_1'xg_2)$, $(g_1xg_2g_2')$, $(1xg_2)$
and $(g_1x1)$ of Definition \ref{DefBibundle} for $\bar{\tau}$
follow, respectively, from conditions $(g_1xg_2g_2')$,
$(g_1g_1'xg_2)$, $(g_1x1)$, and $(1xg_2)$ for $\tau$.
\end{proof}

In order to show transitivity of Morita equivalence, we first need
two lemmas. The first is a special case of the construction in
Prop.~\ref{prop:tensor}.

\begin{lemma}\label{LemXGmodG}
Let a stacky groupoid $\G\rra M$ act (on the right) on a stack $\X$.
Consider the induced action of $\G$ on $\cl{X}\times_M \G$ by
$(x,g)\cdot\bar{g}=(x\bar{g},\bar{g}^{-1}g)$. Then the original
action $\act:\cl{X}\times_M \G\ra \X$, $(x,g)\mapsto x\cdot g$,
induces an identification of $(\cl{X}\times_M \G)/\G$ with $\cl{X}$.
\end{lemma}

\begin{proof}
Since the action is principal on the $\G$-factor, it is also in the
product, see Prop.~\ref{prop:tensor}. The induced action is on the
fibers of $\act: \cl{X}\times_M \G\ra \X$ via a natural
2-isomorphism
$$
\gamma:xg\ra (x\bar{g})(\bar{g}^{-1}g),
$$
given by a composition of structure maps that we leave to the reader
to write down. From the universal property of the prequotient (cf.
Prop.~\ref{LemmaXGtoS}), there is an induced map
$\Phi:(\cl{X}\times_M \G)\pq \G\map \X$, which reads $(x,g)\mapsto
xg$ at the level of objects. At the level of morphisms, $\Phi$ sends
$[\bar{g},(b,j)]:(x,g)\ra (x_1,g_1)$ to the composition $(b\cdot
j)\circ \gamma$:
$$ xg\lmap{\gamma} (x\bar{g})(\bar{g}^{-1}g)\lmap{b\cdot j} x_1g_1.$$
We will prove below that $\Phi$ is an epimorphism and a
monomorphism, so that, using  Prop.~\ref{PropStackification} (iv),
we conclude that the stackification $\Phi^{\sharp}:(\cl{X}\times_M
\G)/ \G\ra\X$ of $\Phi$ is an isomorphism, leading to the desired
identification.

Since any $x\in\Obj(\X)$ is isomorphic to $x\cdot 1=\Phi(x,1)$ for
an appropriate identity $1\in\Obj(\G)$, we see that $\Phi$ is
essentially surjective, hence an epimorphism.

We outline the proof of the fact that $\Phi$ is a monomorphism. We
fix a manifold $U$, and we work with objects and morphisms over $U$.
We have to prove that the section $\Phi_U$ is fully faithful; in
other words, given a morphism $b:xg\ra x_1g_1$ in $\X$, we have to
prove that there exists a unique morphism
$[\bar{g},(\bar{b},j)]:(x,g)\ra (x_1,g_1)$ in $(\cl{X}\times_M
\G)\pq \G$ such that $b=(\bar{b}\cdot j)\circ \gamma$.

For the existence  of  $[\bar{g},(\bar{b},j)]$ we define
$\bar{g}:=gg_1^{-1}$, and we take $j:
\bar{g}^{-1}g=(gg_1^{-1})^{-1}g\ra g_1$ to be a composition of
structure 2-isomorphisms of $\G$ that ``eliminate" $g$. Define
$\bar{b}:x\cdot\bar{g}\ra x_1$ to be the composition:
$$ x(gg_1^{-1})\map (xg)g_1^{-1}\lmap{b\cdot\mr{id}}(x_1g_1)g_1^{-1}\map
x_1$$ where the unspecified maps are structure 2-isomorphisms. The
reader can verify that the naturality of the structure
2-isomorphisms and their higher coherences imply that
$b=(\bar{b}\cdot j)\circ \gamma$.

For the uniqueness of $[\bar{g},(\bar{b},j)]$, we pick another
$[\bar{g}_1,(\bar{b}_1,j_1)]:(x,g)\ra (x_1,g_1)$ such that
$b=(\bar{b}_1\cdot j_1)\circ \gamma_1$. We have to prove that there
exists $j_2:\bar{g}\ra\bar{g}_1$ such that the triangle in
$\cl{X}\times_M\cl{G}$,
$$
\xymatrix{
(x,g)\cdot\bar{g}\ar[rd]^-{(\bar{b},j)}\ar[d]_-{\mr{id}\cdot j_2}&\\
(x,g)\cdot\bar{g}_1\ar[r]_-{(\bar{b}_1,j_1)}& (x_1,g_1), }
$$
commutes, which is equivalent to asking for the following two
triangles, in $\G$ and $\X$ respectively, to commute:
\begin{equation}\label{eq2triang}
\xymatrix{
\bar{g}^{-1}g \ar[rd]^-{j}\ar[d]_-{\inv(j_2)\cdot\mr{id}}&\\
\bar{g}^{-1}_1 g\ar[r]_-{j_1}& g_1 }\qquad\qquad \xymatrix{
x\bar{g} \ar[rd]^-{\bar{b}}\ar[d]_-{\mr{id}\cdot j_2}&\\
x\bar{g}_1\ar[r]_-{\bar{b}_1}& x_1 .}
\end{equation}
We define an auxiliary $j_0:\bar{g}^{-1}\ra \bar{g}^{-1}_1$ by the
composition
$$
\bar{g}^{-1}\map (\bar{g}^{-1}g)g^{-1}\lmap{j\cdot\mr{id}}g_1g^{-1}
\lmap{(j_1\cdot\mr{id})^{-1}}(\bar{g}^{-1}_1g)g^{-1}\map
\bar{g}^{-1}_1,
$$
where the unlabeled maps are compositions of structure
2-isomorphisms. Finally, define $j_2$ as the composition
$$
\bar{g}\map
(\bar{g}^{-1})^{-1}\lmap{\inv(j_0)}(\bar{g}_1^{-1})^{-1}\map
\bar{g}_1,
$$
that is, $j_2$ is essentially $\inv(j_0)$. The reader may check that
the two triangles in \eqref{eq2triang} commute: using higher
coherences, the commutativity of the left triangle follows from the
definition of $j_2$, and the commutativity of the right one follows
from the hypothesis $(\bar{b}\cdot j)\circ \gamma=(\bar{b}_1\cdot
j_1)\circ \gamma_1$.
\end{proof}

\begin{remark}\label{RemIdentifEquiv}
With the assumptions of the last lemma, if we let $\G$ act on
$\X\times_M \G$ by multiplying on the right on the $\G$ factor, then
this action descends to the quotient $(\X\times \G)/\G$ and the
identification $\Phi:\X\times_M \G\ra \X$ of the previous lemma
turns out to be $\G$-equivariant (the action on $\X$ is the original
one). The reader may consult Lemmas \ref{lem:inducedact2} and
\ref{lem:equivarianceofstackyf} to see how an action descends to a
quotient (observe that we could have started with the $\G\times \G$
action $(x,g)\cdot(\tilde{g},\bar{g})=(x\tilde{g},
(\tilde{g}^{-1}g)\bar{g})$, so that Lemma \ref{lem:inducedact2}
applies directly). As far as the equivariance of $\Phi$ is
concerned, we just remark that at the level of objects we have
$(\Phi(x,g))\bar{g}=(xg)\cdot\bar{g}\lmap{\sim} x(g\bar{g})=
\Phi(x,g\bar{g})=\Phi((x,g)\bar{g})$, and further details are left
to the reader.
\end{remark}

\begin{lemma}\label{Fact2new1}
Let the cfg-groupoid $\cl{G}_i\rra M_i$ act (on the right) on the
category fibred in groupoids $\X_i$, and let
$\cl{Z}_i:=\X_i\pq\G_i$, $i=1,2$. Suppose that we are given maps
$\cl{Z}_i\ra M$, where $M$ is a manifold. Then the product groupoid
$\cl{G}_1\times \cl{G}_2$ acts naturally on $\X_1\times_M\X_2$, this
action is on the fibers of the natural map $\X_1\times_M\X_2\ra
\cl{Z}_1\times_M\cl{Z}_2$, and the induced map
$(\X_1\times_M\X_2)\pq (\G_1\times\G_2)\ra\cl{Z}_1\times_M\cl{Z}_2$
is an isomorphism. In other words, we have an identification
\begin{equation}\label{eqnIdProdQuot}
(\X_1\pq\G_1)\times_M (\X_2\pq\G_2) \simeq (\X_1\times_M\X_2)\pq
(\G_1\times\G_2).
\end{equation}
\end{lemma}

\begin{proof}
Use the identification
\begin{equation}
(\cl{X}_1\times_M \cl{X}_2)\times_{(M_1\times M_2)}
(\cl{G}_1\times\cl{G}_2) \simeq
(\cl{X}_1\times_{M_1}\cl{G}_1)\times_M
(\cl{X}_2\times_{M_2}\cl{G}_2)
\end{equation}
to induce the action map
$$(\cl{X}_1\times_M \cl{X}_2)\times_{(M_1\times M_2)} (\cl{G}_1\times\cl{G}_2)
\map \cl{X}_1\times_M \cl{X}_2,\qquad (x_1,x_2)\cdot
(g_1,g_2)=(x_1\cdot g_1,x_2\cdot g_2).
$$
Similarly, all the other data for an action are induced
``componentwise", and the higher coherences are checked likewise.
Moreover, there is an induced projection $\Proj_1\times\Proj_2:
\cl{X}_1\times_M\cl{X}_2\ra\cl{Z}_1\times_M\cl{Z}_2$, and the fact
that the product groupoid acts on the fibers of this map is easily
checked componentwise as well.

We are left with checking that the induced map $\Phi:
(\X_1\times_M\X_2)\pq (\G_1\times\G_2)\ra\cl{Z}_1\times_M\cl{Z}_2$
(cf. Prop.~\ref{LemmaXGtoS}) is an isomorphism. At the level of
objects $\Phi$ is the identity, so that we have just to check that
the fibers of $\Phi$ over a manifold $U$ are fully faithful. So take
$(x_1,x_2)$ and $(\bar{x}_1,\bar{x}_2)$, objects in
$\cl{X}_1\times_M\cl{X}_2$, where here and in the following all the
objects are over $U$, and all the morphisms are over the identity of
$U$. Consider now $g_1,g_2,b_1,b_2$, where $g_i$ is an object of
$\G_i$ such that the composition $x_i\cdot g_i$ makes sense, and
$b_i:x_i\cdot g_i\ra \bar{x}_i$ is a morphism in $\X_i$. Then
$[(g_1,g_2),(b_1,b_2)]$ is a morphism $(x_1,x_2)\ra
(\bar{x}_1,\bar{x}_2)$ in $(\X_1\times_M\X_2)\pq (\G_1\times\G_2)$,
and $([g_1,b_1],[g_2,b_2])$ is a morphism $(x_1,x_2)\ra
(\bar{x}_1,\bar{x}_2)$ in $\cl{Z}_1\times_M\cl{Z}_2$. One can check
that the action of $\Phi$ on morphisms $(x_1,x_2)\ra
(\bar{x}_1,\bar{x}_2)$ is given by
$$[(g_1,g_2),(b_1,b_2)]\mapsto ([g_1,b_1],[g_2,b_2])$$
and that this is a bijection, as desired.

\end{proof}

\begin{remark}\label{RemFact2new1}
With the notation of the previous lemma, assume that the $\G_i$'s
are stacks, the $\X_i$'s are prestacks, and that the actions are
1-free. From Prop. \ref{prop1freeThenPrestack} it follows that the
prequotients $\X_i\pq\G_i$ are prestacks. By stackyfing the
identification (\ref{eqnIdProdQuot}), and recalling that the
stackification of the product of prestacks is the product of the
stackifications (compare with the first lines in the proof of Lemma
\ref{lem:equivarianceofstackyf}),  we see that the result of Lemma
\ref{Fact2new1} is still true, under the new assumptions, if we
replace the prequotients by the quotients.
\end{remark}

\begin{remark}\label{InducedActionXY}
The following is a special case of the previous lemma that will be
useful. Let $\cl{G}\rra M$ be a cfg-groupoid, and $\cl{X},\cl{Y} \in
\Obj(\CFG_{\cl{C}})$. Assume that $\cl{G}$ acts on the left on the
fibers of a morphism $\cl{X}\ra N$, where $N$ is a manifold, and
that we are given a morphism $\cl{Y}\ra N$. Consider the induced
left action of $\cl{G}$ on the fibers of $\cl{X}\times_N \cl{Y}\ra
\cl{Y}$. Then there is a canonical identification
$$
(\cl{X}\times_N \cl{Y})\pq \cl{G}\lmap{\sim}
(\cl{X}\pq\cl{G})\times_N \cl{Y},
$$
given by the identity on objects, and $ [g,(b,c)]\mapsto ([g,b],c) $
on morphisms, where $g\in\G$, $b:xg\ra\bar{x}$ is a morphism in
$\X$, and $c:y\ra\bar{y}$ is a morphism in $\cl{Y}$.
\end{remark}

We are now ready to prove the transitivity of Morita equivalence.

\begin{proposition}\label{PropMorEqRel}
Let $\cl{G}_i\rra M_i$, $i=1,2$, and $\cl{G}\rra M$ be stacky Lie
groupoids, and let $\cl{X}$ be a biprincipal
$\cl{G}_1$-$\cl{G}$-bibundle and $\cl{Y}$ be a biprincipal
$\cl{G}$-$\cl{G}_2$-bibundle. Then the induced $\G$-action on
$\cl{X}\times_M\cl{Y}$ (via \eqref{eq:inducedaction}) is principal,
and the quotient $\cl{X}\otimes_\G \cl{Y} = (\cl{X}\times_M
\cl{Y})/\G$ inherits the structure of a biprincipal
$\G_1$-$\G_2$-bibundle.
\end{proposition}

\begin{proof}
Let us consider the induced right $\cl{G}$-action on
$\cl{X}\times_M\cl{Y}$, as in \eqref{eq:inducedaction}. By
Prop.~\ref{prop:tensor}, the quotient $\cl{Z}:= (\cl{X}\times_M
\cl{Y})/\cl{G}$ is a differentiable stack.
We have to show that this stack is a biprincipal
$\cl{G}_1$-$\cl{G}_2$-bibundle,
$$
\xymatrix{
\cl{G}_1\ar@<0.5ex>[d]\ar@<-0.5ex>[d]&&  \cl{G}_2\ar@<0.5ex>[d]\ar@<-0.5ex>[d]\\
M_1 &\cl{Z} \ar[l]^-{}\ar[r]_-{}& M_2. }
$$

Similarly to what happens in the context of smooth manifolds, the
actions of $\cl{G}_1$ on $\cl{X}$ and of $\cl{G}_2$ on $\cl{Y}$ lift
to (strictly) commuting actions on $\cl{X}\times_M \cl{Y}$.
Moreover, these actions descend to the quotient $\cl{Z}$. The proof
of this fact is very similar to the proof of Lemma
\ref{lem:inducedact2}, and relies on the commutativity of the
actions. For instance, the commutativity of the actions of $\G_1$
and $\G$ over $\X$ implies that the action map
$\G_1\times_{M_1}\cl{X}\times_M\cl{Y}\map \cl{X}\times_M\cl{Y}$ is
$\G$-equivariant, and one applies Remark \ref{RemIsoIso}(b) to
induce the desired action $\G_1\times_{M_1}\cl{Z}\map \cl{Z}$. The
coherence 2-isomorphisms and their higher coherence properties are
induced and proved as in Lemma \ref{lem:inducedact2}. Let us write
explicitly how the induced action operates at the level of morphisms
(at the levels of objects it acts as the original action:
$g_1\cdot(x,y)=(g_1\cdot x, y)$). Let $(j_1,[g,(b_1,b_2)]):(g_1,x,y)
\ra (\bar{g}_1,\bar{x},\bar{y})$ be a morphism in
$\G_1\times_{M_1}\cl{Z}$. In particular, $j_1:g_1\ra \bar{g}_1$ is a
morphism in $\G_1$, $g$ is an object of $\G$, $b_1: xg\ra\bar{x}$ is
a morphism in $\X$, and $b_2: g^{-1}y\ra \bar{y}$ is a morphism in
$\cl{Y}$. We have $j_1\cdot [g,(b_1,b_2)]= [g,(j_1\cdot
b_1)\circ\tau_1^{-1},b_2)]$ where $\tau_1: g_1(xg)\lmap{\sim}
(g_1x)g$ is the commutativity 2-isomorphism of the original actions.
Finally, it is easy to see that the induced action (strictly)
commute.

It remains to show that $\cl{Z}\ra M_2$ is a principal
$\cl{G}_1$-bundle and that $\cl{Z}\ra M_1$ is a  principal
$\cl{G}_2$-bundle. By symmetry, it is enough to prove the assertion
for the $\G_2$-action.

We know that $\cl{Y}\ra M$ is a  submersion and an epimorphism, and
since these properties are stable under base change (see
Prop.~\ref{prop:propsubm}), we conclude that
$\cl{X}\times_M\cl{Y}\ra \cl{X}$ is a submersion and an epimorphism.
Composing with $\cl{X}\ra M_1$, which is a  submersion and an
epimorphism, we get a map $\cl{X}\times_M\cl{Y}\ra M_1$ with the
same properties. Since the triangle
$$
\xymatrix{
\cl{X}\times_M\cl{Y}\ar[r]\ar[rd]&\cl{Z}\ar[d]\\
& M_1 }
$$
commutes, it follows that $\cl{Z}\ra M_1$ is an epimorphism. By
Propositions \ref{propXtoX/Gepi/Gprestackthen}$(a)$ and
\ref{prop:propsubm}$(c)$, it follows that $\cl{Z}\ra M_1$ is a
submersion.

Lastly, we must prove that the map
$\Delta_\cl{Z}:\cl{Z}\times_{M_2}\cl{G}_2\map
\cl{Z}\times_{M_1}\cl{Z}$, $\Delta_\cl{Z}(z,g_2)=(z,z\cdot g_2)$, is
an isomorphism. Let us denote by
$\Proj:\cl{X}\times_M\cl{Y}\ra\cl{Z}$ the projection and consider
the diagram
$$\xymatrix{
& \cl{Y}\times_M(\cl{X}\times_M\cl{G})\times_M
\cl{Y}\ar[r]^-{F_1}\ar[d]_-{F_2}&
(\cl{X}\times_M\cl{Y})\times_{M_1}(\cl{X}\times_M \cl{Y})\ar[dd]^-{\Proj\times\Proj}\\
\cl{X}\times_M\cl{Y}\times_{M_2}\cl{G}_2
\ar[r]^-{F_3}\ar[d]_-{\Proj\times\mr{id}_{\cl{G}_2}}&
\cl{Y}\times_M\cl{X}\times_{M}\cl{Y}&\\
\cl{Z}\times_{M_2}\cl{G}_2\ar[rr]_-{\Delta_\cl{Z}}&&
 \cl{Z}\times_{M_1}\cl{Z}.
}
$$
Here the fibred product
$\cl{Y}\times_M(\cl{X}\times_M\cl{G})\times_M \cl{Y}$ is defined by
the map
$$
\cl{X}\times_M\cl{G}\ra M,\;\;\; (x,g)\mapsto \sour(g)
$$
on the right, and
$$
\cl{X}\times_M\cl{G}\ra M,\;\;\;  (x,g)\mapsto \ma(x)
$$
on the left, where $\ma:\X\ra M$ is the moment map for the
$\cl{G}$-action on $\X$. The general strategy of this part of the
proof is as follows. We define below the maps $F_1,F_2, F_3$, and
appropriate actions in such a way that: the diagram 2-commutes,
$F_1$ and $F_3$ are equivariant isomorphisms, and the vertical maps
are quotients. As we will see (using Prop.~\ref{prop:quotstages}),
the map $\Delta_\cl{Z}$ is induced from the equivariant map $F_1$ on
quotients, from where it follows that $\Delta_\cl{Z}$ is an
isomorphism (cf. Remark \ref{RemIsoIso}).

The map $F_1$ is defined by $F_1(y,x,g,\bar{y})=(x,y,xg,\bar{y})$,
the map $F_2$ is given by $F_2(y,x,g,\bar{y})=(g^{-1}y,xg,\bar{y})$,
while $F_3$ is given by $F_3(x,y,g_2)=(y,x,yg_2)$.

As for the actions, we define
\begin{itemize}
\item a $(\cl{G}\times\cl{G})$-action on
$\cl{Y}\times_M(\cl{X}\times_M\cl{G})\times_M \cl{Y}$  by
$$
(y,x,g,\bar{y})\cdot (\tilde{g},\bar{g})=
(\tilde{g}^{-1}y,x\tilde{g},(\tilde{g}^{-1}g)\bar{g},
\bar{g}^{-1}\bar{y}),
$$
\item a $(\cl{G}\times\cl{G})$-action on
$(\cl{X}\times_M\cl{Y})\times_{M_1}(\cl{X}\times_M \cl{Y})$ by
$$
(x,y,\bar{x},\bar{y})\cdot
(\tilde{g},\bar{g})=(x\tilde{g},\tilde{g}^{-1}y,
\bar{x}\bar{g},\bar{g}^{-1}\bar{y}),
$$
\item a $\cl{G}$-action on $\cl{Y}\times_M\cl{X}\times_{M}\cl{Y}$ by
$$
(y,x,\bar{y})\cdot
\bar{g}=(\bar{g}^{-1}y,x\bar{g},\bar{g}^{-1}\bar{y}),
$$
\item a $\cl{G}$-action on $\cl{X}\times_M\cl{Y}\times_{M_2}\cl{G}_2$
by
$$
(x,y,g_2)\cdot \bar{g}=(x\bar{g},\bar{g}^{-1}y,g_2).
$$
\end{itemize}

The proof now goes as follows:

\noindent{\bf Step 1}: $F_1$ is $(\cl{G}\times\cl{G})$-equivariant
and an isomorphism -- here we use the hypothesis that the map
$\cl{X}\times_M\cl{G}\ra
 \cl{X}\times_{M_1}\cl{X}$,
$(x,g)\mapsto (x,xg)$, is an isomorphism.

\smallskip

\noindent{\bf Step 2}: $F_3$ is $\cl{G}$-equivariant and an
isomorphism -- here we use the hypothesis that the map
$\cl{Y}\times_{M_2}\cl{G}_2\ra \cl{Y}\times_{M}\cl{Y}$,
$(y,g_2)\mapsto (y,yg_2)$, is an isomorphism.

\smallskip

\noindent{\bf Step 3}: $\cl{Z}\times_{M_1}\cl{Z}$ is the quotient of
$(\cl{X}\times_M\cl{Y})\times_{M_1}(\cl{X}\times_M \cl{Y})$ by
$\cl{G}\times\cl{G}$ (by Remark \ref{RemFact2new1}, noticing that
each $\G$-action is principal, so 1-free).

\smallskip

\noindent{\bf Step 4}: We now observe that $F_2$ and
$\Proj\times\mr{id}_{\cl{G}_2}$ are quotients, and that
$\cl{Z}\times_{M_2}\cl{G}_2$ is the quotient of
$\cl{Y}\times_M(\cl{X}\times_M\cl{G})\times_M \cl{Y}$ by
$\cl{G}\times\cl{G}$. Indeed, we can apply Lemma \ref{LemXGmodG}
(and Remark~\ref{InducedActionXY}) to conclude that $F_2$ is the
quotient by the $\cl{G}$-action of the first factor of $\G\times \G$
(in the sense of Lemma \ref{lem:inducedact2}). Moreover, one can
check (see Remark \ref{RemIdentifEquiv}) that the induced
$\G$-action of the second factor on the quotient
$\cl{Y}\times_M\times\cl{X}\times_M\cl{Y}$ coincides with the one
given above. Also, $\Proj\times\mr{id}_{\cl{G}_2}$ is the quotient
by the $\G$-action by Remark~\ref{InducedActionXY}. Since $F_3$ is
an equivariant isomorphism, and noticing that the action of
$\G\times \G$ on $\cl{Y}\times_M(\cl{X}\times_M\cl{G})\times_M
\cl{Y}$ is 1-free, we can use Proposition \ref{prop:quotstages} to
conclude that $\cl{Z}\times_{M_2}\cl{G}_2$ is the desired quotient.

\smallskip

\noindent{\bf Step 5}: The diagram is 2-commutative in the sense
that, choosing a quasi-inverse $F_4$ of $F_3$, there is a
2-isomorphism from $(\Proj\times\Proj)\circ F_1$ to
$\Delta_\cl{Z}\circ (\Proj\times\mr{id}_{\cl{G}_2})\circ F_4\circ
F_2$. To prove this fact we argue as follows. Choose an object
$(y,x,g,\bar{y})$ in
$\cl{Y}\times_M\cl{X}\times_M\cl{G}\times_M\cl{Y}$, and call
$(x_1,y_1,g_2)$ its image via $F_4\circ F_2$; we have to define a
natural isomorphism $(x,y,xg,\bar{y})\ra (x_1,y_1,x_1,y_1g_2)$
between the images of these two objects in
$\cl{Z}\times_{M_1}\cl{Z}$. Since $F_4$ and $F_3$ are
quasi-inverses, we have a natural isomorphism $(b_1,b_2,b_3):
(g^{-1}y,xg,\bar{y})\ra (y_1,x_1,y_1g_2)$ between the images of the
above two objects in $\cl{Y}\times_M\cl{X}\times_M\cl{Y}$. Observe
that $[g,(b_2,b_1)]:(x,y)\ra (x_1,y_1)$ is an isomorphism in
$\cl{Z}$. Moreover, we can consider the image
$[1,(b_2,b_3)\circ\vep]:(xg,\bar{y})\ra (x_1,y_1g_2)$ in $\cl{Z}$ of
the isomorphism $(b_2,b_3)$ in $\cl{X}\times_M\cl{Y}$, see Equation
(\ref{eqnProjExpl}). We conclude the argument by taking
$([g,(b_2,b_1)],[1,(b_2,b_3)\circ\vep])$ as the desired isomorphism
in $\cl{Z}\times_{M_1}\cl{Z}$.
%\end{itemize}
\end{proof}

\begin{remark} The previous proposition in fact shows that one can
compose right-principal stacky bibundles, extending the usual
composition of Hilsum-Skandalis bibundles between Lie groupoids.
Just as Lie groupoids, Hilsum-Skandalis bibundles and bibundle morphisms form a (non-strict) 2-category,
a natural question is whether these ``higher Hilsum-Skandalis bibundles'' of stacky Lie groupoids, along
with suitable notions of 2-morphisms and 3-morphisms, form a 3-category. Initial steps in verifying this were made in \cite[Sections~6, 7]{duli} (other ways to build a higher category for higher groupoids are considered in \cite{bg,RZ}).
\end{remark}

%%%%%%%%%%%%%%%%%%%%%%%%%%%%%%%%%%%%%%%%%%%%%%%%%%%%%%%%%%%%%%%%%%%%%%%%
\subsection*{Relation with Morita equivalence of Lie groupoids}

Theorem \ref{CorMorEqRel} shows that the notion Morita equivalence
of Lie groupoids can be extended to the realm of stacky Lie
groupoids. We now observe that this more general equivalence
relation does not relate representable stacky Lie groupoids (i.e.,
ordinary Lie groupoids) to nonrepresentable ones, nor does it
provide new Morita equivalences between Lie groupoids.  In
particular, when restricted to ordinary Lie groupoids, it recovers
exactly the usual notion of Morita equivalence.

\begin{lemma}\label{LemmaHH}
Let $\cl{G}\rra M$ be a stacky Lie groupoid, and let $K\rra M$ be a
Lie groupoid. Suppose that the diagram
$$
\xymatrix{
K\ar@<0.5ex>[d]\ar@<-0.5ex>[d]&&  \cl{G}\ar@<0.5ex>[d]\ar@<-0.5ex>[d]\\
M &K \ar[l]^-{t}\ar[r]_-{s}& M }
$$
is a biprincipal bibundle (i.e., a Morita equivalence). Then
$\cl{G}$ is canonically isomorphic to $K$ as a stacky Lie groupoid.
\end{lemma}

\begin{proof}
Suppose that $K$ is biprincipal $K$-$\cl{G}$-bibundle as in the
diagram. Notice that, since all spaces involved are manifolds except
for $\cl{G}$, one only has 2-isomorphisms associated with $\cl{G}$
itself. By assumption, the canonical map
\begin{equation}\label{eq:actH}
K\times_{s,M,\tar} \cl{G}\map K\times_{t,M,t} K, \qquad (k,g)\mapsto
(k,kg),
\end{equation}
is an isomorphism. There are left actions of $K$ on $K\times_{M}
\cl{G}$ and $K\times_{M} K$:
$$
k\cdot(k',g)=(kk',g), \qquad k\cdot(k_1,k_2)=(kk_1,kk_2),
$$
and the map \eqref{eq:actH} is equivariant with respect to these
actions. One may verify that the projections
$$
K\times_{M} \cl{G} \map \cl{G}, \qquad (k,g)\mapsto g,
$$
(compare with Remark \ref{InducedActionXY}) and
$$
K\times_{M} K \map K, \qquad (k_1,k_2)\mapsto k_1^{-1}k_2,
$$
are principal left $K$-bundles. From Remark \ref{RemIsoIso},
\eqref{eq:actH} induces an isomorphism
$$
F:\cl{G}\map K,
$$
given by $g\mapsto 1\cdot g$, where $1=1_{\tar(g)}$  is the identity
element of $K$ on which $g$ can act.

In order to complete the proof we have to check that $F$ preserves
the groupoid structure. The morphism $F$ commutes with the source
and the identity by definition of the action, and it commutes with
the target because $\cl{G}$ acts on the fibers of $t:K\ra M$. To
check that $F$ preserves the multiplication, let $g_1$ and $g_2$ be
objects in $\cl{G}$ with $\sour(g_1)=\tar(g_2)$. Then
$$
F(g_1g_2)=1_{\tar(g_1)}(g_1g_2)
=(1_{\tar(g_1)}g_1)g_2=(F(g_1)1_{\tar(g_2)})g_2=F(g_1)(1_{\tar(g_2)}g_2)=
F(g_1)F(g_2).
$$
As for ordinary Lie groupoids, preservation of multiplication and
identity yields preservation of the inverse: $F(g^{-1})=F(g)^{-1}$.
Indeed, since there is an isomorphism in $\cl{G}$ between $gg^{-1}$
and $1_{\tar(g)}$ and since $K$ is a manifold, the images
$F(gg^{-1})$ and $F(1_{\tar(g)})$ under $F$ are equal.
\end{proof}

\begin{proposition}\label{prop:equivusual}
Let $\cl{G}_i\rra M_i$ be stacky Lie groupoids, for $i=1,2$, and let
$\cl{X}$ be a biprincipal $\cl{G}_1$-$\cl{G}_2$-bibundle. If
$\cl{G}_1$ is a Lie groupoid, then $\cl{X}$ is a manifold and
$\cl{G}_2$ is a Lie groupoid.
\end{proposition}

\begin{proof}
Let $X\ra \cl{X}$ be an atlas. Composing with the map $\ma_2:
\cl{X}\ra M_2$ (along which $\G_2$ acts), we obtain a surjective
submersion $X\ra M_2$. Then there is an open cover $(U_\alpha)$ of
$M_2$ such that, for each $\alpha$, there is a section $U_\alpha\ra
X$, and we obtain local sections $U_\alpha \ra \cl{X}$ of $\cl{X}\ra
M_2$. For each $\alpha$, we have isomorphisms
$$
U_\alpha\times_{M_2}
\cl{X}=U_\alpha\times_{\cl{X}}\cl{X}\times_{M_2}\cl{X}=U_\alpha
\times_\cl{X} (\cl{G}_1\times_{M_1} \cl{X})=U_\alpha \times_{M_1}
\cl{G}_1,
$$
so that each $U_\alpha\times_{M_2} \cl{X}$ is a manifold (since
$\G_1$ is a Lie groupoid). The manifolds $ V_\alpha := U_\alpha
\times_{M_2} \cl{X}$ cover $\cl{X}$, in the sense that, for any
morphism $Z \to \X$, where $Z$ is a manifold, the family $(Z_\alpha
= V_\alpha \times_\X Z)_\alpha$ is an open cover of $Z$. It follows
that $\X$ is a manifold.

Let $\cl{X}\otimes_{\cl{G}_1}\cl{X} = (\cl{X}\times_{M_1}
\cl{X})/\G_1 \rra M_2$ be the gauge groupoid of the principal left
$\cl{G}_1$-bundle $\cl{X}\ra M_2$ (see e.g. \cite[Sec.~5.1]{MM}),
which is a Lie groupoid. Then we have biprincipal bibundles
$$
\xymatrix{
\cl{X}\otimes_{\cl{G}_1}\cl{X}\ar@<0.5ex>[d]\ar@<-0.5ex>[d]&&
 \cl{G}_1\ar@<0.5ex>[d]\ar@<-0.5ex>[d]
&&\cl{G}_2 \ar@<0.5ex>[d]\ar@<-0.5ex>[d]\\
M_2 &\cl{X} \ar[l]^-{}\ar[r]_-{}& M_1 & \cl{X} \ar[l]^-{}\ar[r]_-{}&
M_2 }
$$
By Proposition \ref{PropMorEqRel}, we see that
$\cl{X}\otimes_{\cl{G}_1}\cl{X}$ inherits the structure of a
biprincipal bibundle establishing a Morita equivalence between
$\cl{X}\otimes_{\cl{G}_1}\cl{X}$ and $\G_2$. By Lemma \ref{LemmaHH}
we conclude that $\cl{G}_2$ is a Lie groupoid.
\end{proof}

%%%%%%%%%%%%%%%%%%%%%%%%%%%%%%%%%%%%%%%%%%%%%%%%%%%%%%%%%%%%%%%%%%%%%%%%%%%%%
\appendix

\section{Appendices}

\subsection{Proofs of properties of
submersions}\label{app:morphisms}

We will need the following lemma, which is just \cite[Lem.~2.2]{bx}:

\begin{lemma}\label{LemmaBX}
Let $F:\cl{X}\ra \cl{Y}$ be a morphism of stacks, and $Y\ra \cl{Y}$
be an epimorphism from a manifold $Y$. If $Y\times_\cl{Y} \cl{X}$ is
representable and the induced morphism of manifolds $Y\times_\cl{Y}
\cl{X}\ra Y$ is a submersion, then $F$ is representable.
\end{lemma}

\noindent {\bf Proof of Prop.~\ref{prop:charactsub}}. It is clear
that either $(b)$, $(c)$ or $(d)$ imply $(a)$. We now check that
$(a)$ implies each one of them.

\

\noindent $\boldsymbol{(a) \then (b)}$: Assume that
$\cl{X}\ra\cl{Y}$ is a submersion and consider atlases $X\ra\cl{X}$
and $Y\ra\cl{Y}$ such that the induced map of manifolds
$Y\times_{\cl{Y}} X\ra Y$ is a submersion. By Lemma \ref{LemmaBX} it
follows that $X\ra \cl{Y}$ is representable.

\

\noindent $\boldsymbol{(a)\then (c)}$: We have to prove that in the
cartesian diagram
$$ \xymatrix{
Z\ar[r]\ar[d]_-{}&U\ar[d]^-{}\\
\cl{X}'\ar[r]^-{}\ar[d]_-{}&\cl{X}\ar[d]^-{}\\
V\ar[r]_-{}&\cl{Y} }
$$
the map $Z\ra V$ is a submersion. By assumption, there is a
cartesian diagram
$$ \xymatrix{
X_1\ar[r]\ar[d]_-{}&X\ar[d]^-{}\\
\cl{X}_1\ar[r]^-{}\ar[d]_-{}&\cl{X}\ar[d]^-{}\\
Y\ar[r]_-{}&\cl{Y} }
$$
in which $X\ra\cl{X}$ and $Y\ra\cl{Y}$ are atlases, $X_1$ is a
manifold, and $X_1\ra Y$ is a submersion.

We consider the cartesian diagram
 $$\xymatrix@=10pt{
  W_1\ar[rr]
  \ar[rd]\ar[dd]&&
  W\ar[rd]\ar'[d][dd]&\\
  & X_1\ar[dd]\ar[rr]&&
  X\ar[dd] \\
  \cl{X}_1'\ar'[r][rr]\ar[rd]\ar[dd]&&
  \cl{X}'\ar[rd]\ar'[d][dd]&\\
  & \cl{X}_1\ar[dd]\ar[rr]&&
  \cl{X}\ar[dd] \\
  V_1\ar[rd]\ar'[r][rr] &&
  V\ar[rd]&\\
  & Y\ar[rr]  &&
  \cl{Y}
 }
 $$
in which $W, W_1, V_1$ are manifolds. We now verify that $W\ra V$ is
a submersion. Indeed, $W_1\ra V_1$ is a submersion (being the base
change of $X_1\ra Y$), $V_1\ra V$ is a surjective submersion (being
the base change of $Y\ra \cl{Y}$), and $W_1\ra W$ is surjective
(being the base change of $V_1\ra V$).

By considering the fibred product
$$
\xymatrix{
U'\ar[r]\ar[d]_-{}&X\ar[d]^-{}\\
U\ar[r]_-{}&\cl{X} }
$$
where $U'$ is a manifold, $U'\ra X$ is a submersion, and $U'\ra U$
is a surjective submersion, we get cartesian diagrams
$$
\xymatrix{
W'\ar[r]\ar[d]&U'\ar[d]\\
W\ar[r]\ar[d]_-{}&X\ar[d]^-{}\\
\cl{X}'\ar[r]^-{}\ar[d]_-{}&\cl{X}\ar[d]^-{}\\
V\ar[r]_-{}&\cl{Y} }\qquad\qquad \xymatrix{
W'\ar[r]\ar[d]&U'\ar[d]\\
Z\ar[r]\ar[d]_-{}&U\ar[d]^-{}\\
\cl{X}'\ar[r]^-{}\ar[d]_-{}&\cl{X}\ar[d]^-{}\\
V\ar[r]_-{}&\cl{Y} }
$$
in which $W'$ is a  manifold, $W'\ra W$ is a submersion, and $W'\ra
Z$ is a surjective submersion. Since $W\ra V$ is a submersion,  so
is $Z\ra V$, and we are done. (In the last step we are using that
the compositions $U'\ra\cl{X}$ in the two diagrams are isomorphic,
hence the compositions $W'\ra V$ in the two diagrams are equal.)

\

\noindent $\boldsymbol{(a)\then (d)}$: Since $(a)$ and $(b)$ are
equivalent, we may assume that there exists an atlas $X\ra\cl{X}$
such that the composition $X\ra\cl{Y}$ is representable. If we take
the fibred product
$$
\xymatrix{
U'\ar[r]^-{b}\ar[d]^-{c}&X\ar[d]^-{d}\\
U\ar[r]_-{e}&\cl{X} }
$$
then $U'$ is a manifold, $c$ is a surjective submersion and $b$ is a
submersion. Hence $Fec\simeq Fdb$ is representable. We have to prove
that $Fe$ is representable. Let us consider an atlas $Y\ra\cl{Y}$
and the cartesian diagram
$$ \xymatrix{
W'\ar[r]\ar[d]_-{f}&U'\ar[d]^-{c}\\
W\ar[r]^-{}\ar[d]_-{g}&U\ar[d]^-{Fe}\\
Y\ar[r]_-{}&\cl{Y} }$$ where $W$ and $W'$ are manifolds, $f$ is
surjective and $gf$ is a submersion. It follows that $g$ is a
submersion and, by Lemma \ref{LemmaBX}, $Fe$ is representable.

\

\noindent{\bf Proof of Prop.~\ref{prop:propsubm}}. To prove $(a)$,
let $\cl{X}\ra\cl{Y}$ and $\cl{Y}\ra\cl{Z}$ be submersions. By
Prop.~\ref{prop:charactsub} (b), there exists an atlas $X\ra\cl{X}$
such that $X\ra\cl{Y}$ is representable. By Prop.~
\ref{prop:charactsub} (d), $X\ra \cl{Z}$ is representable. Using
Prop.~\ref{prop:charactsub} (b) again, we see that $\X \to \cl{Z}$
is a submersion.

For $(b)$, we have to prove that, if in the cartesian diagram of
differentiable stacks
$$
\xymatrix{\cl{X}'\ar[r]\ar[d]&\cl{X}\ar[d]\\
\cl{Y}'\ar[r]&\cl{Y}}
$$
the morphism $\cl{X}\ra\cl{Y}$ is a submersion, then so is $\cl{X}'
\ra\cl{Y}'$. Since $\cl{X}\ra\cl{Y}$ is a submersion, there exists
an atlas $X\ra \cl{X}$ such that $X\ra \cl{Y}$ is representable. Let
$X'=\cl{X}'\times_\cl{X} X$. Then $X'\ra\cl{X}'$ is an atlas and
$X'\ra\cl{Y}'$ is representable. We conclude that
$\cl{X}'\ra\cl{Y}'$ is a submersion.

To prove $(c)$, let $f:X\ra\cl{X}$ be an atlas. Since $F$ is a
submersion and an epimorphism, $Ff$ is an atlas of $\cl{Y}$. Since
$F'F$ is a submersion, then $F' Ff$ is representable, and we
conclude that $F'$ is a submersion.

%%%%%%%%%%%%%%%%%%%%%%%%%%%%%%%%%%%%%%%%%%%%%%%%%%%%%%%%%%%%%%%%%%%%%%%%%%%%%%%%%
\subsection{Inversion and actions}

Given a left (resp. right) action of a cfg-groupoid $\G$ on a
category fibred in groupoids $\X$, there is a natural way to use the
groupoid inversion to turn it into a right (resp. left) action. This
section discusses some technical aspects of this procedure.

\begin{proposition}\label{prop:invact}
Let the cfg-groupoid  $\cl{G}\rra M$ act on the left on a category
fibred in groupoids $\cl{X}$. Then the inversion on $\G$ defines a
right action of $\cl{G}$ on $\cl{X}$ by
$$
x\cdot g:= g^{-1}\cdot x.
$$
\end{proposition}

The proof follows from a few observations. First, note that we have
natural 2-isomorphisms
\begin{equation}\label{eq:thetachi}
\theta: (hg)^{-1}\ra g^{-1}h^{-1},\qquad \chi: 1^{-1}\ra 1,
\end{equation}
see Lemma~\ref{lemtheta} and \eqref{eq:chi} below for definitions.
We will need to check the commutativity of a few diagrams involving
these 2-isomorphisms, and in order to simplify matters we will
follow some ideas from \cite{Laplaza}.

Let $\cl{G}\rra M$ be a cfg-groupoid. For all $g\in \cl{G}$,
multiplication on the right by $g$ defines a functor
$$
R_g: \sour^{-1}(\tar(g))\ra \sour^{-1}(\sour(g)).
$$
One can prove, as in \cite[Prop.~1.1]{Laplaza}, that $R_g$ is an
equivalence of categories with quasi-inverse $R_{g^{-1}}$. In
particular, $R_g$ is fully faithful, and one can deduce the
commutativity of a diagram in $\sour^{-1}(\tar(g))$ by multiplying
on the right by $g$ and checking the commutativity of the resulting
diagram. Similar statements hold for multiplication on the left.

\begin{lemma}\label{lemtheta}
For all $g,h\in \G$ with $\sour(g)=\tar(h)$  there exists a unique
isomorphism $\theta_{g,h}: (gh)^{-1}\ra h^{-1}g^{-1}$, natural in
$g,h$, for which the following diagram commutes:
$$ \xymatrix{
((gh)^{-1}g)h\ar[rr]^-{(\theta\cdot\mr{id})\mr{id}}&& ((h^{-1}g^{-1})g)h\\
(gh)^{-1}(gh)\ar[d]_-{i_l}\ar[u]^{\alpha}&
&(h^{-1}( g^{-1}g))h\ar[d]^-{\i_l\cdot\mr{id}}\ar[u]^{\alpha\cdot\mr{id}}\\
1&h^{-1}h\ar[l]^-{\i_l}&(h^{-1}1) h.\ar[l]^-{\rho\cdot\mr{id}} }$$
\end{lemma}

The proof of this lemma follows \cite{Laplaza} (see comments after
Prop.~1.7 in this reference).

\begin{lemma}\label{MonsterDiagram}
For all composable $g,h,l\in \cl{G}$ the following diagram commutes:
$$
\xymatrix{ (g\cdot hl)^{-1}\ar[r]^-{\theta}\ar[d]_-{\inv(\alpha)}&
(hl)^{-1}g^{-1}\ar[r]^-{\theta\cdot\mr{id}}&
l^{-1}h^{-1}\cdot g^{-1}\\
(gh\cdot l)^{-1}\ar[r]_-{\theta}&
l^{-1}(gh)^{-1}\ar[r]_-{\mr{id}\cdot\theta}& l^{-1}\cdot
h^{-1}g^{-1}\ar[u]_-{\alpha}}
$$
where $\inv$ is the inverse map of the groupoid.
\end{lemma}

The idea of the proof of this lemma is to first note, as mentioned
above, that it is enough to check the commutativity of the diagram
obtained after multiplying on the right by $(gh)l$. Moving
parentheses and canceling terms of the form $x^{-1}x$ produces a
large diagram, the commutativity of which can be checked applying
the higher coherences of the groupoids and the commutativity of the
diagram in Lemma \ref{lemtheta}.

The 2-isomorphism $\chi$ in \eqref{eq:thetachi} is given by the
composition
\begin{equation}\label{eq:chi}
\xymatrix{ 1^{-1}\ar@{-->}[rd]_-{\chi}& 1^{-1}1\ar[l]_-{\rho}
\ar[d]^-{\i_l}\\
& 1. }
\end{equation}

We will need to use the following compatibilities between $\theta$
and $\chi$.

\begin{lemma}\label{lemchitheta}
For all $g\in\cl{G}$ the following diagrams commute:
$$
\xymatrix{ (1g)^{-1}\ar[r]^{\theta}\ar[d]_-{\inv(\lambda)}&
g^{-1}1^{-1}\ar[d]^-{\mr{id}\cdot\chi}\\
g^{-1}& g^{-1}1\ar[l]^-{\rho} }\qquad \qquad \xymatrix{
(g1)^{-1}\ar[r]^{\theta}\ar[d]_-{\inv(\rho)}&
1^{-1}g^{-1}\ar[d]^-{\chi\cdot\mr{id}}\\
g^{-1}& 1g^{-1}.\ar[l]^-{\lambda} }$$
\end{lemma}

\begin{proof}
One can prove the commutativity of both diagrams by multiplying them
on the right by $g$, and using Lemma \ref{lemtheta} and the higher
coherences of the groupoid. (For the second diagram, one uses the
fact that $\rho=\lambda:1\cdot 1\ra 1$, an equality that can be
proven as in \cite[Thm.~3']{Kelly}.)
\end{proof}

\begin{remark}
In this paper, we have chosen to check the commutativity of all the
necessary diagrams directly, by using higher coherence conditions. A
more general alternative would be to prove a coherence theorem in
the spirit of \cite{Laplaza}.
\end{remark}

To complete the proof of Prop.~\ref{prop:invact}, take the
associativity and identity 2-isomorphisms of the new action as
$$
x\cdot gh\lmap{\beta^*} xg\cdot h\qquad\qquad x1\lmap{\varepsilon^*}
x,
$$
where
$$
\beta^*:\quad  x\cdot
gh=(gh)^{-1}x\quad\lmap{\theta\cdot\mr{id}}\quad h^{-1}g^{-1}\cdot
x\quad \lmap{\beta} \quad h^{-1}\cdot g^{-1}x = xg\cdot h
$$
and
$$ \varepsilon^*:\quad x1=1^{-1}x\quad\lmap{\chi\cdot\mr{id}}\quad 1x\quad
\lmap{\varepsilon}\quad x.
$$
Here $\beta$ and $\varepsilon$ are the 2-isomorphisms of the
original action, and $\theta$ and $\chi$ are as in
\eqref{eq:thetachi}.

The higher coherence $(xghl)$ for $\beta^*$ follows from condition
$(lhgx)$ for $\beta$, and by the commutativity of the diagram in
Lemma~\ref{MonsterDiagram}. The higher coherence $(x1g)$ for the new
action follows from $(g1x)$ for the original action and the
commutativity of the first diagram in Lemma \ref{lemchitheta}.
Similarly, the higher coherence $(xg1)$ for the new action follows
from $(1gx)$ for the original action and the commutativity of the
second diagram in Lemma \ref{lemchitheta}.

Finally, we remark that, as expected, one can interchange left and
right principal bundles by inverting actions.

\begin{proposition}\label{PropLeftPrincRightPrinc}
Let $\cl{G}$ be a stacky Lie groupoid and $\Proj:\cl{X}\ra\cl{S}$ be
a morphism of differentiable stacks. Let $\cl{G}$ act on the left on
$\cl{X}$, and let $x\cdot g:=g^{-1}\cdot x$ be the corresponding
right action. Then $\Proj$ makes $\X$ into a principal
$\cl{G}$-bundle for the left action if and only if it does for the
right action.
\end{proposition}

\begin{proof}
We must check conditions $1.$, $2.$ and $3.$ of Definition
\ref{DefPrincBundle}. Condition $1.$ is immediate. To analyze $2.$,
let $\gamma$  be the 2-isomorphism
$$
\gamma(g,x): \Proj(x)\lmap{\sim} \Proj(gx),
$$
and let $\gamma'$ be the corresponding 2-isomorphism for the right
action:
$$
\gamma'(x,g):=\gamma(g^{-1},x):\quad \Proj(x)\lmap{\sim} \Proj(xg).
$$
We have to show that $\gamma$  makes $\cl{G}$ act on the fibers of
$\Proj$ via the left action if and only if $\gamma'$ makes $\cl{G}$
act on the fibers of $\Proj$ via the right action. We refer to
$(\gamma\beta)$ and $(\gamma\vep)$ as the higher coherences for the
left action, and $(\gamma\beta)'$ and $(\gamma\vep)'$ the analogous
conditions for the right action. Consider the diagram
$$
\xymatrix{ \Proj(x)\ar[d]_-{\gamma}\ar[r]^-{\Proj(\mr{id})}&
\Proj(x)\ar[d]^-{\gamma}\ar[r]^-{\gamma}&
\Proj(g^{-1}x)\ar[d]^-{\gamma}\\
\Proj((gh)^{-1}x)\ar[r]_-{\Proj(\theta\cdot\mr{id})}&
\Proj((h^{-1}g^{-1})x)\ar[r]_-{\Proj(\beta)}& \Proj(h^{-1}(g^{-1}x))
}
$$
where $\theta:(gh)^{-1}\ra h^{-1}g^{-1}$ is the isomorphism in
\eqref{eq:thetachi}. The outer rectangle is the one of condition
$(\gamma\beta)'$, while the right square is the one of
$(\gamma\beta)$. Moreover, the left square is commutative by
naturality of $\gamma$. Hence, condition $(\gamma\beta)$ implies
$(\gamma\beta)'$. For the converse, one has to substitute $g$ and
$h$ in the above diagram with $g^{-1}$ and $h^{-1}$. In order to
prove that $(\gamma\beta)'$ implies $(\gamma\beta)$, one has to use
that there is a 2-isomorphism $\inv \circ \inv \ra \mr{id}_\cl{G}$,
where $\inv$ is the inverse of the groupoid, as well as the
naturality of $\beta$ and $\gamma$. With a similar argument the
reader can check the equivalence between $(\gamma\vep)$ and
$(\gamma\vep)'$.

For condition $3.$, we notice the commutativity of the following
diagram:
$$\xymatrix{
\cl{X}\times_M\cl{G}\ar[r]\ar[d]&
\cl{X}\times_\cl{S}\cl{X}\ar[d]\\
\cl{G}\times_M\cl{X}\ar[r]& \cl{X}\times_\cl{S}\cl{X} }$$ where the
left vertical map is the isomorphism $(x,g)\mapsto (g^{-1},x)$, the
right vertical one is the isomorphism $(x,\sigma,y)\mapsto
(y,\sigma^{-1},x)$, the upper horizontal one is $(x,g)\mapsto
(x,\gamma',xg)$, and the lower horizontal one is $(g,x)\mapsto
(gx,\gamma^{-1},x)$.
\end{proof}

%%%%%%%%%%%%%%%%%%%%%%%%%%%%%%%%%%%%%%%%%%%%%%%%%%%%%%%%%%%%%%%%%%%%%%%%%%%%%%%%%%%%%%%%
\subsection{The prequotient construction}\label{app:preq}

Here we present the {\bf proof of Prop.~\ref{PropPreqCfg}}.

We first show that $\cl{X}\pq\cl{G}$ is a category, which involves
proving that the associativity and the identity axioms for
composition of morphisms hold. In order to illustrate the general
proof method, we give a detailed account of the associativity axiom
for morphisms in $\cl{X}\pq\cl{G}$, leaving the identity axioms to
the reader. We remark that this is a point where the higher
coherence conditions ${\bf (a4)}$ for the action of $\G$ on $\X$
must be used. Indeed, the associativity for composition of morphisms
relies on $(xghl)$, while $(xg1)$  and $(x1g)$ are used to prove the
left and right identity axioms, respectively.

\textit{Associativity of the composition of morphisms in
$\cl{X}\pq\cl{G}$:} Consider three morphisms in $\cl{X}\pq\cl{G}$,
$$
x\lmap{[g,b]} y\lmap{[h,c]}z\lmap{[k,d]} w.
$$
In order to perform the various compositions we arbitrarily choose
pull backs in $\cl{G}$:
$$
\xymatrix{ \mu^*h\ar[r]^-{\mu_h}\ar@{-}[d]&
h\ar@{-}[d]&\nu^*k\ar[r]^-{\nu_k}\ar@{-}[d]&k\ar@{-}[d]
&(\nu\mu)^*k\ar@{-}[d]\ar[r]^-{(\nu\mu)_k}& k\ar@{-}[d]\\
U\ar[r]_-{\mu} &V& V\ar[r]_-{\nu}&W & U\ar[r]_-{\nu\mu}&W }
$$
where $\mu:=\pi_\X(b)$ and $\nu:= \pi_\X(c)$. Then we have
\begin{align*}
[h,c][g,b]&=[g\cdot \mu^*h, \; c\, \scirc\,  (b\cdot
\mu_h)\,\scirc\, \beta(x,g,
\mu^*h)],\\
[k,d][h,c]&=[h\cdot \nu^*k, \; d \, \scirc\,  (c\cdot \nu_k) \,
\scirc\, \beta(y,h, \nu^*k)].
\end{align*}
In order to perform compositions using three morphisms, we have to
choose pullbacks in $\cl{G}$:
$$\xymatrix{
 \mu^*(h\cdot \nu^*k)\ar[r]^-{\mu_{h\cdot \nu^*k}}\ar@{-}[d]& h\cdot \nu^*k\ar@{-}[d]&
\mu^*\nu^*k\ar[r]^-{\mu_{\nu^*k}} \ar@{-}[d]& \nu^*k\ar@{-}[d]\\
U\ar[r]_-{\mu}&V& U\ar[r]_-{\mu}&V, }
$$
but this time we do not make arbitrary choices; instead, we set
$\mu^*(h\cdot \nu^*k):= \mu^*h\cdot (\nu\mu)^*k$, and $\mu_{h\cdot
\nu^*k}:=\mu_h\cdot \mu_{\nu^*k}$, and $\mu^*\nu^*k:=(\nu\mu)^*k$,
and $\mu_{\nu^*k}$ to be the unique map over $\mu$ such that
$\nu_k\mu_{\nu^*k}=(\nu\mu)_k$. Then the compositions read as
follows:
$$
([k,d][h,c])[g,b]=\left[g\cdot (\mu^*h\cdot (\nu\mu)^*k), d \,
\scirc\,  (c\cdot \nu_k) \, \scirc\,  \beta(y,h, \nu^*k) \, \scirc\,
( b\cdot (\mu_h\cdot \mu_{\nu^*k})) \, \scirc\,
\beta(x,g,\mu^*h\cdot (\nu\mu)^*k)\right]
$$
$$
[k,d]([h,c] [g,b])=\left[(g\cdot \mu^*h)\cdot (\nu\mu)^*k, d \,
\scirc\, \Big(\big(c \, \scirc\, (b\cdot \mu_h) \, \scirc\,
\beta(x,g, \mu^*h)\big)\cdot (\nu\mu)_k\Big) \, \scirc\,
\beta(x,g\cdot \mu^*h, (\nu\mu)^*k) \right].
$$
In order to prove the equality between the two ways of composing the
three arrows, we consider the associativity isomorphism in $\cl{G}$
defined in ${\bf (g3)}$:
$$
\alpha:=\alpha(g, \mu^*h, (\nu\mu)^*k): g\cdot (\mu^*h\cdot
(\nu\mu)^*k) \lmap{\sim }      (g\cdot \mu^*h)\cdot (\nu\mu)^*k.
$$
Since $\pi_\G(\alpha)=\mr{id}_U$, it follows that
$\tar(\alpha)=\mr{id}_{\ma(x)}$.

Note that, by the functoriality of the multiplication $\mult$ on
$\G$, the following holds:
$$
(\overline{j}_2\circ \overline{j}_1)\cdot (j_2\circ j_1) =
(\overline{j}_2\cdot j_2)\circ (\overline{j}_1\cdot j_1),
$$
where $j_1,j_2,\overline{j}_1,\overline{j}_2$ are morphisms
$$
g_1\lmap{j_1} g_2 \lmap{j_2} g_3,\;\;\;\;\;
\overline{g}_1\lmap{\overline{j}_1} \overline{g}_2
\lmap{\overline{j}_2} \overline{g}_3
$$
for which the compositions and multiplications above make sense.
Considering the expression for $([k,d][h,c]) [g,b]$, we then have
$$
 \big(c \, \scirc\, (b\cdot \mu_h)
\, \scirc\,  \beta(x,g, \mu^*h)\big)\cdot (\nu\mu)_k=(c\cdot \nu_k)
\, \scirc\, ((b\cdot \mu_h)\cdot \mu_{\nu^*k}) \, \scirc\,
(\beta(x,g, \mu^*h)\cdot \mr{id}_{(\nu\mu)^*k}).
$$
Then proving associativity boils down to proving that $\alpha$ makes
the following diagram (which is \eqref{eq:triang} in our context)
commute:
$$
\xymatrix{ x\cdot\big(g\cdot(\mu^*h\cdot
(\nu\mu)^*k)\big)\ar[r]^-{\beta} \ar[d]_-{\mr{id} \cdot \alpha}&
(x\cdot g)\cdot(\mu^*h\cdot (\nu\mu)^*k)\ar[rr]^-{b\cdot(\mu_h\cdot
 \mu_{\nu^*k})}\ar[dd]_-{\beta}&&
y\cdot (h\cdot \nu^*k)\ar[dd]^-{\beta}&\\
x\cdot ( (g\cdot \mu^*h)\cdot (\nu\mu)^*k)\ar[d]_-{\beta}&&&&w\\
(x\cdot (g\cdot \mu^*h))\cdot (\nu\mu)^*k\ar[r]_-{\beta\cdot
\mr{id}}& ((x\cdot g)\cdot \mu^*h)\cdot (\nu\mu)^*k\ar[rr]_-{(b\cdot
\mu_h)\cdot \mu_{\nu^*k}}&&(y\cdot h)\cdot \nu^*k\ar[r]^-{c\cdot
\nu_k} &z\cdot k \ar[u]_-{d} }
$$
The commutativity of the left square is due to the higher coherence
$(xghl)$, whereas the commutativity of the right square follows from
the naturality of $\beta$. This finishes the proof of the
associativity of the composition of morphisms $\cl{X}\pq\cl{G}$. As
already mentioned, the identity axioms are proven similarly, and we
conclude that $\cl{X}\pq\cl{G}$ is a category.

The fact that $\pi_{\cl{X}\pq\cl{G}}$ in \eqref{eq:projpre} is a
functor follows from definitions \eqref{eq:comppre},\eqref{eq:idpre}
and \eqref{eq:func}, and from the fact that, by the definition of
natural transformation of fibred categories in groupoids,
$\pi_\X(\beta)$ and $\pi_\X(\varepsilon)$ are sent to identities in
$\cl{C}$.

Finally, we must prove that $\cl{X}\pq\cl{G}$ is fibred in groupoids
over $\cl{C}$. We outline the proof, leaving the details to the
reader.
%\textit{$\cl{X}\pq\cl{G}$ is fibred in groupoids over $\cl{C}$:}
Given $\mu: U\to V$ in $\cl{C}$ and $y\in\cl{X}$ with $\pi_\X(y)=V$,
we first have to produce a cartesian arrow $x\rightarrow y$ in
$\cl{X}\pq\cl{G}$ over $\mu$; this proves that $\cl{X}\pq\cl{G}$ is
fibred over $\cl{C}$. To do that we take an arrow $b:x \rightarrow
y$ in $\cl{X}$ over $\mu$ and define an arrow in $\cl{X}\pq\cl{G}$
over $\mu$ by
$$
[\un_{\ma(x)}, b\circ \varepsilon (x)]:x \map y.
$$
A direct computation shows that this arrow is cartesian; we remark
that the higher coherence $(xg1)$ in ${\bf (a4)}$ has to be used.

In order to prove that $\cl{X}\pq\cl{G}$ is fibred in groupoids, we
have to check that any morphism $[g,b]:x\rightarrow y$ in
$\cl{X}\pq\cl{G}$ over an identity of $\cl{C}$ is an isomorphism. We
claim that the inverse is given by
$$
[g,b]^{-1}= [g^{-1}, \varepsilon(x)\circ (\mr{id}_x\cdot
\i_r(g))\circ\beta (x,g,g^{-1})^{-1}\circ (b^{-1}\cdot
\mr{id}_{g^{-1}})].
$$
The proof of the claim is direct; we remark that in proving that
$[g,b][g,b]^{-1}=\mr{id}_y$ we use the higher coherences $(xghl)$,
$(x1g)$, $(xg1)$ of ${\bf (a4)}$, and $(gg^{-1}g)$ of ${\bf (g4)}$.

This concludes the proof of Prop.~\ref{PropPreqCfg}.
\\

We will need the next observation in
Prop.~\ref{prop1freeThenPrestack}.

\begin{remark}\label{rempullbackmorphisminXG}
Given an arrow $[g,b]$ in $\cl{X}\pq\cl{G}$ over an object $V$ of
$\cl{C}$ and an arrow $\mu:U\rightarrow V$ in $\cl{C}$ then
$\mu^*[g,b]=[\mu^*g, \mu^*b].$
\end{remark}

%%%%%%%%%%%%%%%%%%%%%%%%%%%%%%%%%%%%%%%%%%%%%%%%%%%%%%%%%%%%%%%%%%%%%%%%%%%%%%%%%%%%%%%%%
\subsection{Quotient in stages}\label{app:stages}

In this section we consider actions by products of cfg-groupoids.
The main result, used in  Prop.~\ref{PropMorEqRel}, asserts that in
this case quotients can be taken in stages, i.e., with respect to
one factor at a time.

\begin{lemma}\label{lem:equivarianceofstackyf} Let $\cl{Y}$ be a prestack
that carries a (right) action of a stacky groupoid $\G$ along
$\ma:\cl{Y}\ra M$. There is an induced $\G$-action on the
stackification $\cl{Y}^\sharp$ in such a way that the stackification
map $\stack:\cl{Y}\ra\cl{Y}^\sharp$ is $\G$-equivariant. Moreover,
if the $\G$-action on $\cl{Y}$ is 1-free then so is the action on
$\cl{Y}^\sharp$.
\end{lemma}

\begin{proof} First of all we observe that
$\stack\times \mr{id}: \cl{Y}\times_M\cl{G}\ra \cl{Y}^\sharp\times_M
\cl{G}$ is a stackification map. Indeed, $\cl{Y}\times_M\cl{G}$ is a
prestack, $\cl{Y}^\sharp\times_M\cl{G}$ is a stack, and one can
check that $\stack\times\mr{id}$ is a monomorphism and an
epimorphism, hence the observation follows from
Proposition~\ref{PropStackification} (iv). (Here we use that
$\cl{Y}$ is a prestack and not merely a cfg.)

One obtains all the data and higher coherences for the
$\cl{G}$-action on $\cl{Y}^\sharp$ from the universal property of
the stackification. We sketch the main steps leaving the details to
the reader. There is an induced moment map
$\ma^\sharp:\cl{Y}^\sharp\ra M$ such that $\ma^\sharp\stack=\ma$.
Applying the universal property of the stackification map
$\stack\times\mr{id}$ (part (i) of Prop.~\ref{PropStackification})
to the solid diagram:
$$\xymatrix{
\cl{Y}\times_M\cl{G}\ar[r]^-{\act}\ar[d]_-{\stack\times\mr{id}}&
\cl{Y}\ar[d]^-{\stack}\\
\cl{Y}^\sharp\times_M\cl{G}\ar@{-->}[r]_-{\act^\sharp}&
\cl{Y}^\sharp, }
$$
where $\act$ is the original action, one gets a pair
$(\act^\sharp,\delta)$, where
$\delta:\act^\sharp(\stack\times\mr{id})\ra \stack\act$. The
2-isomorphism $\delta$ is the one that makes $\stack$ equivariant.
The associativity 2-isomorphism
$\beta^\sharp:\act^\sharp(\mr{id}\times
\mult)\ra\act^\sharp(\act^\sharp\times\mr{id})$ of the new action is
induced by the $\beta:\act(\mr{id}\times
\mult)\ra\act(\act\times\mr{id})$ of the original action by applying
the universal property of the stackification map
$\stack\times\mr{id}\times \mr{id}$ (part (ii) of
Prop.~\ref{PropStackification}) to the diagram:
$$\xymatrix{
\cl{Y}\times_M\cl{G}\times_M\cl{G} \ar[rr]^-{\act(\mr{id}\times
\mult)}_-{\act(\act\times\mr{id})}
\ar[d]_-{\stack\times\mr{id}\times\mr{id}}&&
\cl{Y}\ar[d]^-{\stack}\\
\cl{Y}^\sharp\times_M\cl{G}\times_M\cl{G}
\ar[rr]^-{\act^\sharp(\mr{id}\times
\mult)}_-{\act^\sharp(\act^\sharp \times\mr{id})}&& \cl{Y}^\sharp. }
$$
The uniqueness part of (ii) of Prop.~\ref{PropStackification}
implies the equivariance higher coherence $(\delta\beta_1\beta_2)$,
see Def.~\ref{DefGequivariant}. The higher coherence $(xghl)$ for
the new action (the ``pentagon") follows from the same coherence for
the old action, the already proven $(\delta\beta_1\beta_2)$, and the
universal property of the stackification map
$\stack\times\mr{id}\times\mr{id}\times\mr{id}$ (the uniqueness part
of Prop.~\ref{PropStackification}, part (ii)). The same ideas can be
used to construct the identity 2-isomorphism $\vep^\sharp$ for the
new action, and to prove first the equivariance higher coherence
$(\delta\vep_1\vep_2)$, and then the action higher coherences
$(x1g)$ and $(xg1)$.

For the last assertion of the lemma, recall that the stackification
map $\stack:\cl{Y}\ra\cl{Y}^\sharp$ the stackification map is a
monomorphism and an epimorphism (see Prop.~\ref{PropStackification}
(iii)), and that it is $\G$-equivariant. Take $y\in\cl{Y}^\sharp$
and $j,j': g\ra \bar{g}$ in $\cl{G}$, where the objects are over the
same manifold $U$, and the morphisms are over $\mr{id}_U$. Assuming
that $\mr{id}_y\cdot j=\mr{id}_y\cdot j'$, we have to prove that
$j=j'$. There exist an open cover $(U_\alpha)$ of $U$, objects
$y_\alpha\in \cl{Y}_{U_\alpha}$, and isomorphisms
$\stack(y_\alpha)\ra y_{|U_\alpha}$. We have
$\stack(\mr{id}_{y_\alpha}\cdot j_{|U_\alpha})=
\mr{id}_{\stack(y_\alpha)}\cdot j_{|U_\alpha}=
\mr{id}_{\stack(y_\alpha)}\cdot j'_{|U_\alpha}=
\stack(\mr{id}_{y_\alpha}\cdot j'_{|U_\alpha})$, where the second
equality follows from $\mr{id}_{y_{|U_\alpha}}\cdot j_{|U_\alpha}=
\mr{id}_{y_{|U_\alpha}}\cdot j'_{|U_\alpha}$ using the isomorphisms
$\stack(y_\alpha)\ra y_{|U_\alpha}$. Hence, $\mr{id}_{y_\alpha}\cdot
j_{|U_\alpha}= \mr{id}_{y_\alpha}\cdot j'_{|U_\alpha}$ ($\stack$
being a monomorphism), and since the $\cl{G}$-action on $\cl{Y}$ is
1-free, it follows that $j_{|U_\alpha}=j'_{|U_\alpha}$. Since
$\cl{G}$ is a stack, we conclude that $j=j'$, as needed.
\end{proof}

Let $\cl{G}_i$ be a cfg-groupoid over $M_i$, $i=1,2$. Suppose that
$\cl{X}$ is a category fibred in groupoids carrying a (right) action
of the product $\cl{G}:=\cl{G}_1\times \cl{G}_2$ along the map
$\ma=(\ma_1,\ma_2): \cl{X}\to M_1\times M_2$.

\begin{lemma}\label{lem:inducedact2}
The $\cl{G}$-action on $\cl{X}$ restricts to an action of $\cl{G}_1$
on $\cl{X}$ (on the fibers of $\ma_2$), and there is an induced
action of $\cl{G}_2$ on $\cl{X}\pq \cl{G}_1$. (Clearly, the same
holds for $\cl{G}_1$ and $\cl{G}_2$ interchanged.)
\end{lemma}

\begin{proof} The action of $\cl{G}_1$ on $\cl{X}$ is defined in terms of the
action of $\cl{G}_1\times\cl{G}_2$ by
$$
x\cdot g_1=x\cdot (g_1, 1_{\ma_2(x)}),
$$
and this restriction process applies to all the data and higher
coherences of an action; moreover, since $\ma_2(x\cdot
g_1)=\sour(1_{\ma_2(x)})=\ma_2(x)$, the action is on the fibers of
$\ma_2:\X\ra M_2$.

Define $\cl{Y}:=\cl{X}\pq \cl{G}_1$, and let $\q_1:\cl{X}\ra\cl{Y}$
be the quotient map. By the universal property of the prequotient
(Prop.~\ref{LemmaXGtoS}) we have an induced map $\cl{Y}\ra M_2$.
Consider the map $\q_1\times \mr{id}_{\cl{G}_2}: \cl{X}\times_{M_2}
\cl{G}_2\ra \cl{Y}\times_{M_2}\cl{G}_2$. In the sequel we use Remark
\ref{InducedActionXY} to obtain an action of $\cl{G}_1$ on
$\cl{X}\times_{M_2}\cl{G}_2$, and to identify
$(\cl{X}\times_{M_2}\cl{G}_2)\pq \cl{G}_1 $ with
$\cl{Y}\times_{M_2}\cl{G}_2$.
 The action
map $\cl{X}\times_{M_2}\cl{G}_2\ra\cl{X}$ is $\cl{G}_1$-equivariant;
indeed, the equivariance 2-isomorphism $\delta$ is given by the
composition
\begin{eqnarray}\label{eqdeltag2g1}
\xymatrix{ (xg_2)g_1=[x\cdot(1,g_2)]\cdot(g_1,1)\ar[d]_-{\delta}&
x\cdot[(1,g_2)\cdot (g_1,1)]=x\cdot(1g_1,g_21)\ar[l]_-{\beta}
\ar[r]^-{\mr{id}(\lambda_1,\rho_2)}&
x\cdot(g_1,g_2)\ar[d]^-{\mr{id}}\\
(xg_1)g_2=[x\cdot(g_1,1)]\cdot(1,g_2)& x\cdot[(g_1,1)\cdot
(1,g_2)]=x\cdot(g_11,1g_2)\ar[l]^-{\beta}
\ar[r]_-{\mr{id}(\rho_1,\lambda_2)}&
x\cdot(g_1,g_2),\\
}
\end{eqnarray}
and we leave the verification of the equivariance higher coherences
to the reader. It follows from Prop.~\ref{PropfoverG}, part (i),
that there is an induced morphism
$\act_\cl{Y}:\cl{Y}\times_{M_2}\cl{G}_2\ra\cl{Y}$, together with a
2-isomorphism
$$
\varphi: \q(x)\cdot g_2\lmap{\sim} \q(x\cdot g_2),\qquad (x,g_2)\in
\cl{X}\times_{M_2}\cl{G}_2.
$$
The uniqueness property of the pair $(\act_\cl{Y},\varphi)$ (see
Prop. \ref{PropfoverG}, part (ii)) induces the extra data making
$\act_\cl{Y}$ into an action. We will illustrate how this works for
the associativity 2-isomorphism. Note that the techniques involved
are very similar to the ones in Lemma
\ref{lem:equivarianceofstackyf}, due to the formal similarity
between the universal property of the stackification (Prop.
\ref{PropStackification}) and the universal properties of the
quotients (Prop. \ref{LemmaXGtoS}) and of equivariant maps (Prop.
\ref{PropfoverG}). Before moving on, we recall that
$\Obj(\cl{X})=\Obj(\cl{X}\pq\cl{G}_1)$ and observe that the action
of $\G_2$ on $\X\pq \G_1$ agrees with the action on $\X$ on objects.

Consider the diagram
$$\xymatrix{
\cl{X}\times_{M_2}\cl{G}_2\times_{M_2}\cl{G}_2\ar[r]^-{F_1,F_2}\ar[d]&
\cl{X}\ar[d]\\
\cl{Y}\times_{M_2}\cl{G}_2\times_{M_2}\cl{G}_2\ar[r]_-{\Phi_1,\Phi_2}&
\cl{Y} }$$ where $F_1(x,g,\bar{g})=x(g\bar{g})$,
$F_2(x,g,\bar{g})=(xg)\bar{g}$, $\Phi_1(y,g,\bar{g})=y(g\bar{g})$,
$\Phi_2(y,g,\bar{g})=(yg)\bar{g}$. We have 2-isomorphisms (induced
by $\varphi$)
$$
\varphi_1:\q(x)\cdot (g\bar{g})\lmap{\sim} \q(x\cdot
(g\bar{g})),\qquad \varphi_2:(\q(x)g)\cdot \bar{g}\lmap{\sim}
\q((xg)\cdot\bar{g}).
$$
Similarly to the equivariance of the action map
$\cl{X}\times_{M_2}\cl{G}_2\ra\cl{X}$, one verifies that $F_1$ and
$F_2$ are $\cl{G}_1$-equivariant, and that the pairs
$(\Phi_i,\varphi_i)$ ($i=1,2$) satisfy condition
$(\varphi\gamma\delta)$ of Prop. \ref{PropfoverG} (i). Applying
Prop. \ref{PropfoverG} (ii) (the existence part), for
$\varsigma=\beta: F_1\ra F_2$ (the associativity 2-isomorphism of
the $\cl{G}_2$ action on $\cl{X}$), we get
$\beta_\cl{Y}:\Phi_1\ra\Phi_2$ (the associativity 2-isomorphism of
$\act_{\cl{Y}}$). %the $\cl{G}_2$ action on $\cl{Y}$).
The higher coherence $(xghl)$ (the pentagon) for $\beta_\cl{Y}$
follows from the one for $\beta$ by applying the uniqueness part of
Prop. \ref{PropfoverG} (ii) to the appropriate 2-isomorphisms that
relate the five obvious maps
$\cl{X}\times_{M_2}\cl{G}_2\times_{M_2}\cl{G}_2\times_{M_2}\cl{G}_2\ra\X$.

%The same argument gives a $\G_2$-action on $\cl{X}$.
\end{proof}

\begin{remark}\label{remexpindac}
The proof of Lemma~\ref{lem:inducedact2} hides the explicit form of
the action induced on the prequotient. At the level of objects, as
already mentioned, we have that $xg_2$ in $\cl{X}\pq\cl{G}_1$ is the
same as in $\cl{X}$. We now write how morphisms are multiplied. Let
$[g_1,b]:x\ra \bar{x}$ and $j_2:g_2\ra \bar{g}_2$ be morphisms in
$\cl{X}\pq\cl{G}_1$ and $\G_2$ respectively (above the same morphism
in $M_2$). In particular, $b:xg_1\ra \bar{x}$ is a morphism in $\X$.
Looking at the proof of Prop.~\ref{LemmaXGtoS} and at Remark
\ref{InducedActionXY}, one can verify that
$$
[g_1,b]\cdot j_2=[g_1, (b\cdot j_2)\circ \delta]
$$
where $\delta$ is defined in (\ref{eqdeltag2g1}).
\end{remark}

\begin{lemma}\label{lem:2stagisopre}
There is a canonical isomorphism $(\X\pq\G_1)\pq
\G_2\lmap{\sim}\X\pq \G$.
\end{lemma}

\begin{proof} The universal property of the
prequotient (Prop.~\ref{LemmaXGtoS}) induces the maps $\Phi_1$ and
$\Phi$ in the diagram
$$
\xymatrix{
\cl{X}\ar[d]\ar[rd]&\\
\X\pq\G_1\ar[r]_-{\Phi_1}\ar[d]&
\X\pq \G\\
(\X\pq\G_1)\pq\G_2\ar[ru]_-{\Phi}& }
$$
The canonical map $\Phi$ can be described explicitly and the
following argument shows that it is an isomorphism. Note that, at
the level of objects, all the arrows in the above diagram are
identities, so that we just have to check that $\Phi$ is fully
faithful. Let $x,y\in\Obj(\X)$. We consider triples $(g_1,g_2,b)$
where $g_i\in\G_i$, $\tar(g_i)=\ma_i(x)$, and $b:(x\cdot g_2)\cdot
g_1\ra y$ is a morphism in $\X$. We denote by $\xi_1: x\cdot
(g_1,g_2)\lmap{\sim} (x\cdot g_2)\cdot g_1$ the isomorphism of the
first line of \eqref{eqdeltag2g1}. Associated with such a triple, we
have a morphism
 $[g_2,[g_1,b]]:x\ra y$ in
$(\X\pq\G_1)\pq\G_2$, and a morphism $[(g_1,g_2), b\circ\xi_1]:x\ra
y$ in $\X\pq\G$. Given another triple
$(\bar{g}_1,\bar{g}_2,\bar{b})$, one can check that
$[g_2,[g_1,b]]=[\bar{g}_2,[\bar{g}_1,\bar{b}]]$ if and only if
$[(g_1,g_2), b\circ\xi_1]=[(\bar{g}_1,\bar{g}_2),
\bar{b}\circ\xi_1]$ (the condition for both equalities being the
existence of $j_1:g_1\ra\bar{g}_1$ and $j_2:g_2\ra\bar{g}_2$ making
the appropriate diagrams commute, compare with Eq.
(\ref{eq:triang})). It follows that the formula defining the action
of $\Phi$ on morphisms, given by
$$
[g_2,[g_1,b]]\quad\mapsto\quad [(g_1,g_2), b\circ\xi_1],
$$
is a bijection, hence $\Phi$ is fully faithful, as we wanted to
verify. (One can check that the previous formula indeed gives the
action of $\Phi$ by looking at the proof of Prop.~\ref{LemmaXGtoS}.)
\end{proof}

\begin{lemma}\label{lem:free2}
If the $\cl{G}$-action on $\cl{X}$ is 1-free (Def. \ref{def1free}),
then:
\begin{itemize}
\item[(a)] The induced
action of $\cl{G}_1$ on $\cl{X}$ is 1-free.
\item[(b)] The induced
action of $\cl{G}_2$ on $\cl{X}\pq \cl{G}_1$ is 1-free.
\end{itemize}
\end{lemma}

\begin{proof}
The action of $\G$ being 1-free means that $\mr{id}_x\cdot(j_1,j_2)=
\mr{id}_x\cdot(j'_1,j'_2)$ implies that $j_1=j'_1$ and $j_2=j'_2$.

In order to prove that the $\G_1$-action on $\X$ is 1-free, assume
that $\mr{id}_x\cdot j_1=\mr{id}_x \cdot j'_1$. Using the definition
of the induced action we have $\mr{id}\cdot(j_1,1_2)=\mr{id}_x\cdot
j_1=\mr{id}_x\cdot j'_1= \mr{id}\cdot(j'_1,1_2)$, and the hypothesis
of 1-freeness of the $\G$-action implies that $j_1=j'_1$, as needed.

We now prove that the $\G_2$-action on the prequotient
$\cl{X}\pq\cl{G}_1$ is 1-free. Assume that
$\widetilde{\mr{id}}_x\cdot j_2= \widetilde{\mr{id}}_x\cdot j'_2$,
where the tilde is used to remind us that
$\widetilde{\mr{id}}_x=[1_1, \vep_1] $ is an identity arrow in
$\cl{X}\pq\cl{G}_1$ rather than $\X$. Here $1_1$ and $\vep_1$ are
the appropriate identity (neutral element) and identity
2-isomorphism of $\G_1$. Moreover, $j_2,j'_2:g_2\ra\bar{g}_2$.

From Remark \ref{remexpindac} we see that $[1_1,(\vep_1\cdot
j_2)\circ\delta]=\widetilde{\mr{id}}_x\cdot j_2=
\widetilde{\mr{id}}_x\cdot j_2'= [1_1,(\vep_1\cdot
j_2')\circ\delta]$ as morphisms $xg_2\ra x\bar{g}_2$ in $\X\pq\G_1$.
Hence there exists an isomorphism $j_1:1_1\ra 1_1$ in $\G_1$ such
that the outer rectangle
$$
\xymatrix{
(xg_2)1_1\ar[r]^-{\xi_1}\ar[d]_-{(\mr{id}_{x}\mr{id}_{g_2}) j_1}&
x(1_1,g_2)\ar[r]^-{\xi_2}\ar[d]_-{\mr{id}_x(j_1,\mr{id}_{g_2})}&
(x1_1)g_2\ar[r]^-{\vep_1\mr{id}_{g_2}}& xg_2\ar[r]^-{\mr{id}_xj_2}&
x\bar{g}_2\ar[d]_-{\mr{id}}\\
(xg_2)1_1\ar[r]_-{\xi_1}& x(1_1,g_2)\ar[r]_-{\xi_2}&
(x1_1)g_2\ar[r]_-{\vep_1\mr{id}_{g_2}}& xg_2\ar[r]_-{\mr{id}_xj_2'}&
x\bar{g}_2 }$$
%$$\xymatrix{
%(xg_2)\cdot 1_1\ar[rrd]^-{(\mr{id}_xj_2)\circ\vep_1}
%\ar[d]_-{\mr{id}_{xg_2}\cdot j_1}&&\\
%(xg_2)\cdot 1_1\ar[rr]_-{(\mr{id}_xj_2)\circ\vep_1}&& x\bar{g}_2
%}$$
commutes, where $\xi_1,\xi_2$ are defined by the first and second
line of \eqref{eqdeltag2g1}, respectively (one has
$\xi_2\xi_1=\delta$). Since the left rectangle commutes ($\xi_1$ is
a natural transformation) then the right one commutes as well. One
can show that $(\vep_1\cdot\mr{id}_{g_2})\circ \xi_2$ is the
identity of $x(1_1,g_2)=xg_2$ (one uses higher coherence $(x1g)$ of
the original $\G$-action on $\X$, and $\rho_1=\lambda_1:1_1\cdot
1_1\ra 1_1$). As a result, the above right rectangle becomes the
(commutative) triangle
$$\xymatrix{
x(1_1,g_2)\ar[rrd]^-{\mr{id}_x(\mr{id}_{1_1},j_2)}
\ar[d]_-{\mr{id}_x(j_1,\mr{id}_{g_2})}&&\\
x(1_1,g_2)\ar[rr]_-{\mr{id}_x(\mr{id}_{1_1},j'_2)}&&
x(1_1,\bar{g}_2) }$$ from which it follows that
$\mr{id}_x(j_1,j'_2)=\mr{id}_x(\mr{id}_{1_1},j_2)$. By the
assumption of 1-freeness of the original $\G$-action we conclude
that $j_2=j'_2$ (and $j_1=\mr{id}_{1_1}$), as needed.
\end{proof}

\begin{proposition}\label{prop:quotstages}
Assume that  $\cl{X}$ is a prestack, $\G_1,\G_2$ are stacky
groupoids, and the $\G$-action on $\X$ is 1-free, where
$\G=\G_1\times \G_2$. Then there is a canonical isomorphism
$$
(\X/\G_1)/\G_2\lmap{\sim} \X/(\G_1\times \G_2).
$$
\end{proposition}

\begin{proof}
Consider the diagram
$$\xymatrix{
(\X/\G_1)\pq \G_2\ar[d]^-{\stack}& (\X\pq\G_1)\pq
\G_2\ar[r]^-{\Phi}\ar[l]_-{\Psi}\ar[d]^-{\stack}&
\X\pq(\G_1\times \G_2)\ar[d]^-{\stack}\\
(\X/\G_1)/ \G_2& (\X\pq\G_1)/
\G_2\ar[r]^-{\Phi^\sharp}\ar[l]_-{\Psi^\sharp}& \X/(\G_1\times \G_2)
}$$ where the vertical maps are stackifications. We first consider
the right square. From Proposition \ref{prop1freeThenPrestack} and
Lemma \ref{lem:free2}
 it follows that $\X\pq\G_1$,
$(\X\pq\G_1)\pq \G_2$ and $ \X\pq(\G_1\times \G_2)$ are prestacks.
The map $\Phi$ is the isomorphism of Lemma \ref{lem:2stagisopre}
(hence $\Phi^\sharp$ is an isomorphism as well). Let us look at the
map $\Psi$. The stackification map $\X\pq\G_1\ra \X/\G_1$
 is $\G_2$-equivariant by
Lemma \ref{lem:equivarianceofstackyf}; from general facts, see
Prop.~\ref{PropStackification} (iii), it is also a monomorphism and
an epimorphism. We define $\Psi$ to be the induced map between the
prequotients by $\G_2$ (Proposition \ref{PropfoverG}), and we use
Proposition \ref{prop:equivmonoepi} to deduce that $\Psi$ is a
monomorphism and an epimorphism. Next, it follows from Lemma
\ref{lem:equivarianceofstackyf} (and Prop.
\ref{prop1freeThenPrestack}) that $(\X/\G_1)\pq \G_2$ is a prestack;
hence, we can apply Prop.~\ref{PropStackification} (iv) to conclude
that $\Psi^\sharp$ is an isomorphism. Finally, the desired canonical
isomorphism is:
$$ \Phi^\sharp\circ (\Psi^\sharp)^{-1}:
(\X/\G_1)/\G_2\lmap{\sim} \X/(\G_1\times \G_2).$$

\end{proof}

%%%%%%%%%%%%%%%%%%%%%%%%%%%%%%%% References %%%%%%%%%%%%%%%%%%%%%%%%%%%%%%%%%%%%%%%%%%%%%%%%%%%%%%%%%%
\bibliographystyle{amsalpha}

\begin{footnotesize}

\end{footnotesize}

%\vspace{5mm}
%\today
\end{document}